\documentclass{article}
\PassOptionsToPackage{numbers,sort&compress}{natbib}
\usepackage[preprint]{neurips_2026}

\usepackage{amsmath,amsthm,amssymb,mathtools}
\usepackage{algorithm}
\usepackage{algpseudocode}
\usepackage{microtype}
\usepackage{hyperref}
\usepackage{enumitem}
\usepackage{graphicx}
\usepackage{xcolor}
\usepackage{caption}
\usepackage{subcaption}
\usepackage{booktabs}
\usepackage{array}
\usepackage{tikz}
\usetikzlibrary{arrows.meta,calc}
\usepackage{caption}

\DeclareMathOperator*{\argmin}{arg\,min}

\DeclareMathOperator{\BetaDist}{Beta}
\DeclareMathOperator{\range}{range}
\DeclareMathOperator{\diam}{diam}

\newtheorem{assumption}{Assumption}
\newtheorem{theorem}{Theorem}
\newtheorem{lemma}{Lemma}
\newtheorem{proposition}{Proposition}
\newtheorem{corollary}{Corollary}
\newtheorem{remark}{Remark}

\newcommand{\memo}[1]{}
\newcommand{\redTodo}[1]{}
\newcommand{\akiko}[1]{}
\newcommand{\pl}[1]{}

\newcommand{\rsfwincludegraphics}[2]{%
  \IfFileExists{#2}{\includegraphics[#1]{#2}}{%
    \fbox{\parbox[c][0.24\textheight][c]{0.9\linewidth}{\centering Missing figure asset: \texttt{\detokenize{#2}}}}%
  }%
}

\title{Random-Subspace Frank--Wolfe over \\ Strongly Convex Sets}
\author{
Pierre-Louis Poirion\\
Center for Advanced Intelligence Project\\
RIKEN\\
Tokyo, Japan
\And
Sebastian Pokutta\\
Zuse Institute Berlin and TU Berlin\\
Berlin, Germany
\And
Akiko Takeda\\
Graduate School of Information Science and Technology\\
The University of Tokyo\\
$\&$ Center for Advanced Intelligence Project\\
RIKEN\\
Tokyo, Japan
}

\begin{document}

\maketitle

\begin{abstract}
Frank--Wolfe methods avoid projections, but over curved feasible regions the
full-space linear minimization oracle (LMO) can itself become the computational
bottleneck. We introduce random-subspace Frank--Wolfe (RSFW), the first Frank--Wolfe framework, to our knowledge, that replaces the ambient LMO by exact LMOs
over random low-dimensional affine sections of a general feasible set, while
preserving feasibility in the original space. For smooth convex objectives over
compact strongly convex feasible sets, we prove a dimension-explicit
approximate-oracle inequality and derive the standard \(O(1/k)\) open-loop
rate, with high-probability and almost-sure counterparts. Under short steps and
a gradient lower bound, the same geometric control yields linear convergence,
and we extend the sublinear theory to finite-sum stochastic gradients. We also
show that random sections can improve the local curvature model controlling
short steps: for smooth objectives, the quadratic model along a sampled section
is governed by the compressed Hessian, yielding computable \(d\times d\)
curvature constants for quadratic objectives over balls and ellipsoids. These
results provide a geometric theory of oracle-side randomization in
projection-free optimization.
\end{abstract}
\section{Introduction}
We consider the constrained convex optimization problem
\begin{align}
\min_{x \in C} f(x), \quad C \subset \mathbb{R}^n.
\label{optprob}
\end{align}
Throughout, we focus on regimes where \(C\) is a strongly convex
feasible set with sufficient boundary regularity; precise assumptions are given in Section~\ref{sec:setup}.

\paragraph{Projection-free optimization and its bottleneck.}
Frank--Wolfe (FW) methods are attractive because they replace Euclidean
projections by linear minimization oracles (LMOs), thereby preserving structure
such as sparsity, atomicity, or low rank, while also providing a natural
dual-gap certificate \cite{FrankWolfe1956,Jaggi2013FW}. However, this
projection-free tradeoff is only favorable when the LMO is itself tractable.
The relative difficulty of projection and linear minimization depends strongly
on the feasible set \cite{CombettesPokutta2021Complexity}, and over curved or
dense non-polyhedral domains---such as ellipsoids, matrix norm balls, and
spectrahedra---an exact FW step may require solving a large linear system,
singular vector computation, or eigendecomposition. In such settings, the
full-space LMO becomes the dominant computational bottleneck.

\paragraph{Random-subspace Frank--Wolfe.}
To address this bottleneck, we introduce random-subspace Frank--Wolfe (RSFW), to our knowledge the first
Frank--Wolfe framework that replaces the ambient LMO by exact LMOs over random
low-dimensional affine sections of a general feasible set. At iteration $k$, RSFW samples a random $d$-dimensional subspace
$U_k \subset \mathbb{R}^n$ with $d \ll n$, and solves the linearized problem
exactly over the affine section
\[
C \cap (x_k + U_k).
\]
We refer to the LMO over this slice as a \emph{section oracle},
with randomness arising from the choice of $U_k$.
The iterate remains feasible in the ambient space, but the LMO is replaced by
an exact low-dimensional section oracle. Thus RSFW preserves the
projection-free structure of Frank--Wolfe while significantly reducing the
oracle dimension. This oracle-side viewpoint is distinct from objective-side sketching,
coordinate updates, and relaxed full-domain oracles, as detailed next.

\paragraph{Relation to prior work.}
RSFW is best viewed as an oracle-side sketching method: instead of sketching the
objective model or relaxing the full LMO, it sketches the feasible domain seen
by the LMO and solves the resulting section oracle exactly. This distinguishes
it from coordinate and block-coordinate methods, which reduce cost by updating
only part of the variable or by exploiting product structure in the feasible set
\cite{Nesterov2012CoordinateDescent,RichtarikTakac2014RBCD,LacosteJulienJaggiSchmidtPletscher2013BlockCoordinateFW}.
It also differs from sketching and random-subspace methods that compress
objective-side linear algebra, update directions, or stored iterate
representations
\cite{PilanciWainwright2016IterativeHessianSketch,LacottePilanciPavone2019AdaptiveRandomSubspaces,YurtseverUdellTroppCevher2017Sketchy},
and from inexact, lazy, weak-separation, or backtracking FW methods that relax
or avoid repeated calls to the full-domain oracle
\cite{BraunPokuttaZink2017Lazifying,BraunPokuttaTuWright2019Blended,PedregosaNegiarAskariJaggi2020Backtracking,ZhouSun2022ApproximateFW}.
A more detailed comparison, including random-embedding methods that solve the
original objective on random feasible subproblems, is given in
Appendix~\ref{app:related_work}.

\paragraph{Geometric challenge and insight.}
A central challenge in RSFW is that a random affine section may discard the
point that realizes the full Frank--Wolfe gap. In contrast to the full-space
oracle, which has access to the entire feasible region, a section oracle only
sees a low-dimensional slice of $C$, and may therefore miss the face or extreme
point that yields the maximal descent.
This issue is not merely technical. Near a sharp corner of a polytope, the
remaining region that carries FW improvement can be confined to a narrow cone or
face, which a generic random affine section through the current iterate may miss
entirely. The resulting section progress can be arbitrarily small; this failure
mode is illustrated in Appendix~\ref{app:polytope_failure}. This failure mode motivates our geometric assumptions on the
feasible set.

Our analysis shows that such degeneracy is avoided when \(C\) has both
sufficient boundary regularity and strong convexity. The boundary regularity
rules out sharp corners and provides uniform interior rolling balls, while
strong convexity rules out flat faces and supplies the outer quadratic support
needed to compare local ball gaps with the full FW gap on \(C\). Together,
these two mechanisms imply that random sections retain a dimension-explicit
fraction of the global descent geometry. Intuitively, the feasible region is
neither too sharp nor too flat near the boundary, so low-dimensional sections
still capture meaningful descent directions. This geometric mechanism underlies
our approximate-oracle guarantees for RSFW. 
\paragraph{Contributions.}
Our contributions are as follows:

(1) \textbf{Approximate oracle theory.}
A central difficulty of RSFW is that a random section may miss the point
achieving the full FW gap. We address this by establishing a
dimension-explicit approximate-oracle inequality: under suitable curvature of
the feasible set, random affine sections preserve a constant fraction of the
full FW progress in expectation. This yields the classical \(O(1/k)\) open-loop rate,
together with high-probability and almost-sure guarantees.

(2) \textbf{Linear convergence.}
We show that the oracle-side sketching mechanism remains effective for adaptive
short steps, where the analysis is more delicate since the descent depends not
only on the FW gap but also on the projected gradient length. By combining
high-probability control of the section-wise progress with a gradient lower
bound and strong convexity, we establish linear convergence of RSFW in
expectation, along with high-probability and almost-sure counterparts.

(3) \textbf{Stochastic extensions.}
We extend the sublinear theory to finite-sum stochastic gradients, where the
gradient is estimated by mini-batching while the section is sampled
independently. The key observation is
that the random subspace and the stochastic gradient can be sampled
independently, allowing the approximate-oracle argument to combine with
standard variance-reduction and growing-batch techniques
\cite{HazanKale2012ProjectionFreeOnline,HazanLuo2016VRPF,
LocatelloYurtseverFercoqCevher2019SFW,
NegiarDresdnerTsaiElGhaouiLocatelloFreundPedregosa2020SFW,
DresdnerVladareanRatschLocatelloCevherYurtsever2022OneSample}.

(4) \textbf{Curvature improvement via random sections.}
Beyond reducing dimensionality, we show that random sections can improve the
local curvature model governing short steps. For smooth objectives, the
relevant curvature along a sampled section is determined by a compressed
Hessian, rather than a conservative ambient bound. In structured settings such
as quadratic objectives over balls and ellipsoids, this yields explicit
\(d \times d\) curvature constants and can justify larger admissible step
sizes.

\paragraph{Notation.}
We write \(B(x,r)\) for the closed Euclidean ball centered at \(x\) with radius
\(r\), and \(\mathbb S^{n-1}:=\{x\in\mathbb R^n:\|x\|=1\}\). Unless otherwise
specified, \(\|\cdot\|\) denotes the Euclidean norm for vectors and the induced
operator norm for matrices. For a symmetric matrix \(A\), we denote its smallest
and largest eigenvalues by \(\lambda_{\min}(A)\) and \(\lambda_{\max}(A)\).

\section{Problem setup and algorithm}\label{sec:setup}

We consider \eqref{optprob} under the following assumptions.
\begin{assumption}[Feasible set $C$]
\label{ass:C_clean}
$C\subset \mathbb R^n$ is nonempty, compact, convex and its boundary $\partial C$ is $\mathcal{C}^2$. We assume $0\in C$, and $C$ is $\beta_C$-strongly convex as a set: there exists $\beta_C>0$ such that for all $x,y\in C$ and $\lambda\in[0,1]$,
\begin{equation}
B\!\left(\lambda x+(1-\lambda)y,\ \beta_C\,\lambda(1-\lambda)\,\|x-y\|^2\right)\subset C.
\label{eq:UC_clean}
\end{equation}
\end{assumption}

The exponent-\(2\) uniform convexity of \(C\) in \eqref{eq:UC_clean} prevents flat boundary
pieces and gives the outer quadratic support needed for the ball-gap
comparison. 
The \(\mathcal{C}^2\) boundary assumption is used only to obtain a
uniform interior rolling-ball radius; it can be weakened to \(\mathcal{C}^{1,1}\) (defined in Remark~\ref{rem:c11}) boundary, or replaced directly by the uniform rolling-ball condition
(see Remark~\ref{rem:c11} in the Appendix). Thus the set \(C\) is geometrically
sandwiched between inner and outer tangent Euclidean balls, analogously to how
a smooth strongly convex function is sandwiched between two quadratic models.

\begin{assumption}[Objective $f$]
\label{ass:f_clean}
$f:C\to\mathbb R$ is convex, continuously differentiable on $C$, and $\nabla f$ is $L$-Lipschitz on $C$: $\|\nabla f(x)-\nabla f(y)\|\le L\|x-y\|$ for all $x,y\in C$.
\end{assumption}

\paragraph{Boundary regularization.}
The explicit \(\mathcal C^{1,1}\) boundary assumption can be enforced by an arbitrarily
small outer regularization. Let \(C\subset\mathbb R^n\) be a compact
\(\beta_C\)-strongly convex convex set. For \(\varepsilon>0\), define the
parallel body $C_\varepsilon:=C+\varepsilon B_2.$
Then \(C_\varepsilon\) has \(\mathcal C^{1,1}\) boundary and is
\(\beta_{C,\varepsilon}\)-strongly convex with
$\beta_{C,\varepsilon}
    =
    \frac{\beta_C}{1+2\varepsilon\beta_C}.$
Thus the boundary regularity assumption can be obtained through an arbitrarily
small outer approximation, with
\(\beta_{C,\varepsilon}\uparrow\beta_C\) as \(\varepsilon\downarrow0\). See
Appendix~\ref{app:parallel_regularization}.

For \(x\in C\), let
$
s(x)\in\argmin_{y\in C}\langle\nabla f(x),y\rangle,\ \ \ 
\Delta(x):=\langle\nabla f(x),x-s(x)\rangle
=\max_{y\in C}\langle\nabla f(x),x-y\rangle .
$
The standard convexity inequality yields $f(x)-f^*\le \Delta(x)$, where $f^*=f(x^*)$ for any minimizer $x^*\in\argmin_C f$.

At iteration \(k\), draw a uniform \(d\)-dimensional subspace
\(U_k\subset\mathbb R^n\), represented by a matrix
\(P_k\in\mathbb R^{d\times n}\) with \(P_kP_k^\top=I_d\) and
\(U_k=\range(P_k^\top)\). Given the current iterate $x_k\in C$ and gradient $g_k:=\nabla f(x_k)$, the subspace linear minimization oracle returns
\begin{equation}
s_k\in \argmin\{\langle g_k,y\rangle:\ y\in C\cap(x_k+U_k)\},
\qquad d_k:=s_k-x_k\in U_k.
\label{eq:subLMO_clean}
\end{equation}

\begin{algorithm}[t]
\caption{Random-Subspace Frank--Wolfe (RSFW)}
\label{alg:RSFW}
\begin{algorithmic}[1]
\Require \(C\subset\mathbb R^n\), dimension \(d\), \(x_0\in C\), stepsize rule
\For{\(k=0,1,2,\ldots\)}
    \State Sample Haar--Stiefel \(P_k\), set \(U_k=\range(P_k^\top)\), and \(g_k=\nabla f(x_k)\).
    \State Compute \(s_k\in\argmin_{s\in C\cap(x_k+U_k)}\langle g_k,s\rangle\).
    \State Choose \(\alpha_k\in[0,1]\) by the prescribed stepsize rule and set \(x_{k+1}=x_k+\alpha_k(s_k-x_k)\).
\EndFor
\end{algorithmic}
\end{algorithm}
In Algorithm~\ref{alg:RSFW}, the only difference from standard Frank--Wolfe is that instead of minimizing
the linear model over all of \(C\), RSFW minimizes it over the random affine
section \(C\cap(x_k+U_k)\). This preserves feasibility because the returned
point still lies in \(C\), but replaces the ambient LMO by a \(d\)-dimensional
section oracle. The benefit is therefore oracle-side: when the full LMO over a
curved set is expensive, the section LMO can be substantially cheaper while
retaining a quantifiable fraction of the FW progress.
We consider two different stepsizea. In the open-loop analysis we use
\begin{equation}
\alpha_k=\frac{2}{\beta_0 k+2},
\label{eq:open_loop_stepsize_alg}
\end{equation}
where $\beta_0$ is the approximate-oracle constant specified in the corresponding theorem.
In the short-step analysis we use
\begin{equation}
\alpha_k
=
\min\!\left\{
1,\,
\frac{\langle g_k,x_k-s_k\rangle}
     {L_k\|s_k-x_k\|^2}
\right\},
\label{eq:short_step_alg}
\end{equation}
with the convention $\alpha_k=0$ if $s_k=x_k$. In the generic theory one may take
$L_k=L$, while in the quadratic section-wise analysis $L_k$ is replaced by the restricted
curvature constant on $U_k$.
Equivalently, writing $s_k=x_k+P_k^\top w_k$, the subspace LMO can be solved in coordinates as
\[
w_k\in\argmin_{w\in\mathbb R^d}
\{\langle P_k g_k,w\rangle:\ x_k+P_k^\top w\in C\}.
\]
Compared with standard Frank--Wolfe, the short-step rule is unchanged except
that it uses the section LMO solution \(s_k\). The curvature parameter \(L_k\)
only needs to upper bound the Lipschitz constant of \(f\) along the section
direction \(s_k-x_k\); hence it is never larger than a valid ambient Lipschitz
bound \(L\), and can be substantially smaller. The generic theory takes
\(L_k=L\), while Section~\ref{sec:quadratic_ellipsoid_gain} proves that random
sections can yield much sharper section-wise curvature bounds. The open-loop
rule replaces the classical \(\alpha_k=2/(k+2)\) by
\(\alpha_k=2/(\beta_0k+2)\).

\section{Random-subspace efficiency ratios}
\label{sec:efficiency_ratios}

This section isolates the random-subspace component of the proof. The argument
first studies a model problem in which \(C\) is a Euclidean ball and asks what
fraction of the FW gap remains visible after restricting the LMO to a random
affine section. The ratios \(\gamma\) and \(\Gamma\), defined below, quantify
this retained progress. The convergence proofs then transfer the ball estimate
to a general strongly convex feasible set by a sandwich comparison: the
rolling-ball condition supplies an interior tangent ball, while strong
convexity supplies the outer quadratic support that compares the ball gap with
the full FW gap on \(C\).

Let \(P\in\mathbb R^{d\times n}\) be Haar--Stiefel, \(PP^\top=I_d\). For
\(u\in S^{n-1}\) and \(v\in\mathbb R^n\) with \(\|v\|\le1\) and
\(\langle u,v\rangle<1\), define
\[
\gamma(P;u,v):=
\frac{\|Pu\|\sqrt{1-\|v\|^2+\|Pv\|^2}-\langle Pu,Pv\rangle}
{1-\langle u,v\rangle},
\qquad
\Gamma(P;u,v):=
\frac{\|Pu\|\,\|Pv\|-\langle Pu,Pv\rangle}
{1-\langle u,v\rangle}.
\]
Here \(\gamma\) is the exact ratio for an interior ball section, whereas
\(\Gamma\) is the simpler boundary-ball ratio used for uniform lower bounds (see
Corollary~\ref{cor:gammaIntGeGamma} in Appendix~\ref{app:preliminaries}). When \(\|v\|=1\), the two ratios coincide.

The random variable \(\Gamma(P;u,v)\) measures the fraction of FW progress on a
ball that is retained after observing the relevant directions only through a
random \(d\)-dimensional section. Its expectation is the quantity needed for the
open-loop approximate-oracle analysis, since it controls the average fraction
of full-space FW progress visible to the section oracle.

A key technical input is the following uniform expected lower bound $\mathbb{E}[\Gamma(P;u,v)]$. To this end, let \(c_{n,d}\) denote any uniform lower bound on
$\inf_{u,v\in S^{n-1}:\ \langle u,v\rangle <1}
\mathbb E[\Gamma(P;u,v)].$

\begin{center}
\begin{minipage}[c]{0.57\linewidth}
\begin{proposition}[Expected section efficiency on balls]
\label{prop:EGamma_identity}
For Haar--Stiefel \(P\) and any \(u,v\in S^{n-1}\) with
\(\rho=\langle u,v\rangle<1\), if \(d\ge3\), then
\begin{equation}
\frac dn-\frac{4(n-d)}{n(n-1)}
\le
\mathbb E[\Gamma(P;u,v)]
\le
\frac dn .
\label{eq:EGamma_lower_clean}
\end{equation}
\end{proposition}

\smallskip
\noindent
Thus \(d/n\) is the natural scale: on the comparison ball, a random
\(d\)-dimensional section retains, on average, a fraction of FW progress within
\(O(1/n)\) of \(d/n\).
\end{minipage}
\hfill
\begin{minipage}[c]{0.38\linewidth}
\centering
\resizebox{\linewidth}{!}{%
\begin{tikzpicture}[x=1cm,y=1cm,>=Latex,line cap=round,line join=round]

\def\a{3.55}
\def\b{2.20}
\def\t{42}

\pgfmathsetmacro{\rball}{\b*\b/\a}
\pgfmathsetmacro{\xk}{\a*cos(\t)}
\pgfmathsetmacro{\yk}{\b*sin(\t)}

\pgfmathsetmacro{\nnorm}{sqrt((cos(\t)/\a)^2 + (sin(\t)/\b)^2)}
\pgfmathsetmacro{\nx}{(cos(\t)/\a)/\nnorm}
\pgfmathsetmacro{\ny}{(sin(\t)/\b)/\nnorm}

\pgfmathsetmacro{\cx}{\xk - \rball*\nx}
\pgfmathsetmacro{\cy}{\yk - \rball*\ny}

\def\rot{80}
\pgfmathsetmacro{\dx}{cos(\rot)*\nx - sin(\rot)*\ny}
\pgfmathsetmacro{\dy}{sin(\rot)*\nx + cos(\rot)*\ny}

\pgfmathsetmacro{\sbx}{\cx + \rball*\dx}
\pgfmathsetmacro{\sby}{\cy + \rball*\dy}

\def\phi{18}
\pgfmathsetmacro{\ux}{cos(\phi)}
\pgfmathsetmacro{\uy}{sin(\phi)}

\pgfmathsetmacro{\tauu}{-2*\rball*(\nx*\ux + \ny*\uy)}
\pgfmathsetmacro{\sux}{\xk + \tauu*\ux}
\pgfmathsetmacro{\suy}{\yk + \tauu*\uy}

\tikzset{
  outer/.style={draw=gray!75!black, fill=gray!8, line width=0.95pt},
  ball/.style={draw=blue!75!black, fill=blue!6, line width=0.9pt},
  rad/.style={blue!75!black, dashed, line width=0.75pt},
  grad/.style={red!75!black, -Latex, line width=1.25pt},
  section/.style={green!45!black, line width=1.05pt},
  point/.style={circle, fill=black, inner sep=1.55pt},
  bpoint/.style={circle, fill=blue!75!black, inner sep=1.5pt},
  gpoint/.style={circle, fill=green!45!black, inner sep=1.55pt}
}

\draw[outer] (0,0) ellipse [x radius=\a, y radius=\b];
\node[font=\Large] at (-2.65,1.18) {$C$};
\draw[ball] (\cx,\cy) circle (\rball);
\node[blue!75!black, font=\large] at (0.1,1.35) {$B(c_k,r)$};

\node[point,label=above right:$x_k$] at (\xk,\yk) {};
\node[bpoint,label={[blue!75!black]below right:$c_k$}] at (\cx,\cy) {};
\draw[rad] (\cx,\cy) -- (\xk,\yk);
\node[blue!75!black] at ({0.55*(\cx+\xk)+0.14},{0.55*(\cy+\yk)+0.08}) {$r$};

\draw[grad] (\xk,\yk) -- ({\xk+1.30*\dx},{\yk+1.30*\dy});
\node[red!75!black, anchor=west]
at ({\xk+0.78*\dx+0.05},{\yk+0.78*\dy+0.18}) {$-\nabla f(x_k)$};

\node[bpoint,label={[blue!75!black] right:$s_B$}] at (\sbx,\sby) {};
\draw[rad] (\cx,\cy) -- (\sbx,\sby);

\draw[section]
({\xk-3.05*\ux},{\yk-3.05*\uy}) --
({\xk+1.55*\ux},{\yk+1.55*\uy});
\node[green!45!black, anchor=west]
at ({\xk-2.0*\ux+0.10},{\yk-2.0*\uy-0.0}) {$x_k+U$};

\node[gpoint,label={[green!45!black]left:$s_U$}] at (\sux,\suy) {};
\draw[green!45!black, line width=1.45pt] (\xk,\yk) -- (\sux,\suy);

\end{tikzpicture}%
}
\par\smallskip
{\footnotesize
\textbf{Figure.} Ball comparison underlying \(\Gamma(P;u,v)\). The point
\(s_B\) is the LMO solution on the ball \(B(c_k,r)\), while
\(s_U\) is the LMO solution after restricting to the affine section
\(B(c_k,r)\cap(x_k+U)\).
}
\end{minipage}
\end{center}

Proposition~\ref{prop:EGamma_identity} gives a closed uniform estimate for the
expected section efficiency and identifies \(d/n\) as the natural scale:
the average fraction of FW progress retained by a random \(d\)-dimensional section
on the comparison ball is within 
within \(O(1/n)\) of \(d/n\). This proposition is the bridge from random geometry to optimization: after the
ball-gap comparison, it gives a dimension-explicit fraction of the full FW gap
that remains visible to the random section oracle.
\begin{remark}
The main approximate-oracle constants use Proposition~\ref{prop:EGamma_identity},
namely \(c_{n,d}=d/n-4(n-d)/(n(n-1))\), when this quantity is positive.
Proposition~\ref{prop:Gamma_dn_improved} is kept only as a complementary
angle-uniform determinant-based bound with positive \((d-1)/(n-1)\) scale.
\end{remark}

\paragraph{Expected and high-probability section efficiencies.}
For the short-step analysis, the expected value of
\(\Gamma(P;u,v)\) alone is not sufficient. The one-step decrease contains the product of the section gap and the projected gradient length. After the geometric comparison argument, this leads to mixed quantities of the form
\[
    \mathbb E\!\left[\|Pu\|\,\Gamma(P;u,v)\right].
\]
A lower bound on \(\mathbb E[\Gamma(P;u,v)]\) does not by itself
control such a product, and the variables \(\|Pu\|\) and \(\Gamma(P;u,v)\) need not be positively correlated. We therefore complement the expectation identity with a high-probability lower bound for \(\Gamma(P;u,v)\). Combined, in Proposition~\ref{prop:pos_corr_general}, with a one-vector Johnson--Lindenstrauss lower bound  for \(\|Pu\|\), this yields the mixed-moment estimate needed in the short-step proof. The same
high-probability control is also used later to obtain fixed-confidence and
almost-sure guarantees for the open-loop scheme.
\begin{proposition}[Uniform high-probability lower bound for $\Gamma(P;u,v)$]
\label{prop:GammaHP}
Fix $\delta_0:=1-\cos(\tfrac{1}{10})$.
Let $P\in\mathbb R^{d\times n}$ be Haar on the Stiefel manifold, and assume $0<\varepsilon<\min\!\left\{\frac18,\frac{\delta_0}{8}\right\}.$ There exists constants \(c_\Gamma,C_\Gamma\) such that for every $u,v\in S^{n-1}$ such that $\langle u,v \rangle <1$,
\begin{equation}
\mathbb P\!\left(
\Gamma(P;u,v)\ge \frac{d}{2n}
\right)
\ge
1-C_{\Gamma}e^{-c_{\Gamma}\varepsilon^2d}.
\label{eq:Gammahp_unif}
\end{equation}
\end{proposition}

\section{Convergence guarantees}
\subsection{Open-loop convergence}
\label{sec:open_loop}

We next turn the efficiency estimates into convergence guarantees for the
open-loop stepsize rule. The only algorithmic input needed for the expectation
proof is an approximate-oracle inequality: in expectation, the random section
oracle must recover a fixed fraction of the full FW gap. This fraction is
obtained by comparing the FW gap on \(C\) with the gap on a suitable ball
contained in \(C\), and then applying the expected lower bound for
\(\Gamma(P;u,v)\). The comparison is the geometric core of the proof. The rolling-ball condition
supplies an inner tangent ball, while strong convexity supplies the outer
quadratic support comparing the local ball gap with the full FW gap on \(C\).
The ratio of these geometric scales yields the uniform constant
\(\beta_0^{\mathrm{unif}}\) from
Proposition~\ref{prop:approxOracle_global_clean}, which governs the convergence
rate.

The resulting global comparison factor is
\begin{equation}
\kappa_{\mathrm{unif}}
:=
\min\!\left\{\kappa_0,\frac{R_{\min}}{D}\right\},
\qquad
\kappa_0:=2\beta_C R_{\min},
\label{eq:kappa_unif_main}
\end{equation}
where $D:=\diam(C)$, $\beta_C$ is from Assumption~\ref{ass:C_clean}, and $R_{\min}$ is the minimum rolling-ball radius from Lemma~\ref{lem:tangentBalls_clean}. Geometrically, \(R_{\min}\) is the radius of the smallest interior rolling ball
guaranteed uniformly along the boundary of \(C\). This is made precise by the global ball--gap comparison,
Proposition~\ref{prop:BallGap_global_clean}.

For a fixed admissible \(\varepsilon\) as in Proposition~\ref{prop:GammaHP}, set
\begin{equation}\label{eq:hp_constants}
c_{\rm hp}:=\min\{c_\Gamma,c_{\rm JL}\},\ 
C_{\rm hp}:=C_\Gamma+C_{\rm JL},\ 
p_{\rm hp}:=\bigl[1-C_{\rm hp}e^{-c_{\rm hp}\varepsilon^2d}\bigr]_+,\ 
\beta_{\rm hp}:=\kappa_{\rm unif}\frac d{2n}.
\end{equation}
Here \(c_\Gamma,C_\Gamma\) are defined in the proof of Proposition~\ref{prop:GammaHP}, and
\(c_{\rm JL},C_{\rm JL}\) are the constants in the one-vector projection bound
used in the proof.
\begin{proposition}[Uniform approximate oracle in expectation]
\label{prop:approxOracle_global_clean}
Assume Assumptions~\ref{ass:C_clean}--\ref{ass:f_clean}, and let $\beta_0^{\mathrm{unif}}:=\kappa_{\mathrm{unif}}\,c_{n,d},$ where \(c_{n,d}\) is any
valid lower bound on \(\mathbb E[\Gamma(P;u,v)]\) uniform over
\(u,v\in S^{n-1}\) with \(\langle u,v\rangle<1\).
Then for every iterate $x_k\in C$ with $g_k:=\nabla f(x_k)\neq 0$,
\begin{equation}
\mathbb E\big[\langle g_k,d_k\rangle\mid x_k\big]
\le -\beta_0^{\mathrm{unif}}\,\Delta(x_k).
\label{eq:approxOracle_global_clean}
\end{equation}
\end{proposition}
This is where the two parts of the geometry enter the optimization analysis:
\(c_{n,d}\) is the average fraction of the ball FW gap retained after random
sectioning, while \(\kappa_{\rm unif}\) transfers this ball-gap estimate to the
original set \(C\) through the rolling-ball/strong-convexity sandwich. Hence the
section LMO is an approximate FW oracle in expectation, with quality
\(\beta_0^{\rm unif}\).

\begin{theorem}[Expected open-loop convergence]
\label{thm:convergence_clean}
Assume Assumptions~\ref{ass:C_clean}--\ref{ass:f_clean}. Run RSFW with
\(\alpha_k=2/(\beta_0 k+2)\), where
\(\beta_0:=\beta_0^{\mathrm{unif}}=\kappa_{\mathrm{unif}}\,c_{n,d}\). Then, for all \(k\ge0\),
\[
\mathbb E[f(x_k)-f^*]
\le
\frac{M}{k+2/\beta_0},
\qquad
M:=\max\left\{
\frac{2(f(x_0)-f^*)}{\beta_0},
\frac{2LD^2}{\beta_0^2}
\right\}.
\]
In particular, \(\mathbb E[f(x_k)-f^*]=O(1/k)\).
\end{theorem}
The expectation proof uses only a lower bound, $c_{n,d}$, on the average random-section
ratio \(\mathbb{E}[\Gamma(P;u,v)]\). For path-wise open-loop guarantees we instead use the high-probability bound Proposition~\ref{prop:GammaHP}. It gives good oracle events with uniformly positive conditional probability; a weighted martingale argument controls their cumulative deviation from the mean.

\begin{theorem}[Almost-sure $O(1/k)$ convergence for the open-loop stepsize]
\label{thm:hp_open_loop_as}
Assume Assumptions~\ref{ass:C_clean}--\ref{ass:f_clean}, under the previous notations, define 
$\bar\beta:=p_{\mathrm{hp}}\beta_{\mathrm{hp}},\ \ 
\tau:=\frac{2}{\bar\beta},$
and run RSFW with open-loop stepsize $
\alpha_k:=\frac{2}{\bar\beta k+2}
=\frac{2}{\bar\beta\left(k+\frac{2}{\bar\beta}\right)}.$
Then there exists an almost surely finite random constant $C_\infty$ such that
\begin{equation}
f(x_k)-f^*\le \frac{C_\infty}{k+\tau}
\qquad\text{for all }k\ge 0
\quad\text{almost surely}.
\label{eq:as_ok_rate}
\end{equation}
In particular, $f(x_k)-f^*=O(1/k)$ almost surely.
\end{theorem}

\begin{corollary}[Fixed-confidence deterministic $O(1/k)$ bound]
\label{cor:hp_open_loop_deterministic}
Under the assumptions of Theorem~\ref{thm:hp_open_loop_as}, for every $\eta\in(0,1)$ there exists an explicit deterministic constant $C_\eta$ such that
\begin{equation}
\mathbb P\!\left(
f(x_k)-f^*\le \frac{C_\eta}{k+\tau}
\ \text{for all }k\ge 0
\right)\ge 1-\eta.
\label{eq:deterministic_confidence_rate}
\end{equation}
\end{corollary}

\subsection{Short-step convergence}
\label{sec:short_step}

We now study the adaptive short-step rule. Unlike the open-loop recursion, the short-step descent depends on both the section FW gap and the visible gradient length. This is why the expected value of \(\Gamma(P;u,v)\) alone is not enough: the proof requires control of mixed quantities such as \(\|Pu\|\Gamma(P;u,v)\). We obtain this control by combining the
high-probability lower bound for \(\Gamma\) with a one-vector projection bound.

\begin{assumption}[Gradient lower bound]
\label{ass:grad_lower}
Assume  that $
c_0:=\inf_{x\in C}\|\nabla f(x)\|>0.$
\end{assumption}

\begin{theorem}[Expected linear convergence under short steps]
\label{thm:linear_convergence}
Assume Assumptions~\ref{ass:C_clean}--\ref{ass:f_clean} and
Assumption~\ref{ass:grad_lower}. Run RSFW with the short-step rule
\eqref{eq:short_step_alg} and \(L_k=L\). With the constants in
\eqref{eq:hp_constants}, assume \(p_{\rm hp}>0\). Then, for all \(k\ge0\),
\begin{equation}
\mathbb E[f(x_k)-f^*]
\le
(1-r_{\rm lin})^k\bigl(f(x_0)-f^*\bigr),
\qquad
r_{\rm lin}
:=
p_{\rm hp}\beta_{\rm hp}
\min\!\left\{
\frac12,\,
\frac{\beta_C c_0}{8L}
\sqrt{(1-\varepsilon)\frac dn}
\right\}.
\label{eq:linear_rate}
\end{equation}
\end{theorem}
\begin{remark}[Scope of the uniformly convex extension]
\label{rem:q_uniform_scope}
The results above use strong convexity of the feasible set in the Euclidean metric.
They should not be read as an optimal short-step theory for general
\((\beta_C,q)\)-uniformly convex sets with \(q>2\), such as \(\ell_p\)-balls
with \(p>2\). Extending the \(q>2\) rate would require a replacement for the
Euclidean tangent-ball comparison and for the \(\Gamma\)-coefficient estimates in
non-Euclidean section geometries, where the rotational invariance used in
Proposition~\ref{prop:EGamma_identity}, or Proposition~\ref{prop:Gamma_dn_improved} is no longer available. The present
paper therefore keeps the \(q>2\) extension as future work rather than claiming
the weaker, non-optimal constants that one could obtain from crude Euclidean
inner/outer ball bounds when such bounds are available.

The Euclidean outer-ball workaround also does not cover the motivating
\(\ell_p\) examples. Indeed, let \(C=B_p\subset\mathbb R^n\) with \(p>2\), and 
look at \(x=e_1\) (here $e_1$ denotes the first vector of the standard basis of $\mathbb{R}^n$), whose inward normal is \(\mu=-e_1\). If a Euclidean
comparison ball \(B(x+R\mu,R)\) contained \(B_p\), then for
$
y_t=\bigl((1-t^p)^{1/p},t,0,\ldots,0\bigr)\in\partial B_p
$
one would need
$
t^2+a_t^2\le 2Ra_t,
\qquad
a_t:=1-(1-t^p)^{1/p}.
$
Since \(a_t\sim t^p/p\) as \(t\downarrow0\), this forces
$
R\ge \frac{t^2+a_t^2}{2a_t}
\sim \frac p2\,t^{2-p}\to\infty.
$
Thus no finite Euclidean \(R_{\max}\) exists at this boundary point. The
missing ingredient is therefore genuinely non-Euclidean: one needs a
section-efficiency coefficient adapted to the \(p\)-power boundary geometry,
not only a Euclidean tangent-ball sandwich.
\end{remark}

For the short-step rule, Proposition~\ref{prop:GammaHP} yields a one-step good event with probability bounded below uniformly in the iterate. A supermartingale argument then gives the following almost-sure and fixed-confidence bounds.

\begin{theorem}[High-probability geometric convergence under short steps]
\label{thm:hp_linear}
Under the conditions of Theorem~\ref{thm:linear_convergence}, for any $\theta\in(0,r_{\mathrm{lin}})$ and any $\xi \in(0,1)$, with probability at least $1-\xi$,
\begin{equation}
f(x_k)-f^*\le \frac{h_0}{\xi}(1-\theta)^k
\qquad\text{for all }k\ge 0\text{ simultaneously.}
\label{eq:hp_linear}
\end{equation}
\end{theorem}

\begin{corollary}[Almost-sure geometric convergence under short steps]
\label{cor:as_linear}
Under the conditions of Theorem~\ref{thm:linear_convergence}, for every $\theta\in(0,r_{\mathrm{lin}})$,
\begin{equation}
\frac{f(x_k)-f^*}{(1-\theta)^k}\xrightarrow[k\to\infty]{}0
\qquad\text{almost surely.}
\label{eq:as_linear_rate}
\end{equation}
\end{corollary}

\subsection{Finite-sum stochastic extension}\label{sec:stoch}

The approximate-oracle analysis extends to finite sums
\(f=N^{-1}\sum_{i=1}^N f_i\), where each \(f_i\) is convex and \(L\)-smooth.
At iteration \(k\), form the mini-batch estimator
\(\widehat g_k=b_k^{-1}\sum_{j=1}^{b_k}\nabla f_{i_{k,j}}(x_k)\), draw
\(U_k=\range(P_k^\top)\) independently of the mini-batch, and compute the
section LMO with \(\widehat g_k\):
\[
s_k\in\argmin_{s\in C\cap(x_k+U_k)}\langle \widehat g_k,s\rangle,\qquad
x_{k+1}=x_k+\alpha_k(s_k-x_k).
\]
The only observation needed is independence: conditional on
\((x_k,\widehat g_k)\), the random section is still Haar--Stiefel, so the
deterministic section-efficiency bound applies to \(\widehat g_k\) exactly as it
does to the true gradient.
\begin{theorem}[Finite-sum stochastic RSFW]
\label{thm:main_stoch_rsfw}
Assume the deterministic RSFW approximate-oracle inequality holds with constant
\(0<\beta_0\le1\), and let
$
\sigma_{\rm fs}^2:=
\sup_{x\in C}\frac1N\sum_{i=1}^N
\|\nabla f_i(x)-\nabla f(x)\|^2 .
$
Choose \(\alpha_k=2/(\beta_0 k+2)\) and
\(b_k=\lceil (k+2/\beta_0)^2/A_{\rm mb}\rceil\) for fixed
\(A_{\rm mb}>0\). Then there is a constant \(M_{\rm fs}\), depending only on
\(f(x_0)-f^*\), \(L\), \(\sigma_{\rm fs}\), \(\diam(C)\), \(A_{\rm mb}\), and
\(\beta_0\), such that, for all \(k\ge0\),
\[
\mathbb E[f(x_k)-f^*]\le\frac{M_{\rm fs}}{k+2/\beta_0}.
\]
\end{theorem}
The proof, given in the appendix, combines the deterministic approximate-oracle
estimate with the standard mini-batch variance bound and the usual scalar
induction from growing-batch stochastic Frank--Wolfe analyses
\cite{HazanLuo2016VRPF,NegiarDresdnerTsaiElGhaouiLocatelloFreundPedregosa2020SFW,DresdnerVladareanRatschLocatelloCevherYurtsever2022OneSample}.

\section{Random sections improve local quadratic models}
\label{sec:quadratic_ellipsoid_gain}

The preceding results show that random sections preserve enough Frank--Wolfe
progress to guarantee convergence. We now highlight a second effect: for
short-step rules, random sections can also improve the local curvature model
that determines the admissible step size.

This effect is relevant beyond exact quadratics. Suppose \(f\) is convex, \(\mathcal{C}^2\)
near \(C\), and has \(M_H\)-Lipschitz Hessian. Conditionally on the past,
\(Q_k:=\nabla^2 f(x_k)\) is deterministic, and for any
\(d_k\in\range(P_k^\top)\), 
$
f(x_k+t d_k)
\le
f(x_k)
+t\langle g_k,d_k\rangle
+\frac{t^2}{2}d_k^\top Q_kd_k
+\frac{M_H}{6}t^3\|d_k\|^3,\ t\in[0,1].
$
Thus, on a local regime where
$
\frac{M_H}{L_{P_k}}\alpha_k\|d_k\|\to0,\ 
L_{P_k}:=\lambda_{\max}(P_kQ_kP_k^\top),
$
the cubic remainder is negligible relative to the quadratic short-step decrease.
The exact quadratic case below isolates the compressed-Hessian mechanism seen
by RSFW.

Consider$
f(x)=\frac12x^\top Qx+\langle r,x\rangle+c,\  Q\succeq0,$
over the ellipsoid $C=\{x\in\mathbb R^n:x^\top Mx\le1\},\ M\succ0.$ For a sampled section, define
\[
A_k:=P_kMP_k^\top,
\qquad
b_k:=P_kMx_k,
\qquad
\delta_k:=1-x_k^\top Mx_k+b_k^\top A_k^{-1}b_k,
\qquad
L_{P_k}:=\lambda_{\max}(P_kQP_k^\top).
\]
When \(\delta_k>0\), define the section geometry constant
\begin{equation}
\beta_{\mathrm{sec},k}
:=
\frac{\lambda_{\min}(A_k)}
{2\sqrt{\delta_k\,\lambda_{\max}(A_k)}}.
\label{eq:beta_sec_k_main}
\end{equation}
This coefficient is the geometric curvature scale of the ellipsoidal section:
the section gap satisfies
$
\frac{\mathrm{gap}_k}{\|d_k\|^2}
\ge
\beta_{\mathrm{sec},k}\|P_kg_k\|.
$
Equivalently, if
$
R_{\min,k}:=\sqrt{\frac{\delta_k}{\lambda_{\max}(A_k)}},
\ 
R_{\max,k}:=\sqrt{\frac{\delta_k}{\lambda_{\min}(A_k)}},
$
are the smallest and largest section semi-axis bounds, then
$
\beta_{\mathrm{sec},k}
=
\frac{R_{\min,k}}{2R_{\max,k}^2}.$

\begin{proposition}[Quadratic short-step model on a random ellipsoidal section]
\label{prop:random_ellipsoid_short_step_improvement}
Let \(P_k\in\mathbb R^{d\times n}\) be an independent Haar--Stiefel matrix.
Assume \(\delta_k>0\), \(P_kg_k\neq0\), and that the quadratic short-step is on
the genuine short-step branch \(\alpha_k^{\rm quad}\in(0,1)\). Then
\begin{equation}
h_{k+1}
\le
h_k
-
\frac{\beta_{\mathrm{sec},k}}{2L_{P_k}}
\,\|P_kg_k\|\,\mathrm{gap}_k,
\label{eq:random_ellipsoid_short_step_descent_beta_sec}
\end{equation}
where \(h_k:=f(x_k)-f^*\),
\(\mathrm{gap}_k:=\langle g_k,x_k-s_k\rangle\), and \(s_k\) is the exact LMO
over \(C\cap(x_k+\operatorname{range}(P_k^\top))\).

Moreover, let
$
\bar\mu:=\frac{\operatorname{tr}(M)}{n},
\ 
\bar\lambda:=\frac{\operatorname{tr}(Q)}{n},
$
assume $\bar \lambda>0$, and define
$
r_M:=\frac{\operatorname{err}_M(d,\eta/2)}{\bar\mu},\ 
r_Q:=\frac{\operatorname{err}_Q(d,\eta/2)}{\bar\lambda},
$
where the Haar-compression error \(\operatorname{err}_H\) is given in
Proposition~\ref{prop:haar_spectral_sandwich}. If \(r_M<1\), then with
probability at least \(1-\eta\),
\begin{equation}
\beta_{\mathrm{sec},k}
\ge
\frac{\sqrt{\bar\mu}}{2\sqrt{\delta_k}}
\frac{1-r_M}{\sqrt{1+r_M}},
\qquad
L_{P_k}\le \bar\lambda(1+r_Q).
\label{eq:beta_sec_and_LP_random_bounds}
\end{equation}
Equivalently, the section curvature coefficient in
\eqref{eq:random_ellipsoid_short_step_descent_beta_sec} improves over the
conservative ambient coefficient by the factor
$
\frac{\beta_{\mathrm{sec},k}/(2L_{P_k})}
     {\beta_{\mathrm{amb}}/(2L_{\mathrm{amb}})}
\ge
\mathcal I_k(\eta),$
where
$
L_{\mathrm{amb}}:=\lambda_{\max}(Q),
\ 
\beta_{\mathrm{amb}}
:=
\frac{\lambda_{\min}(M)}{2\sqrt{\lambda_{\max}(M)}},
$
and
\begin{equation}
\mathcal I_k(\eta)
:=
\frac{\lambda_{\max}(Q)}{\bar\lambda}
\frac{\sqrt{\bar\mu\,\lambda_{\max}(M)}}
     {\lambda_{\min}(M)\sqrt{\delta_k}}
\frac{1-r_M}{(1+r_Q)\sqrt{1+r_M}}.
\label{eq:Ik_eta}
\end{equation}
\end{proposition}
The comparison should be read at the level of the curvature coefficient
\(\beta_{\mathrm{sec},k}/(2L_{P_k})\). The section bound is governed by the
average spectral levels \(\bar\lambda\) and \(\bar\mu\), up to the section-size
factor \(\delta_k\) and explicit Haar-compression errors. Relative to the
ambient conservative coefficient, the gain can be large when the objective
curvature or ellipsoid geometry is anisotropic.

\section{Numerical illustration}
\label{sec:numerics}

The experiments illustrate the oracle mechanisms predicted by the analysis, not
an exhaustive benchmark. All kernel experiments use HIGGS subsets~\cite{BaldiSadowskiWhiteson2014HIGGS,Whiteson2014HIGGSUCI}
with binary labels \(y_i\in\{-1,+1\}\). Given inputs \(x_i\), we form a frozen
random-feature matrix \(\Phi\in\mathbb R^{n\times m}\) and use the matrix-free
kernel \(K=\Phi\Phi^\top+\rho_K I\). The optimization variable is the kernel
coefficient vector \(a\in\mathbb R^n\), with score \((Ka)_i\), and the objective is
\(f(a)=n^{-1}\sum_i\log(1+\exp(-y_i(Ka)_i))+(\lambda/2)a^\top Ka\), over an
ellipsoid \(a^\top Ha\le1\). The non-HIGGS experiment is graph semi-supervised learning on
Covertype~\cite{Blackard1998Covertype}. We write
\(\mathcal I_{\rm lab}\) for the labeled nodes and use
\(f(u)=(2|\mathcal I_{\rm lab}|)^{-1}
\|u_{\mathcal I_{\rm lab}}-y_{\mathcal I_{\rm lab}}\|^2\), with
non-ellipsoidal constraint
\(C_G=\{u:(\mu/2)\|u\|^2+(\gamma_G/2)u^\top L_Gu+(\beta_4/4)\|u\|_4^4\le\tau\}\).

For ellipsoids, full FW uses the full LMO
\(s(g)=-H^{-1}g/\sqrt{g^\top H^{-1}g}\), while RSFW solves the same problem in
a section \(a+Sz\), using \(A=S^\top HS\), \(b=S^\top Ha\), and \(a^\top Ha\).
For the graph experiment, the full LMO has no closed form and requires solving a nonlinear optimality
system involving the graph Laplacian,
whereas RSFW compresses the graph part to
\(S^\top(\mu I+\gamma_G L_G)S\) and \(S^\top(\mu I+\gamma_G L_G)u\). The graph
short step uses the exact directional denominator
\(|\mathcal I_{\rm lab}|^{-1}\|(s-u)_{\mathcal I_{\rm lab}}\|^2\). Full FW is run once; RSFW is
averaged over independent subspace draws with one-standard-deviation bands.

Figure~\ref{fig:numerics_all} shows four regimes. In the fixed-\(L\) grid, each
method uses the smallest feasible short-step constant from the same grid, where
feasible means that the resulting short steps do not increase the objective
over the run; the selected values are \(L_{\rm full}=10^2\) and
\(L_{\rm RSFW}=1\). The section-curvature diagnostic computes curvature from
\(\lambda_{\max}((KS)^\top(KS)/(4n)+\lambda S^\top KS)\). The matrix-free
ellipsoid experiment uses \(H=K+\rho_HI\) and gives full FW a strong
feature-direct oracle, so the comparison does not rely on CG failure. The graph
experiment shows that, even on a moderate-size non-ellipsoidal problem, RSFW can
beat full FW in wall-clock time by replacing expensive ambient LMOs with cheaper
section oracles. Further details are in Appendix~\ref{app:numerics_details}.

\begin{figure}[t]
\centering
\captionsetup[subfigure]{font=footnotesize,skip=3pt}
\captionsetup{font=footnotesize,skip=4pt}
\setlength{\abovecaptionskip}{3pt}
\setlength{\belowcaptionskip}{-2pt}

\begin{subfigure}[t]{0.485\linewidth}
    \centering
    \rsfwincludegraphics{width=\linewidth}{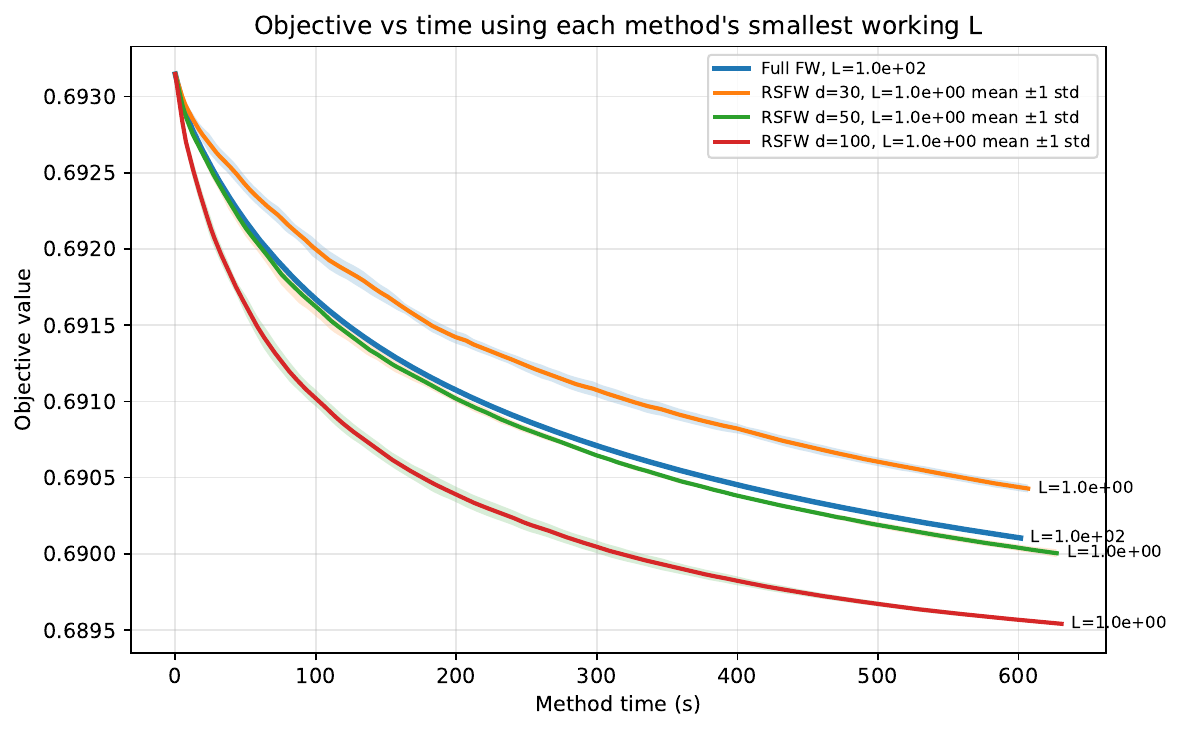}
    \caption{Fixed-\(L\) grid.}
    \label{fig:numerics_grid}
\end{subfigure}
\hfill
\begin{subfigure}[t]{0.485\linewidth}
    \centering
    \rsfwincludegraphics{width=\linewidth}{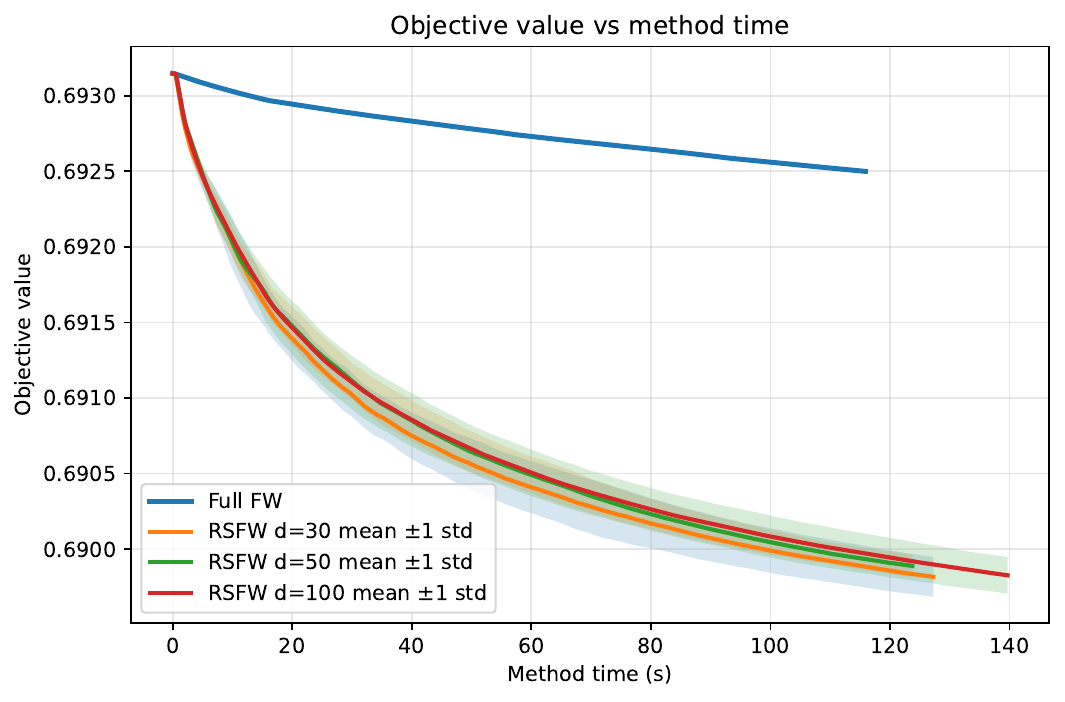}
    \caption{Section curvature.}
    \label{fig:numerics_diag}
\end{subfigure}

\vspace{0.35em}

\begin{subfigure}[t]{0.485\linewidth}
    \centering
    \rsfwincludegraphics{width=\linewidth}{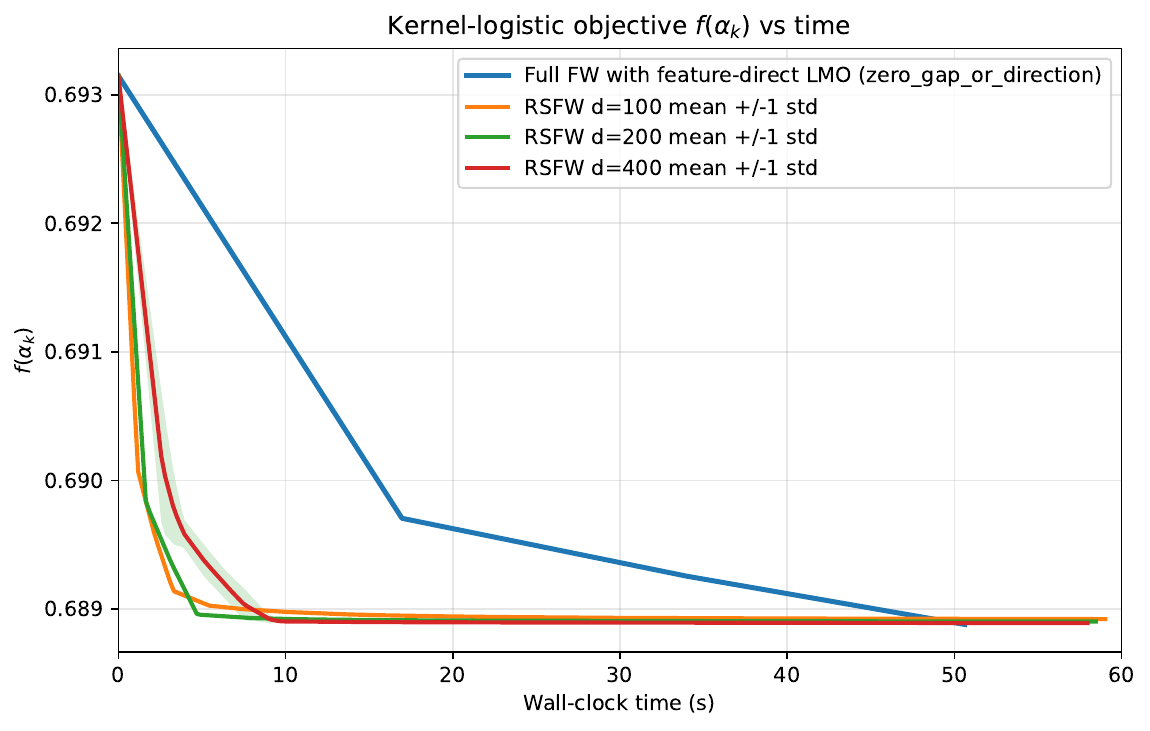}
    \caption{Matrix-free ellipsoid.}
    \label{fig:numerics_matvec}
\end{subfigure}
\hfill
\begin{subfigure}[t]{0.485\linewidth}
    \centering
    \rsfwincludegraphics{width=\linewidth}{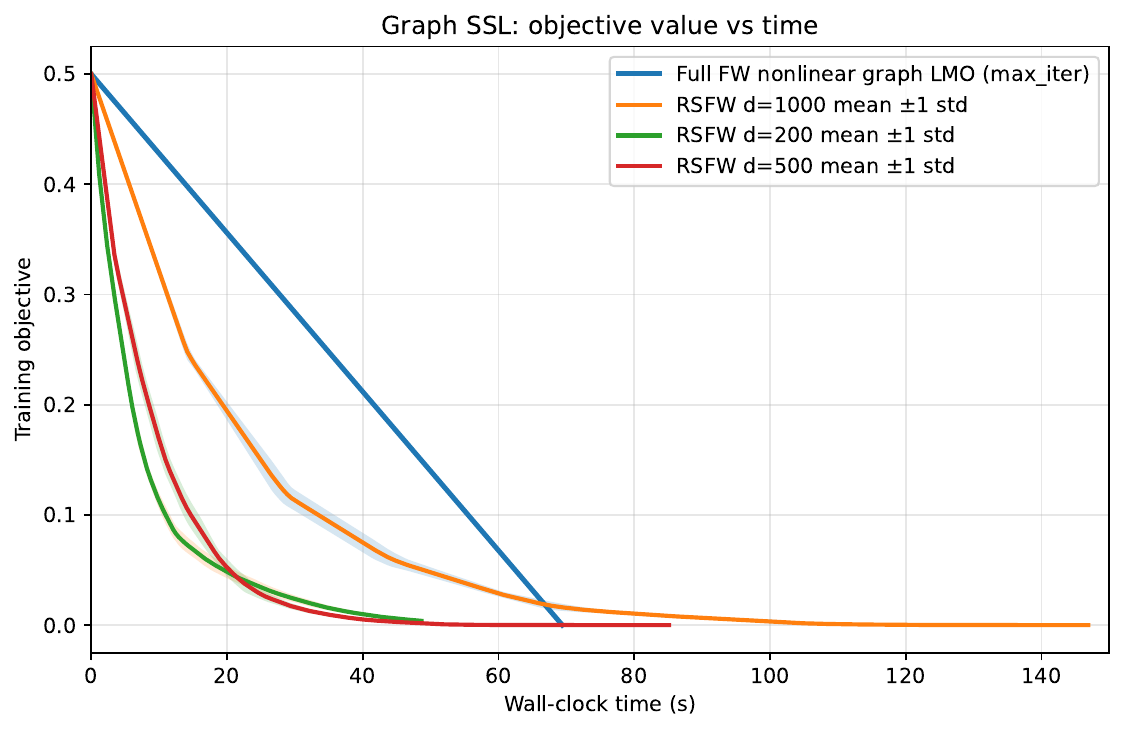}
    \caption{Graph SSL.}
    \label{fig:numerics_graph}
\end{subfigure}

\caption{Numerical illustrations of RSFW. The panels show: smallest feasible
short-step constant, section-wise curvature, matrix-free ellipsoid oracles, and
a non-ellipsoidal graph constraint.}
\label{fig:numerics_all}
\end{figure}
\section{Conclusion and limitations}
We developed a geometric convergence theory for random-subspace
Frank--Wolfe methods. The main message is that, on sufficiently curved feasible
sets, exact LMOs over random affine sections retain a dimension-explicit
fraction of the full Frank--Wolfe descent geometry. This yields
approximate-oracle guarantees, open-loop \(O(1/k)\) convergence, path-wise
variants, and linear convergence under short steps, a gradient lower bound, and
strong convexity. The same geometric viewpoint also explains why, for quadratic
models over balls and ellipsoids, short-step rules can benefit from compressed
\(d\times d\) curvature quantities rather than worst-case ambient constants.
The experiments illustrate these mechanisms on kernel-logistic and graph-based
examples, and show that RSFW can outperform full FW in wall-clock time when
ambient LMOs are substantially more expensive than section LMOs.
\paragraph{Limitations.}
The guarantees rely on genuine geometric curvature and boundary regularity; in
particular they do not cover polytopes or sharp-cornered domains. Outer
parallel regularization can turn a compact strongly convex set into a
\(\mathcal C^{1,1}\) body to which the theory applies, but this changes the feasible
region and weakens the constants, so it should be viewed as a controlled
regularized variant rather than a free removal of the assumption.

Future work includes extending the section-efficiency analysis beyond strong convexity, developing structured or data-adaptive random sections,
and designing practical inexact-oracle variants with provable guarantees.

\bibliographystyle{plain}
\bibliography{refs}

\appendix


\newpage 
\paragraph{Appendix roadmap.}
Appendix~\ref{app:related_work} discusses additional related work. Appendix~\ref{app:preliminaries} collects the deterministic and probabilistic
building blocks: tangent balls, ball-section formulas, and the two-dimensional
Haar--Stiefel Gram decomposition. Appendix~\ref{app:expected-gamma} proves the
expectation identities and lower bounds for the section-efficiency ratio
\(\Gamma\). Appendix~\ref{app:expected-to-descent} turns these ratio bounds into
approximate-oracle and convergence guarantees by comparing \(C\) with tangent
balls. Appendix~\ref{app:hp-linear} gives the high-probability and almost-sure
variants. Appendix~\ref{app:random_haar_quadratic} proves the Haar spectral
compression estimate and the quadratic ellipsoidal short-step result.
\section{Additional related work}
\label{app:related_work}

\paragraph{Geometry-dependent Frank--Wolfe theory.}
Classical Frank--Wolfe methods replace projections by linear minimization
oracles and provide natural primal-dual gap certificates
\cite{FrankWolfe1956,Jaggi2013FW}. Modern surveys emphasize the breadth of this
projection-free viewpoint and the role of oracle models in FW-type methods
\cite{BomzeRinaldiZeffiro2024FWFriends}. A large body of work shows that FW
rates depend strongly on the geometry of the feasible region. On polytopes,
away-step and pairwise variants can achieve global linear convergence
\cite{LacosteJulienJaggi2015FWLinear}. On curved domains, vanilla FW can be
faster than the worst-case \(O(1/k)\) rate; strongly convex sets yield improved
rates \cite{GarberHazan2015FW}, and uniformly convex sets provide a broader
geometric framework in the existing FW literature
\cite{KerdreuxDAspremontPokutta2021UniformConvexFW}. Recent work has also
studied affine-invariant analyses and open-loop stepsizes
\cite{KerdreuxLiuLacosteJulienScieur2021Affine,WirthKerdreuxPokutta2023OpenLoop,WirthPenaPokutta2025OpenLoop}.
Our work is complementary: instead of modifying the FW variant or focusing only
on stepsize design, we change the oracle domain by restricting each LMO to a
random affine section.

\paragraph{Coordinate, block-coordinate, and product-domain methods.}
Randomized coordinate and block-coordinate methods reduce per-iteration cost by
updating only part of the variable
\cite{Nesterov2012CoordinateDescent,RichtarikTakac2014RBCD}. Block-coordinate
Frank--Wolfe methods typically exploit product or block structure in the
feasible set \cite{LacosteJulienJaggiSchmidtPletscher2013BlockCoordinateFW},
and recent projection-free methods on product domains further develop this
viewpoint \cite{BomzeRinaldiZeffiro2025ProductDomains}. RSFW has a different
motivation: it does not require a product decomposition of \(C\). Instead, it
samples a random low-dimensional affine section of the original feasible region
and solves the LMO exactly on that section.

\paragraph{Random embeddings and random-subspace optimization.}
Random subspaces have also been used to reduce high-dimensional optimization
problems to lower-dimensional subproblems. In unconstrained optimization, a
common strategy is to choose a random direction or random subspace and then take
a gradient-free, Newton-type, regularized Newton, or model-based step in that
subspace; examples include random pursuit
\cite{StichMullerGartner2013RandomPursuit}, stochastic gradient-free subspace
methods \cite{KozakBeckerDoostanTenorio2021Subspace}, randomized subspace
regularized Newton methods \cite{FujiPoirionTakeda2025RSRNM}, and
model-based derivative-free subspace methods
\cite{CartisRoberts2023SubspaceDFO}. In Bayesian optimization, random
embeddings reduce the search to a low-dimensional embedded domain, as in REMBO
\cite{WangZoghiHutterMathesonDeFreitas2013REMBO} and subsequent embedded
Bayesian-optimization frameworks
\cite{NayebiMunteanuPoloczek2019EmbeddedBO}. Closest in spirit to our
feasible-section viewpoint are the works of Cartis, Massart, and Otemissov,
who study global optimization using feasible random embeddings, including a
bound-constrained setting where the reduced subproblem optimizes the original
objective over a random low-dimensional subspace subject to feasibility
constraints
\cite{CartisMassartOtemissov2023BoundConstrained,CartisMassartOtemissov2023GlobalEmbeddings}.
These methods are not Frank--Wolfe methods: their reduced subproblems optimize,
exactly or approximately, the original nonlinear objective or a local model. Their convergence analysis is driven by the probability that a random reduced
problem is successful; in the general setting this probability is expressed
through conic integral geometry and can be small in high-dimensional
random-intersection regimes, while more favorable polynomial bounds are obtained
in special aligned low-effective-dimensional settings. By
contrast, RSFW is an iterative convex first-order method: at each current
iterate, it solves a linear minimization problem exactly over the random
feasible section
\(C\cap(x_k+\operatorname{range}(P_k^\top))\), and the analysis quantifies the
resulting one-step Frank--Wolfe progress.

\paragraph{Sketching, subspace approximations, and representation-side randomization.}
Sketching and random-subspace methods are widely used to reduce computational,
linear-algebra, or storage costs. Examples include randomized sketches of convex
programs and iterative Hessian sketching for constrained least-squares problems
\cite{PilanciWainwright2015SharpSketches,PilanciWainwright2016IterativeHessianSketch},
Newton sketching for randomized second-order optimization
\cite{PilanciWainwright2017NewtonSketch}, randomized subspace approximations
for high-dimensional optimization
\cite{LacottePilanciPavone2019AdaptiveRandomSubspaces,LacottePilanci2022RandomizedSubspaceMethods},
randomized sketch descent for linearly constrained optimization
\cite{NecoaraTakac2021RandomSketchDescent}, and sketch-and-project methods for
linear systems and related iterative algorithms
\cite{GowerRichtarik2015SketchAndProject,DerezinskiRebrova2024SharpSketchProject}.
Sketch-based low-rank matrix optimization and sketchy representations provide
another example of reducing storage or representation costs
\cite{YurtseverUdellTroppCevher2017Sketchy}. These methods sketch or restrict
the linear algebra, Hessian or system solves, variable subspaces, constraint
systems, or stored iterate representations used by the algorithm. RSFW instead
restricts the feasible domain seen by the LMO and solves the resulting section
oracle exactly.

\paragraph{Inexact, lazy, decomposed, and stochastic Frank--Wolfe.}
Several FW variants reduce oracle cost by allowing approximate LMOs, weak
separation, lazy oracle calls, blended updates, or adaptive curvature estimates
\cite{BraunPokuttaZink2017Lazifying,BraunPokuttaTuWright2019Blended,PedregosaNegiarAskariJaggi2020Backtracking,ZhouSun2022ApproximateFW}.
Other approaches replace a difficult LMO over an intersection by calls to LMOs
over the individual component sets \cite{WoodstockPokutta2025SplittingCG}.
Stochastic FW methods instead approximate the gradient through sampling,
mini-batches, or variance reduction
\cite{HazanKale2012ProjectionFreeOnline,HazanLuo2016VRPF,LocatelloYurtseverFercoqCevher2019SFW,NegiarDresdnerTsaiElGhaouiLocatelloFreundPedregosa2020SFW,DresdnerVladareanRatschLocatelloCevherYurtsever2022OneSample}.
RSFW is orthogonal to these approaches: the section oracle is exact, and the
approximation comes from restricting the feasible region to a random affine
section. In the finite-sum setting, the random section can be sampled
independently of the stochastic gradient estimator, allowing the deterministic
oracle-side analysis to combine with standard mini-batch variance bounds.

\paragraph{Spectrahedral projection-free methods.}
Optimization over the spectrahedron is a canonical setting where projection-free
methods are attractive because projections require expensive eigenvalue
decompositions, while FW-type steps can often be implemented through leading
eigenvector computations \cite{Garber2016Spectrahedron}. Recent work has also
developed randomized FW-type methods over the spectrahedron with linear
convergence under additional structural assumptions such as quadratic growth and
strict complementarity \cite{Garber2026RandomizedSpectrahedron}. RSFW is
different in scope: it does not exploit the facial structure of the
spectrahedron, but studies the effect of restricting the LMO to random affine
sections of curved feasible sets.

\section{Technical preliminaries}
\label{app:preliminaries}

\subsection{Curvature inequality}

\begin{lemma}[Curvature inequality]
\label{lem:curvature_clean}
Under Assumption~\ref{ass:f_clean}, for any $x,y\in C$ and $\alpha\in[0,1]$, letting $x^+=(1-\alpha)x+\alpha y$, we have
\begin{equation}
f(x^+)\le f(x)+\alpha\langle \nabla f(x),y-x\rangle+\frac{L}{2}\alpha^2\|y-x\|^2.
\label{eq:curvineq_clean}
\end{equation}
\end{lemma}

\subsection{Inner tangent balls}

\begin{lemma}[Inner tangent balls; uniform radius lower bound]
\label{lem:tangentBalls_clean}
Assume $C$ satisfies Assumption~\ref{ass:C_clean} and $\partial C$ is $\mathcal{C}^2$.
Then for every $x\in\partial C$ there exist an inward unit normal $\mu(x)\in S^{n-1}$ and a radius $R(x)>0$ such that
\begin{equation}
B(x+R(x)\mu(x),R(x))\subset C.
\label{eq:tangentBall_clean}
\end{equation}
Moreover, there exists $R_{\min}>0$ depending only on $C$ such that $R(x)\ge R_{\min}$ for all $x\in\partial C$.
\end{lemma}

\begin{proof}
This is a standard rolling-ball result for convex bodies with $\mathcal{C}^2$ boundary. In particular, Theorem~7 in \cite{SlotLaurent2022RollingBall} states that a convex body with $\mathcal{C}^2$ boundary admits an inscribed tangent ball of uniform positive radius at every boundary point. Applying that theorem to $C$ yields \eqref{eq:tangentBall_clean} and the existence of a uniform lower bound $R_{\min}>0$.
\end{proof}

\begin{remark}\label{rem:c11}
    We say that \(\partial C\) is \(\mathcal{C}^{1,1}\) if, around every boundary point,
\(\partial C\) is locally the graph of a \(\mathcal{C}^1\) function whose gradient is
Lipschitz. Equivalently, the outer normal map is locally Lipschitz on
\(\partial C\). For compact convex bodies, \(\mathcal{C}^{1,1}\) boundary is sufficient to obtain a
uniform interior rolling-ball radius. Thus the \(\mathcal{C}^2\) assumption can be
replaced by \(\mathcal{C}^{1,1}\), or more directly by the uniform interior rolling-ball
condition used in Lemma~\ref{lem:tangentBalls_clean}.
\end{remark}

\subsection{Outer parallel regularization}
\label{app:parallel_regularization}

\begin{proposition}[Parallel regularization]
\label{prop:parallel_regularization}
Let \(C\subset\mathbb R^n\) be a compact \(\beta_C\)-strongly convex convex
set. For every \(\varepsilon>0\), the parallel body
\[
    C_\varepsilon:=C+\varepsilon B_2
\]
has \(\mathcal C^{1,1}\) boundary and is \(\beta_{C,\varepsilon}\)-strongly convex with
\[
    \beta_{C,\varepsilon}
    =
    \frac{\beta_C}{1+2\varepsilon\beta_C}.
\]
Moreover \(d_H(C,C_\varepsilon)\le\varepsilon\).
\end{proposition}

\begin{proof}
Set
\[
    R_C:=\frac{1}{2\beta_C}.
\]
The outer-ball argument used in Lemma~\ref{lem:BallGap_clean} shows that, for
every boundary point \(x\in\partial C\) and every outward unit normal \(\nu\) at
\(x\),
\[
    C\subset B(x-R_C\nu,R_C).
\]
Equivalently, \(C\) admits outer tangent balls of radius \(R_C\).

Fix \(\varepsilon>0\), let \(y\in\partial C_\varepsilon\), and let \(\nu\) be an
outward unit normal to \(C_\varepsilon\) at \(y\). Since \(C\) is compact, the
support function is attained. Using
\[
    h_{C_\varepsilon}(\nu)=h_C(\nu)+\varepsilon,
\]
choose \(x\in C\) such that \(\langle \nu,x\rangle=h_C(\nu)\). Since
\(y\in C+\varepsilon B_2\) and
\(\langle \nu,y\rangle=h_{C_\varepsilon}(\nu)\), we have
\[
    y=x+\varepsilon\nu .
\]

The outer-ball property for \(C\) gives
\[
    C\subset B(x-R_C\nu,R_C).
\]
Taking Minkowski sums with \(\varepsilon B_2\), we obtain
\[
    C_\varepsilon
    =
    C+\varepsilon B_2
    \subset
    B(x-R_C\nu,R_C)+\varepsilon B_2
    =
    B(x-R_C\nu,R_C+\varepsilon).
\]
This ball is tangent to \(C_\varepsilon\) at \(y=x+\varepsilon\nu\). Hence
\(C_\varepsilon\) has outer tangent balls of radius \(R_C+\varepsilon\) at all
boundary points. By the same outer-ball characterization of strong convexity,
\(C_\varepsilon\) is
\[
    \frac{1}{2(R_C+\varepsilon)}
    =
    \frac{\beta_C}{1+2\varepsilon\beta_C}
\]
strongly convex.

It remains to prove the boundary regularity. At the same point
\(y=x+\varepsilon\nu\), the ball \(B(x,\varepsilon)\) is contained in
\(C_\varepsilon\) and tangent to \(\partial C_\varepsilon\) at \(y\). Also, the
exterior ball \(B(y+\varepsilon\nu,\varepsilon)\) is disjoint from
\(C_\varepsilon\). Indeed, for every \(z\in C_\varepsilon\),
\(\langle\nu,z-y\rangle\le0\), and hence
\[
    \|z-(y+\varepsilon\nu)\|^2
    =
    \|z-y\|^2
    -2\varepsilon\langle\nu,z-y\rangle
    +\varepsilon^2
    \ge
    \varepsilon^2 .
\]
Thus \(C_\varepsilon\) satisfies uniform interior and exterior rolling-ball
conditions with radius \(\varepsilon\). For convex bodies, this two-sided
rolling-ball condition is equivalent to \(\mathcal C^{1,1}\) boundary regularity.

Finally, \(C\subset C_\varepsilon\), and every point of \(C_\varepsilon\) lies
within distance at most \(\varepsilon\) of \(C\). Hence
\(d_H(C,C_\varepsilon)\le\varepsilon\).
\end{proof}

\subsection{A Bartlett-type decomposition of the projected Gram matrix}

A central tool throughout this section is the two-dimensional projected Gram structure of a Haar--Stiefel random subspace.
Fix orthonormal $e_1,e_2\in\mathbb R^n$ and let $P\in\mathbb R^{d\times n}$ be Haar--Stiefel ($PP^\top=I_d$, $d\ge 2$).
Let
\[
M:=(PE)^\top(PE),\qquad E=(e_1\ \ e_2),
\]
be the corresponding $2\times 2$ Gram matrix.

\begin{lemma}
\label{lem:Bartlett_prelim}
Assume \(2\le d\le n-1\). There exist independent random variables $R$, $T$, and $B$ such that
\begin{equation}
R\sim \BetaDist\!\left(\tfrac d2,\tfrac{n-d}2\right),
\qquad
B\sim \BetaDist\!\left(\tfrac{d-1}2,\tfrac{n-d-1}2\right) \mbox{ if \(d\le n-2\) otherwise \(B\equiv 1\) },
\label{eq:Bartlett_corrected_dists}
\end{equation}
and
\begin{equation}
\mathbb E[T]=0,
\qquad
\mathbb E[T^2]=\frac{1}{n-1},
\qquad
\mathbb E[T^4]=\frac{3}{(n-1)(n+1)}.
\label{eq:Bartlett_T_moments}
\end{equation}
Moreover,
\begin{equation}
\|Pe_1\|^2 = R,
\label{eq:Bartlett_corrected_R}
\end{equation}
\begin{equation}
\langle Pe_1,Pe_2\rangle = \sqrt{R(1-R)}\,T,
\label{eq:Bartlett_corrected_cross}
\end{equation}
\begin{equation}
\|Pe_2\|^2 = (1-R)T^2 + (1-T^2)B.
\label{eq:Bartlett_corrected_norm2}
\end{equation}
In particular, the Schur complement of $M$ with respect to its $(1,1)$ entry is
\begin{equation}
S_{\mathrm{Sch}}
:=
\|Pe_2\|^2-\frac{\langle Pe_1,Pe_2\rangle^2}{\|Pe_1\|^2}
=
(1-T^2)B,
\label{eq:Bartlett_corrected_schur}
\end{equation}
and therefore
\begin{equation}
\det(M)=R\,S_{\mathrm{Sch}}.
\label{eq:Bartlett_corrected_det}
\end{equation}
\end{lemma}

\begin{proof}
The decomposition is a standard consequence of the rotational invariance of
Haar--Stiefel projections and classical beta/Wishart identities; see, e.g.,
\cite{Muirhead1982}. First we state the overall proof idea.
\paragraph{Proof idea.}
We first expose the projection of one fixed direction \(e_1\). Its squared
projected length is a beta variable \(R\). Conditional on this projection, the
remaining \((d-1)\)-dimensional part of the random subspace is Haar inside an
\((n-2)\)-dimensional orthogonal complement. This separates the radial variable
\(R\), the angular coordinate \(T\), and the residual beta variable \(B\).

Let $\Pi:=P^\top P$ be the orthogonal projector onto the random $d$-dimensional subspace $U:=\mathrm{range}(P^\top)$.

Since $U$ is Haar on the Grassmannian, the squared length of the projection of a fixed unit vector onto $U$ has the Beta law
\[
R:=\|\Pi e_1\|^2=\|Pe_1\|^2
\sim \BetaDist\!\left(\tfrac d2,\tfrac{n-d}2\right).
\]
Also,
\[
\langle e_1,\Pi e_1\rangle=\|\Pi e_1\|^2=R,
\]
so the component of $\Pi e_1$ along $e_1$ is exactly $Re_1$.
Since $\|\Pi e_1\|^2=R$, its orthogonal component has squared norm
\[
\|\Pi e_1-Re_1\|^2 = R-R^2 = R(1-R).
\]
Hence there exists a random unit vector $\xi\in S^{n-2}\subset e_1^\perp$ such that
\begin{equation}
\Pi e_1 = R e_1 + \sqrt{R(1-R)}\,\xi.
\label{eq:A3_Pi_e1_rep}
\end{equation}
By rotational invariance in $e_1^\perp$, the random vector $\xi$ is uniform on $S^{n-2}$ and independent of $R$.

Define
\[
T:=\langle e_2,\xi\rangle\in[-1,1].
\]
Then $T$ is independent of $R$, and since $T$ is the first coordinate of a uniform point on $S^{n-2}$, its moments are
\[
\mathbb E[T]=0,
\qquad
\mathbb E[T^2]=\frac{1}{n-1},
\qquad
\mathbb E[T^4]=\frac{3}{(n-1)(n+1)}.
\]

Now let
\[
\widehat p:=\frac{\Pi e_1}{\|\Pi e_1\|}=\sqrt{R}\,e_1+\sqrt{1-R}\,\xi.
\]
Then $\widehat p\in U$.
Let
\[
E:=\mathrm{span}\{e_1,\xi\}.
\]
The residual $e_1-\Pi e_1$ lies in $E$ and is orthogonal to $U$, while $\widehat p\in U\cap E$.
Therefore the remaining part of $U$ lies in $E^\perp$.
Hence $U$ decomposes orthogonally as
\[
U=\mathrm{span}\{\widehat p\}\oplus U',
\]
where, conditionally on $(R,\xi)$, the subspace $U'$ is Haar of dimension $d-1$ inside the $(n-2)$-dimensional space $E^\perp$.

Let $\Pi'$ denote the orthogonal projector onto $U'$.
Then
\[
\Pi=\widehat p\,\widehat p^\top+\Pi'.
\]

We first compute the cross term:
\[
\langle Pe_1,Pe_2\rangle
=
\langle \Pi e_1,e_2\rangle.
\]
Using \eqref{eq:A3_Pi_e1_rep} and $\langle e_1,e_2\rangle=0$,
\[
\langle \Pi e_1,e_2\rangle
=
\sqrt{R(1-R)}\,\langle \xi,e_2\rangle
=
\sqrt{R(1-R)}\,T,
\]
which proves \eqref{eq:Bartlett_corrected_cross}.

Next,
\[
\|Pe_2\|^2=\|\Pi e_2\|^2
=
|\langle e_2,\widehat p\rangle|^2+\|\Pi' e_2\|^2,
\]
since the two components are orthogonal.
Because
\[
\langle e_2,\widehat p\rangle
=
\sqrt{1-R}\,\langle e_2,\xi\rangle
=
\sqrt{1-R}\,T,
\]
the first term equals $(1-R)T^2$.

For the second term, only the $E^\perp$-component of $e_2$ contributes.
Its squared norm is
\[
\|\operatorname{Proj}_{E^\perp} e_2\|^2 = 1-T^2.
\]
Conditional on $(R,T)$, the squared length of the projection of a fixed unit vector in $E^\perp$ onto the Haar $(d-1)$-subspace $U'$ has law
\[
B\sim \BetaDist\!\left(\tfrac{d-1}2,\tfrac{n-d-1}2\right),
\]
independent of $(R,T)$.
Therefore
\[
\|\Pi' e_2\|^2=(1-T^2)B,
\]
which proves \eqref{eq:Bartlett_corrected_norm2}.

Finally,
\[
S_{\mathrm{Sch}}
=
\|Pe_2\|^2-\frac{\langle Pe_1,Pe_2\rangle^2}{\|Pe_1\|^2}
=
\bigl((1-R)T^2+(1-T^2)B\bigr)-\frac{R(1-R)T^2}{R}
=
(1-T^2)B,
\]
which proves \eqref{eq:Bartlett_corrected_schur}. The determinant identity \eqref{eq:Bartlett_corrected_det} follows immediately.
\end{proof}

\begin{remark}[Role of the corrected Bartlett variables]
\label{rem:Bartlett_role}
Let $u=e_1$ and let
\[
v=\rho e_1+\sqrt{1-\rho^2}\,e_2,
\qquad
\rho=\langle u,v\rangle.
\]
Set $\sigma:=\sqrt{1-\rho^2}$.
Then the projected quantities satisfy
\begin{equation}
\|Pu\|^2 = R,
\label{eq:Bartlett_role_Pu}
\end{equation}
\begin{equation}
\langle Pu,Pv\rangle
=
\rho R+\sigma\sqrt{R(1-R)}\,T,
\label{eq:Bartlett_role_inner}
\end{equation}
\begin{equation}
\|Pv\|^2
=
\bigl(\rho\sqrt R+\sigma\sqrt{1-R}\,T\bigr)^2+\sigma^2(1-T^2)B.
\label{eq:Bartlett_role_Pv}
\end{equation}
Thus the radial variable $R$ is independent of the remaining shape variables $(T,B)$.
This is the decoupling used in the expectation analysis of Section~B and in the mixed-moment lower bounds stated below.
\end{remark}

\subsection{Deterministic reduction to ball sections}

\begin{lemma}[Subspace oracle dominates a ball section]
\label{lem:subvsball}
Let $x_k\in C$, let $g_k=\nabla f(x_k)$, let $U_k=\range(\Pi_k)$, and let $\mathcal B\subset C$ be any nonempty subset containing $x_k$. If $d_k=s_k-x_k$ is defined by \eqref{eq:subLMO_clean}, then
\begin{equation}
\langle g_k,d_k\rangle
\le \min\{\langle g_k,y-x_k\rangle:\ y\in \mathcal B\cap(x_k+U_k)\}.
\label{eq:subDom}
\end{equation}
\end{lemma}

\begin{proof}
Since $\mathcal B\subset C$, we have $\mathcal B\cap(x_k+U_k)\subset C\cap(x_k+U_k)$. The point $s_k$ minimizes $y\mapsto \langle g_k,y\rangle$ over the larger set $C\cap(x_k+U_k)$, hence also dominates the minimum over the smaller set $\mathcal B\cap(x_k+U_k)$. Subtracting $\langle g_k,x_k\rangle$ gives \eqref{eq:subDom}.
\end{proof}

\subsection{Ball-section improvement formulas}

\begin{lemma}[Exact ball-section improvement (boundary base point)]
\label{lem:ballSectionExact_clean}
Let $x_k\in\mathbb R^n$, $g_k\in\mathbb R^n$, $\mu_k\in S^{n-1}$, $r>0$, and let $U_k$ be a subspace with projector $\Pi_k$. Assume $\Pi_k g_k\neq 0$. Then the minimizer of $y\mapsto \langle g_k,y\rangle$ over the section
\begin{equation}
\{y\in U_k:\ x_k+y\in B(x_k+r\mu_k,r)\}
=\{y\in U_k:\ \|y-r\mu_k\|\le r\}
\label{eq:SectionSet}
\end{equation}
is given by
\begin{equation}
y^* = r\Pi_k \mu_k - r\|\Pi_k \mu_k\|\frac{\Pi_k g_k}{\|\Pi_k g_k\|},
\label{eq:tstar}
\end{equation}
and the corresponding improvement equals
\begin{equation}
\Delta_{k}^{\partial}(r)
:=
\langle g_k,0\rangle-\min\{\langle g_k,y\rangle:\ y\in U_k,\ \|y-r\mu_k\|\le r\}
=
r\Big(\|\Pi_k g_k\|\,\|\Pi_k \mu_k\|-\langle \Pi_k g_k,\Pi_k \mu_k\rangle\Big).
\label{eq:BallImproveExact}
\end{equation}
Moreover, in the full-space case ($U_k=\mathbb R^n$, $\Pi_k=I_n$), the corresponding improvement is
\begin{equation}
\Delta_{\mathrm{full},k}^{\partial}(r)
=
r\Big(\|g_k\|-\langle g_k,\mu_k\rangle\Big).
\label{eq:BallImproveExact_full}
\end{equation}
\end{lemma}

\begin{proof}
For $y\in U_k$, decompose $r\mu_k=r\Pi_k\mu_k+r(I-\Pi_k)\mu_k$. Since $y-r\Pi_k\mu_k\in U_k$ and $r(I-\Pi_k)\mu_k\in U_k^\perp$, they are orthogonal, hence
\[
\|y-r\mu_k\|^2=\|y-r\Pi_k\mu_k\|^2+r^2\|(I-\Pi_k)\mu_k\|^2.
\]
Thus the section is the Euclidean ball in $U_k$ of radius $r\|\Pi_k\mu_k\|$ centered at $r\Pi_k\mu_k$, and minimizing a linear functional over that ball gives \eqref{eq:tstar} and \eqref{eq:BallImproveExact}. The full-space formula is the special case $\Pi_k=I_n$.
\end{proof}

\begin{lemma}[Exact ball-section improvement (interior base point)]
\label{lem:ballSectionInterior_clean}
Let $x_k\in\mathbb R^n$, $g_k\in\mathbb R^n$, $r>0$, and let $\mathcal B_k=B(c_k,r)$ be a ball. Assume $x_k\in \mathcal B_k$ and define
\begin{equation}
v_k:=\frac{c_k-x_k}{r}\in\mathbb R^n,
\qquad \|v_k\|\le 1.
\label{eq:vInteriorDef}
\end{equation}
Let $U_k$ be a subspace with projector $\Pi_k$, and assume $\Pi_k g_k\neq 0$. Then the minimizer of $y\mapsto \langle g_k,y\rangle$ over the section
\begin{equation}
\{y\in U_k:\ x_k+y\in \mathcal B_k\}=\{y\in U_k:\ \|y-rv_k\|\le r\}
\label{eq:SectionSetInterior}
\end{equation}
is
\begin{equation}
y^*=r\Pi_k v_k-r\sqrt{1-\|v_k\|^2+\|\Pi_k v_k\|^2}\,\frac{\Pi_k g_k}{\|\Pi_k g_k\|},
\label{eq:tstarInterior}
\end{equation}
and the corresponding improvement equals
\begin{align}
\Delta_{k}^{\mathrm{int}}(c_k,r)
:=&
\langle g_k,0\rangle-\min\{\langle g_k,y\rangle:\ y\in U_k,\ \|y-rv_k\|\le r\}\\
=&
r\Big(\|\Pi_k g_k\|\sqrt{1-\|v_k\|^2+\|\Pi_k v_k\|^2}-\langle \Pi_k g_k,\Pi_k v_k\rangle\Big).
\label{eq:BallImproveInterior}
\end{align}
Moreover, in the full-space case ($U_k=\mathbb R^n$, $\Pi_k=I_n$),
\begin{equation}
\Delta_{\mathrm{full},k}^{\mathrm{int}}(c_k,r)
=
r\Big(\|g_k\|-\langle g_k,v_k\rangle\Big).
\label{eq:BallImproveInterior_full}
\end{equation}
\end{lemma}

\begin{proof}
For $y\in U_k$, orthogonal decomposition gives
\[
\|y-rv_k\|^2=\|y-r\Pi_k v_k\|^2+r^2\|(I-\Pi_k)v_k\|^2.
\]
Thus the section is the Euclidean ball in $U_k$ centered at $r\Pi_k v_k$ and with radius
\[
r\sqrt{1-\|v_k\|^2+\|\Pi_k v_k\|^2}.
\]
Minimizing the linear functional over that ball yields \eqref{eq:tstarInterior}, and evaluating the objective value gives \eqref{eq:BallImproveInterior}. Setting $\Pi_k=I_n$ yields \eqref{eq:BallImproveInterior_full}.
\end{proof}

\begin{corollary}[Interior ratio dominates the boundary expression]
\label{cor:gammaIntGeGamma}
Fix $u\in S^{n-1}$ and $v\in\mathbb R^n$ with $\|v\|\le 1$ and $\langle u,v\rangle<1$. Then
\begin{equation}
\gamma(P;u,v)\ge \Gamma(P;u,v).
\label{eq:gammaIntGeGamma}
\end{equation}
In particular, when $\|v\|=1$, the two ratios coincide.
\end{corollary}

\begin{proof}
Since $1-\|v\|^2\ge 0$, we have
\[
\sqrt{1-\|v\|^2+\|Pv\|^2}\ge \|Pv\|.
\]
Multiplying by $\|P u\|\ge 0$, subtracting $\langle P u,P v\rangle$, and dividing by $1-\langle u,v\rangle>0$ gives the claim.
\end{proof}

\begin{lemma}[Interior ratio lower bound via a unit-direction surrogate]
\label{lem:gamma_scaled_unit}
Let \(u,\mu\in S^{n-1}\) and \(t\in[0,1]\). Set \(v:=t\mu\). Then
\begin{equation}\label{eq:gamma_scaled_unit}
\gamma(P;u,t\mu)
\ge
\frac{(1-t)\|Pu\|+t(1-\langle u,\mu\rangle)\Gamma(P;u,\mu)}
{(1-t)+t(1-\langle u,\mu\rangle)}.
\end{equation}
In particular,
\begin{equation}\label{eq:gamma_scaled_unit2}
\gamma(P;u,t\mu)\ge \min\{\|Pu\|,\Gamma(P;u,\mu)\}.
\end{equation}
\end{lemma}

\begin{proof}
Write
\[
a:=\|P u\|,
\qquad
b:=\|P \mu\|,
\qquad
m:=\langle P u,P \mu\rangle,
\qquad
\rho:=\langle u,\mu\rangle.
\]
Then
\[
\gamma(P;u,t\mu)
=
\frac{a\sqrt{1-t^2+t^2b^2}-tm}{1-t\rho}.
\]
We first claim that
\begin{equation}
\sqrt{1-t^2+t^2b^2}\ge (1-t)+tb.
\label{eq:sqrt_affine_lb}
\end{equation}
Indeed, after squaring both sides, \eqref{eq:sqrt_affine_lb} is equivalent to
\[
1-t^2+t^2b^2\ge (1-t+tb)^2,
\]
and the difference between the two sides is
\[
1-t^2+t^2b^2-(1-t+tb)^2
=2(1-b)t(1-t)\ge 0,
\]
because $b\le 1$ and $t\in[0,1]$.
Using \eqref{eq:sqrt_affine_lb}, we obtain
\[
\gamma(P;u,t\mu)
\ge
\frac{(1-t)a+t(ab-m)}{(1-t)+t(1-\rho)}.
\]
Now note that
\[
\Gamma(P;u,\mu)=\frac{ab-m}{1-\rho}.
\]
Therefore the last lower bound can be rewritten as
\[
\gamma(P;u,t\mu)
\ge
\frac{(1-t)\cdot 1\cdot a+t(1-\rho)\Gamma(P;u,\mu)}{(1-t)\cdot 1+t(1-\rho)}.
\]
A weighted average of two real numbers is always at least their minimum, hence
\[
\gamma(P;u,t\mu)\ge \min\{a,\Gamma(P;u,\mu)\}.
\]
This is exactly \eqref{eq:gamma_scaled_unit} and \eqref{eq:gamma_scaled_unit2}.
\end{proof}

\begin{remark}
Corollary~\ref{cor:gammaIntGeGamma} with $v=t\mu$ gives
\[
\gamma(P;u,t\mu)\ge \Gamma(P;u,t\mu)
=
\frac{t(1-\langle u,\mu\rangle)}{1-t\langle u,\mu\rangle}\,\Gamma(P;u,\mu),
\]
which can be weak when $t$ is small. Lemma~\ref{lem:gamma_scaled_unit} provides a different and more useful lower bound in this regime by comparing $\gamma(P;u,t\mu)$ to a unit-direction surrogate involving $\mu$.
\end{remark}

\begin{lemma}[JL estimate for Haar--Stiefel subspaces on finite sets]
\label{lem:JL_clean}
Let $P\in\mathbb R^{d\times n}$ be Haar on $\{P:PP^\top=I_d\}$. There exist absolute constants $c_{\mathrm{JL}},C_{\mathrm{JL}}>0$ such that for every $\varepsilon\in(0,1/2)$ and finite $\mathcal Z\subset S^{n-1}$,
\begin{equation}
\mathbb P\!\left(\forall z\in\mathcal Z:\ \left|\|Pz\|^2-\frac dn\right|\le \varepsilon\frac dn\right)
\ge 1-C_{\mathrm{JL}}|\mathcal Z|\exp(-c_{\mathrm{JL}}\varepsilon^2 d).
\label{eq:JLfinite_clean}
\end{equation}
\end{lemma}

\section{Expected lower bounds for the efficiency ratio}
\label{app:expected-gamma}

This section proves the quantitative lower bounds for the section-efficiency
ratio \(\Gamma(P;u,v)\). The first result gives a sharp identity showing that
\(\mathbb E\Gamma\) is within \(O(1/n)\) of \(d/n\). The second route gives a cruder but uniformly positive lower bound of order
\(d/n\). We keep it as a complementary estimate because it exposes the
two-dimensional determinant structure and remains positive even when the sharp
identity-based lower bound is not useful.

\subsection{Exact identity for $\mathbb E[\Gamma]$}

\begin{proof}[Proof of Proposition~\ref{prop:EGamma_identity}]
Let
\[
Z(P;u,v):=\|Pu\|\,\|Pv\|-\langle Pu,Pv\rangle,
\qquad
\Gamma(P;u,v)=\frac{Z(P;u,v)}{1-\rho}.
\]
The algebraic identity
\[
2Z(P;u,v)=\|P(u-v)\|^2-(\|Pu\|-\|Pv\|)^2
\]
holds for every matrix $P$.
Taking expectations and using
\(\mathbb E\|P(u-v)\|^2=(d/n)\|u-v\|^2=2(d/n)(1-\rho)\), we obtain
\[
\mathbb E[\Gamma(P;u,v)]
=
\frac dn
-
\frac{\mathbb E(\|Pu\|-\|Pv\|)^2}{2(1-\rho)}.
\]
The lower bound follows from Corollary~\ref{cor:clean_bound} below, and the
upper bound is immediate from the nonnegativity of the correction term.
\end{proof}

Let $\Pi:=P^\top P$ be the orthogonal projector onto the random $d$-dimensional subspace $U:=\mathrm{range}(P^\top)$.
Fix unit vectors $u,v\in S^{n-1}$ and write
\begin{equation}
\rho:=\langle u,v\rangle <1.
\label{eq:def_rho_c}
\end{equation}
Define
\begin{equation}
Y:=\|Pv\|^2=\|\Pi v\|^2,
\qquad
X:=\|Pu\|^2=\|\Pi u\|^2.
\label{eq:def_Y_X}
\end{equation}

\subsection{Exact conditioning structure}

\begin{lemma}[Exact conditional decomposition]
\label{lem:cond_decomp}
By rotational invariance, assume without loss of generality that
\begin{equation}
v=e_1,
\qquad
u=\rho e_1+\sqrt{1-\rho^2}\,e_2.
\label{eq:wlog_basis}
\end{equation}
Set
\[
\sigma:=\sqrt{1-\rho^2}.
\]
Assume \(2\le d<n\), then there exist random variables $(S,T,B)$ such that:
\begin{enumerate}
\item $S\in(0,1)$ almost surely and
\begin{equation}
S\sim \BetaDist\!\left(\tfrac d2,\tfrac{n-d}{2}\right).
\label{eq:S_beta}
\end{equation}
\item $T\in[-1,1]$ is independent of $S$ and satisfies
\begin{equation}
\mathbb E[T]=0,
\qquad
\mathbb E[T^2]=\frac{1}{n-1},
\qquad
\mathbb E[T^4]=\frac{3}{(n-1)(n+1)}.
\label{eq:T_moments}
\end{equation}
\item \(B\) is independent of \((S,T)\). If \(d\le n-2\), then
\[
B\sim \BetaDist\!\left(\tfrac{d-1}{2},\tfrac{n-d-1}{2}\right),
\]
while if \(d=n-1\), we use the degenerate convention \(B\equiv1\).
\item Defining
\begin{equation}
\alpha:=\rho\sqrt S+\sigma\sqrt{1-S}\,T,
\label{eq:def_alpha}
\end{equation}
one has
\begin{equation}
X=\alpha^2+\sigma^2(1-T^2)B.
\label{eq:X_decomp}
\end{equation}
\end{enumerate}
\end{lemma}

\begin{proof}
By Lemma~\ref{lem:Bartlett_prelim}, applied to the pair $(e_1,e_2)$, there exist independent random variables $(R,T,B)$ such that
\[
R\sim \BetaDist\!\left(\tfrac d2,\tfrac{n-d}{2}\right),
\qquad
\mbox{with the same degenerate convention \(B\equiv1\) when \(d=n-1\).}
\]
\[
\mathbb E[T]=0,
\qquad
\mathbb E[T^2]=\frac{1}{n-1},
\qquad
\mathbb E[T^4]=\frac{3}{(n-1)(n+1)},
\]
and
\begin{equation}
\|Pe_1\|^2 = R,
\label{eq:B2_R}
\end{equation}
\begin{equation}
\langle Pe_1,Pe_2\rangle = \sqrt{R(1-R)}\,T,
\label{eq:B2_cross}
\end{equation}
\begin{equation}
\|Pe_2\|^2 = (1-R)T^2 + (1-T^2)B.
\label{eq:B2_Pe2}
\end{equation}

Since $v=e_1$, we have
\[
S=\|Pv\|^2=\|Pe_1\|^2=R.
\]
Hence $S$ has the distribution \eqref{eq:S_beta}, and $T$ and $B$ are independent of $S$.

Now
\[
Pu=\rho Pe_1+\sigma Pe_2.
\]
Therefore
\begin{align*}
X
&=\|Pu\|^2 \\
&=\rho^2\|Pe_1\|^2 + 2\rho\sigma\langle Pe_1,Pe_2\rangle + \sigma^2\|Pe_2\|^2 \\
&=\rho^2 S + 2\rho\sigma\sqrt{S(1-S)}\,T + \sigma^2\bigl((1-S)T^2 + (1-T^2)B\bigr) \\
&=\bigl(\rho\sqrt S+\sigma\sqrt{1-S}\,T\bigr)^2 + \sigma^2(1-T^2)B.
\end{align*}
This is exactly \eqref{eq:X_decomp}.
\end{proof}

\subsection{Bounding $\mathbb E(\|Pu\|-\|Pv\|)^2$}

We start from the exact identity
\begin{equation}
(\sqrt X-\sqrt Y)^2=\frac{(X-Y)^2}{(\sqrt X+\sqrt Y)^2}.
\label{eq:sqrt_identity}
\end{equation}
Since $(\sqrt X+\sqrt Y)^2\ge Y$, we have the pointwise bound
\begin{equation}
(\|Pu\|-\|Pv\|)^2\le \frac{(X-Y)^2}{Y}.
\label{eq:drop_denom_to_Y}
\end{equation}

\begin{proposition}[Exact formula for \(\mathbb E((X-Y)^2/Y)\)]
\label{prop:exact_Delta2_over_Y}
Assume \(3\le d<n\). Then
\[
\mathbb E\!\left[\frac{(X-Y)^2}{Y}\right]
=
\frac{4(n-d)}{n(n-1)}(1-\rho^2)
\left(
\rho^2+
\frac{(1-\rho^2)(n-d+2)}{2(d-2)(n+1)}
\right).
\]
\end{proposition}

\begin{proof}
Let us give first a global overview of the proof.
\paragraph{Proof idea:}
The identity for \(\mathbb E[\Gamma]\) reduces the problem to controlling
\(\mathbb E(\|Pu\|-\|Pv\|)^2\). We write
\(X=\|Pu\|^2\), \(Y=\|Pv\|^2\), use
\((\sqrt X-\sqrt Y)^2\le (X-Y)^2/Y\), and then compute the latter exactly using
the conditional decomposition from Lemma~\ref{lem:cond_decomp}. The calculation
is only algebraic: after conditioning on \(S=Y\), all odd powers of \(T\)
vanish and the remaining beta moments are explicit.

Set
\[
\sigma:=\sqrt{1-\rho^2}.
\]
By Lemma~\ref{lem:cond_decomp},
\begin{equation}
X-S=(\alpha^2-S)+\sigma^2(1-T^2)B.
\label{eq:Delta_split_corrected}
\end{equation}
Therefore
\begin{equation}
(X-S)^2
=
(\alpha^2-S)^2
+
2(\alpha^2-S)\sigma^2(1-T^2)B
+
\sigma^4(1-T^2)^2B^2.
\label{eq:expand_square_corrected}
\end{equation}

Conditional on $(S,T)$, the Beta variable
\[
B\sim \BetaDist\!\left(\tfrac{d-1}{2},\tfrac{n-d-1}{2}\right)
\]
is independent of $(S,T)$, hence
\begin{equation}
\mathbb E\big[(X-S)^2\mid S,T\big]
=
(\alpha^2-S)^2
+
2(\alpha^2-S)\sigma^2(1-T^2)\,\mathbb E[B]
+
\sigma^4(1-T^2)^2\,\mathbb E[B^2].
\label{eq:cond_second_moment_corrected}
\end{equation}
For this Beta distribution,
\begin{equation}
\mathbb E[B]=\frac{d-1}{n-2},
\qquad
\mathbb E[B^2]=\frac{d^2-1}{n(n-2)}.
\label{eq:B_moments_corrected}
\end{equation}

We now average over $T$ conditional on $S$.
Using
\[
\alpha=\rho\sqrt S+\sigma\sqrt{1-S}\,T,
\]
we get
\begin{equation}
\alpha^2-S
=
\sigma^2\bigl(-S+(1-S)T^2\bigr)
+
2\rho\sigma\sqrt{S(1-S)}\,T.
\label{eq:alpha2_minus_S_corrected}
\end{equation}
Define
\[
A(S,T):=\sigma^2\bigl(-S+(1-S)T^2\bigr),
\qquad
L(S):=2\rho\sigma\sqrt{S(1-S)},
\]
so that
\[
\alpha^2-S=A(S,T)+L(S)\,T.
\]

Since $T$ is symmetric and $\mathbb E[T]=\mathbb E[T^3]=0$, all odd terms in $T$ vanish after taking expectation.
Using \eqref{eq:T_moments}, we obtain
\begin{align}
\mathbb E\big[(\alpha^2-S)^2\mid S\big]
&=
\mathbb E\big[A(S,T)^2\mid S\big]
+
\mathbb E\big[L(S)^2T^2\mid S\big]
\notag\\
&=
\sigma^4\,
\mathbb E\!\left[\bigl(-S+(1-S)T^2\bigr)^2\mid S\right]
+
4\rho^2\sigma^2 S(1-S)\,\mathbb E[T^2]
\notag\\
&=
\sigma^4\left(
S^2
-2S(1-S)\mathbb E[T^2]
+(1-S)^2\mathbb E[T^4]
\right)
+
\frac{4\rho^2\sigma^2}{n-1}S(1-S)
\notag\\
&=
\sigma^4\left(
S^2
-\frac{2S(1-S)}{n-1}
+\frac{3(1-S)^2}{(n-1)(n+1)}
\right)
+
\frac{4\rho^2\sigma^2}{n-1}S(1-S).
\label{eq:ET_alphaS_sq_corrected}
\end{align}

Similarly,
\begin{align}
\mathbb E\big[(\alpha^2-S)(1-T^2)\mid S\big]
&=
\mathbb E\big[A(S,T)(1-T^2)\mid S\big]
\notag\\
&=
\sigma^2\,\mathbb E\!\left[\bigl(-S+(1-S)T^2\bigr)(1-T^2)\mid S\right]
\notag\\
&=
\sigma^2\left(
-S(1-\mathbb E[T^2])+(1-S)(\mathbb E[T^2]-\mathbb E[T^4])
\right)
\notag\\
&=
\sigma^2\left(
-S\frac{n-2}{n-1}
+
(1-S)\frac{n-2}{(n-1)(n+1)}
\right).
\label{eq:ET_cross_corrected}
\end{align}

Finally,
\begin{align}
\mathbb E\big[(1-T^2)^2\big]
&=
1-2\mathbb E[T^2]+\mathbb E[T^4]
\notag\\
&=
1-\frac{2}{n-1}+\frac{3}{(n-1)(n+1)}
\notag\\
&=
\frac{n(n-2)}{(n-1)(n+1)}.
\label{eq:ET_one_minus_T2_sq_corrected}
\end{align}

Substituting \eqref{eq:B_moments_corrected}, \eqref{eq:ET_alphaS_sq_corrected}, \eqref{eq:ET_cross_corrected}, and \eqref{eq:ET_one_minus_T2_sq_corrected} into \eqref{eq:cond_second_moment_corrected}, we get
\begin{align}
\mathbb E\big[(X-S)^2\mid S\big]
&=
\sigma^4\left(
S^2
-\frac{2S(1-S)}{n-1}
+\frac{3(1-S)^2}{(n-1)(n+1)}
\right)
+
\frac{4\rho^2\sigma^2}{n-1}S(1-S)
\notag\\
&\quad
+
2\sigma^2
\left[
\sigma^2\left(
-S\frac{n-2}{n-1}
+
(1-S)\frac{n-2}{(n-1)(n+1)}
\right)
\right]
\frac{d-1}{n-2}
\notag\\
&\quad
+
\sigma^4\frac{n(n-2)}{(n-1)(n+1)}\frac{d^2-1}{n(n-2)}.
\label{eq:after_substitution_corrected}
\end{align}

Collecting the $\sigma^4$-terms and simplifying yields
\begin{equation}
\mathbb E\big[(X-S)^2\mid S\big]
=
\frac{4\rho^2\sigma^2}{n-1}S(1-S)
+
\frac{\sigma^4}{(n-1)(n+1)}
\Bigl((n^2+2n+4)S^2-2(dn+2d+2)S+d(d+2)\Bigr).
\label{eq:cond_second_moment_compact_corrected}
\end{equation}

We now divide by $S$ and take expectation over $S$:
\begin{align}
\mathbb E\!\left[\frac{(X-S)^2}{S}\right]
&=
\frac{4\rho^2\sigma^2}{n-1}\,\mathbb E[1-S]
\notag\\
&\quad
+
\frac{\sigma^4}{(n-1)(n+1)}
\left(
(n^2+2n+4)\mathbb E[S]
-2(dn+2d+2)
+d(d+2)\mathbb E\!\left[\frac1S\right]
\right).
\label{eq:Delta2_over_S_split_corrected}
\end{align}

Since
\[
S\sim \BetaDist\!\left(\tfrac d2,\tfrac{n-d}{2}\right),
\]
we have
\begin{equation}
\mathbb E[S]=\frac{d}{n},
\qquad
\mathbb E[1-S]=\frac{n-d}{n},
\qquad
\mathbb E\!\left[\frac1S\right]=\frac{n-2}{d-2},
\quad d>2.
\label{eq:S_moments_corrected}
\end{equation}
Substituting \eqref{eq:S_moments_corrected} into \eqref{eq:Delta2_over_S_split_corrected} gives
\begin{align}
\mathbb E\!\left[\frac{(X-S)^2}{S}\right]
&=
\frac{4\rho^2\sigma^2(n-d)}{n(n-1)}
\notag\\
&\quad
+
\frac{\sigma^4}{(n-1)(n+1)}
\left(
(n^2+2n+4)\frac dn
-2(dn+2d+2)
+d(d+2)\frac{n-2}{d-2}
\right).
\label{eq:after_S_moments_corrected}
\end{align}

The bracket simplifies to
\begin{equation}
(n^2+2n+4)\frac dn
-2(dn+2d+2)
+d(d+2)\frac{n-2}{d-2}
=
\frac{2(n-d)(n-d+2)(n+1)}{n(d-2)}.
\label{eq:bracket_common_den_corrected}
\end{equation}
Substituting \eqref{eq:bracket_common_den_corrected} into \eqref{eq:after_S_moments_corrected}, we obtain
\begin{align}
\mathbb E\!\left[\frac{(X-S)^2}{S}\right]
&=
\frac{4\rho^2\sigma^2(n-d)}{n(n-1)}
+
\frac{2\sigma^4(n-d)(n-d+2)}{n(n-1)(d-2)(n+1)}
\notag\\
&=
\frac{4(n-d)}{n(n-1)}(1-\rho^2)
\left(
\rho^2+\frac{(1-\rho^2)(n-d+2)}{2(d-2)(n+1)}
\right).
\label{eq:almost_done_corrected}
\end{align}
This is the claimed formula.
\end{proof}

\begin{corollary}
\label{cor:clean_bound}
Assume \(3\le d<n\). Then
\begin{equation}
\mathbb E\big[(\|Pu\|-\|Pv\|)^2\big]
\le
\mathbb E\!\left[\frac{(X-Y)^2}{Y}\right]
\le
\frac{8(n-d)}{n(n-1)}(1-\rho).
\label{eq:final_bound}
\end{equation}
In particular,
\begin{equation}
\mathbb E\big[(\|Pu\|-\|Pv\|)^2\big]
\le
C_{n,d}\,\frac{1-\langle u,v\rangle}{n},
\qquad
C_{n,d}:=\frac{8(n-d)}{n-1}\le 8.
\label{eq:clean_bound_scaled}
\end{equation}
\end{corollary}

\begin{proof}
The first inequality is \eqref{eq:drop_denom_to_Y}. By
Proposition~\ref{prop:exact_Delta2_over_Y}, with the corrected exact formula,
\[
\mathbb E\!\left[\frac{(X-Y)^2}{Y}\right]
=
\frac{4(n-d)}{n(n-1)}(1-\rho^2)
\left(
\rho^2+
\frac{(1-\rho^2)(n-d+2)}{2(d-2)(n+1)}
\right).
\]
Set
\[
A_{n,d}:=\frac{n-d+2}{2(d-2)(n+1)}.
\]
Since \(3\le d<n\),
\[
A_{n,d}\le \frac{n-1}{2(n+1)}<1.
\]
Hence
\[
\rho^2+A_{n,d}(1-\rho^2)\le 1.
\]
Therefore
\[
\mathbb E\!\left[\frac{(X-Y)^2}{Y}\right]
\le
\frac{4(n-d)}{n(n-1)}(1-\rho^2).
\]
Finally,
\[
1-\rho^2=(1-\rho)(1+\rho)\le 2(1-\rho),
\]
so
\[
\mathbb E\!\left[\frac{(X-Y)^2}{Y}\right]
\le
\frac{8(n-d)}{n(n-1)}(1-\rho).
\]
This proves \eqref{eq:final_bound}. The scaled form
\eqref{eq:clean_bound_scaled} follows by writing
\[
\frac{8(n-d)}{n(n-1)}
=
\frac{1}{n}\frac{8(n-d)}{n-1}.
\]
\end{proof}

\subsection{A secondary angle-uniform lower bound}

The identity in Proposition~\ref{prop:EGamma_identity} gives the sharp estimate
used in the main approximate-oracle argument. We also record a cruder but useful
angle-uniform lower bound. Its proof has a different geometric meaning: the
quantity \(\Gamma(P;u,v)\) is controlled from below by 
the smallest eigenvalue of the projected Gram matrix on the plane
\(\operatorname{span}\{u,v\}\), equivalently the squared smallest singular
value of the restricted projection \(\operatorname{span}\{u,v\}\).
This naturally produces the factor \(d-1\), reflecting the two-dimensional
projected Gram determinant.

\begin{lemma}[Moments of the projected Gram determinant]
\label{lem:Dmoments_clean}
Let \(M=(PE)^\top(PE)\) be the \(2\times2\) Gram matrix associated with an
orthonormal basis \(E=(e_1\ \ e_2)\) of a two-dimensional subspace. Assume
\(2\le d\le n-1\), and let \(D:=\det(M)\). With the notation of
Lemma~\ref{lem:Bartlett_prelim}, and with the convention \(B\equiv1\) when
\(d=n-1\), one has
\[
    D=R(1-T^2)B.
\]
Consequently,
\[
    \mathbb E[D]=\frac{d(d-1)}{n(n-1)}
\]
and
\[
    \mathbb E[D^2]
    =
    \frac{d(d+2)}{n(n+2)}
    \cdot
    \frac{(d-1)(d+1)}{(n-1)(n+1)}.
\]
Moreover,
\[
    \frac{\mathbb E[D^2]}{\mathbb E[D]^2}
    =
    \frac{(d+2)(d+1)}{d(d-1)}
    \cdot
    \frac{n(n-1)}{(n+2)(n+1)}
    \le 6 .
\]
\end{lemma}

\begin{proof}
By Lemma~\ref{lem:Bartlett_prelim},
\[
\|Pe_1\|^2 = R,\qquad
\langle Pe_1,Pe_2\rangle = \sqrt{R(1-R)}\,T,
\]
and
\[
\|Pe_2\|^2 = (1-R)T^2+(1-T^2)B.
\]
Thus
\[
D
=
\|Pe_1\|^2\|Pe_2\|^2-\langle Pe_1,Pe_2\rangle^2
=
R(1-T^2)B.
\]
The variables \(R,T,B\) are independent, with
\[
\mathbb E[R]=\frac dn,\qquad
\mathbb E[1-T^2]=\frac{n-2}{n-1},\qquad
\mathbb E[B]=\frac{d-1}{n-2},
\]
where the last identity is also valid under the convention \(B\equiv1\) when
\(d=n-1\). Hence
\[
\mathbb E[D]
=
\frac dn\cdot\frac{n-2}{n-1}\cdot\frac{d-1}{n-2}
=
\frac{d(d-1)}{n(n-1)}.
\]

Similarly,
\[
\mathbb E[D^2]
=
\mathbb E[R^2]\,
\mathbb E[(1-T^2)^2]\,
\mathbb E[B^2].
\]
Using
\[
\mathbb E[R^2]=\frac{d(d+2)}{n(n+2)},\qquad
\mathbb E[(1-T^2)^2]=\frac{n(n-2)}{(n-1)(n+1)},
\]
and
\[
\mathbb E[B^2]=\frac{(d-1)(d+1)}{n(n-2)},
\]
we obtain the displayed formula for \(\mathbb E[D^2]\). Dividing by
\(\mathbb E[D]^2\) gives
\[
\frac{\mathbb E[D^2]}{\mathbb E[D]^2}
=
\frac{(d+2)(d+1)}{d(d-1)}
\cdot
\frac{n(n-1)}{(n+2)(n+1)}.
\]
The second factor is at most \(1\), and the first is decreasing for \(d\ge2\)
and equals \(6\) at \(d=2\). This proves the bound.
\end{proof}

\begin{proposition}[A determinant-based lower bound for \(\mathbb E\Gamma\)]
\label{prop:Gamma_dn_improved}
Let \(P\) be Haar--Stiefel with \(d\ge2\) and \(n\ge d\). For any
\(u,v\in S^{n-1}\) with \(\rho:=\langle u,v\rangle<1\),
\[
    \mathbb E[\Gamma(P;u,v)]
    \ge
    \frac{d-1}{96(n-1)} .
\]
In particular, \(\mathbb E[\Gamma(P;u,v)]\gtrsim d/n\) uniformly over
\(\rho<1\).
\end{proposition}

\begin{proof}
If \(d=n\), then \(P\) is orthogonal and
\(\Gamma(P;u,v)=1\) almost surely, so the claim is immediate. Hence assume
\(2\le d\le n-1\).

Let \(V:=\operatorname{span}\{u,v\}\), and choose an orthonormal basis
\(E=(e_1\ \ e_2)\) of \(V\), with
\[
    e_1=u,\qquad
    e_2=\frac{v-\rho u}{\sqrt{1-\rho^2}}.
\]
Set \(M:=(PE)^\top(PE)\). For every \(z\in V\),
\[
    \|Pz\|^2\ge \lambda_{\min}(M)\|z\|^2 .
\]
Define
\[
    Z(P;u,v):=\|Pu\|\,\|Pv\|-\langle Pu,Pv\rangle.
\]
Since \(\|Pv\|>0\) almost surely, let
\[
    \lambda:=\frac{\|Pu\|}{\|Pv\|}.
\]
Then
\[
    \|Pu-\lambda Pv\|^2=2\lambda Z(P;u,v),
\]
and therefore
\[
    Z(P;u,v)
    =
    \frac{1}{2\lambda}\|P(u-\lambda v)\|^2
    \ge
    \frac{\lambda_{\min}(M)}{2\lambda}\|u-\lambda v\|^2.
\]
But
\[
    \|u-\lambda v\|^2
    =
    (\lambda-1)^2+2\lambda(1-\rho)
    \ge
    2\lambda(1-\rho).
\]
Hence
\[
    \Gamma(P;u,v)
    =
    \frac{Z(P;u,v)}{1-\rho}
    \ge
    \lambda_{\min}(M).
\]
Thus
\[
    \mathbb E[\Gamma(P;u,v)]
    \ge
    \mathbb E[\lambda_{\min}(M)].
\]

Let \(D:=\det(M)\) and \(\Theta:=\operatorname{tr}(M)\). Since \(M\succeq0\) is
\(2\times2\),
\[
    \lambda_{\min}(M)\ge \frac{D}{\Theta}.
\]
We lower bound \(\mathbb E[D/\Theta]\). For \(t>0\),
\[
\mathbb E\!\left[\frac{D}{\Theta}\right]
\ge
\frac1t\left(
    \mathbb E[D]
    -
    \mathbb E[D\mathbf 1_{\{\Theta>t\}}]
\right).
\]
By Cauchy--Schwarz,
\[
    \mathbb E[D\mathbf 1_{\{\Theta>t\}}]
    \le
    \sqrt{\mathbb E[D^2]\,\mathbb P(\Theta>t)}.
\]
Choose \(t=K\mathbb E[\Theta]\), \(K\ge1\). Markov's inequality gives
\(\mathbb P(\Theta>t)\le1/K\). Therefore
\[
\mathbb E\!\left[\frac{D}{\Theta}\right]
\ge
\frac{1}{K\mathbb E[\Theta]}
\left(
    \mathbb E[D]
    -
    \sqrt{\frac{\mathbb E[D^2]}{K}}
\right).
\]
By Lemma~\ref{lem:Dmoments_clean},
\[
    \mathbb E[D^2]\le 6\,\mathbb E[D]^2.
\]
Taking \(K=24\) gives
\[
    \sqrt{\frac{\mathbb E[D^2]}{K}}
    \le
    \frac12\mathbb E[D],
\]
so
\[
    \mathbb E\!\left[\frac{D}{\Theta}\right]
    \ge
    \frac{\mathbb E[D]}{48\,\mathbb E[\Theta]}.
\]
Finally,
\[
    \mathbb E[\Theta]
    =
    \mathbb E[\|Pe_1\|^2+\|Pe_2\|^2]
    =
    \frac{2d}{n},
\]
and Lemma~\ref{lem:Dmoments_clean} gives
\[
    \mathbb E[D]=\frac{d(d-1)}{n(n-1)}.
\]
Therefore
\[
\mathbb E[\Gamma(P;u,v)]
\ge
\frac{1}{48}
\frac{d(d-1)}{n(n-1)}
\frac{n}{2d}
=
\frac{d-1}{96(n-1)}.
\]
\end{proof}
\begin{remark}
The main approximate-oracle constant uses the sharper estimate from
Proposition~\ref{prop:EGamma_identity}, namely
\(c_{n,d}=d/n-4(n-d)/(n(n-1))\), when this quantity is positive. Proposition
\ref{prop:Gamma_dn_improved} is kept as a complementary angle-uniform bound
whose proof exposes the projected two-dimensional determinant structure and
retains a positive \((d-1)/(n-1)\) scale.
\end{remark}

\section{From ratio bounds to expected descent}
\label{app:expected-to-descent}
This section translates section-efficiency estimates into optimization
progress. The geometric step is to compare the full feasible set with tangent
balls. The probabilistic step is then applied only to those balls, where the
section improvement is explicit.
\subsection{Boundary case: factorization through $\Gamma$}

Assume $x_k\in\partial C$. Let
\[
\mu_k:=\mu(x_k),
\qquad
r_k:=R(x_k),
\]
as given by Lemma~\ref{lem:tangentBalls_clean}, so that
\[
\mathcal B_k:=B(x_k+r_k\mu_k,r_k)\subset C.
\]
If $g_k\neq 0$, define
\[
u_k:=\frac{g_k}{\|g_k\|},
\qquad
\delta_k:=1-\langle u_k,\mu_k\rangle.
\]
Then Lemma~\ref{lem:ballSectionExact_clean} yields the exact ball-section improvement
\begin{equation}
\Delta_k^{\mathrm B}(r_k)
=
r_k\|g_k\|\,\delta_k\,\Gamma(P_k;u_k,\mu_k).
\label{eq:BallImproveGamma_clean}
\end{equation}

\subsection{Boundary ball improvement versus the Frank--Wolfe gap}

For $x\in C$ and $g\neq 0$, define
\[
\Delta_C(x;g):=\max_{y\in C}\langle g,x-y\rangle.
\]
In particular, at the iterate $x_k$ we write $\Delta(x_k):=\Delta_C(x_k;g_k)$.

The next lemma is the geometric analogue of the familiar optimization
sandwich between smoothness and strong convexity. For a smooth strongly convex
function, one compares the function locally between two quadratic models. Here
we compare the feasible set \(C\) between two tangent balls. The inner tangent
ball comes from the rolling-ball condition, while the outer tangent ball is
forced by the strong convexity of \(C\). Since the Frank--Wolfe gap over a
tangent ball is explicit and scales linearly with the ball radius, these two
balls sandwich the Frank--Wolfe gap over \(C\).
\begin{lemma}[Ball--gap comparability on the boundary]
\label{lem:BallGap_clean}
Assume $C$ satisfies Assumption~\ref{ass:C_clean} and that $\partial C$ is $\mathcal{C}^2$. Let $x_k\in\partial C$, let $\mu_k:=\mu(x_k)$ be the inward unit normal, and let $R(x_k)$ be the inner rolling-ball radius from Lemma~\ref{lem:tangentBalls_clean}. Then for every $g_k=\nabla f(x_k)\neq 0$, with
\[
\hat g_k:=\frac{g_k}{\|g_k\|},
\qquad
\rho_k:=\langle \hat g_k,\mu_k\rangle,
\]
one has
\begin{equation}
\Delta_{\mathrm{ball}}^{\partial}(x_k)
:=
R(x_k)\|g_k\|(1-\rho_k)
\ge
\kappa_0\,\Delta(x_k),
\qquad
\kappa_0:=2\beta_C R_{\min}\in(0,1].
\label{eq:BallGap_boundary_clean}
\end{equation}
\end{lemma}

\begin{proof}
Write $x:=x_k$, $\mu:=\mu_k$, $R_{\mathrm{in}}:=R(x_k)$, $g:=g_k$, and
\[
\hat g:=\frac{g}{\|g\|},
\qquad
\rho:=\langle \hat g,\mu\rangle.
\]
We also write
\[
\Delta_C(x):=\langle g,x\rangle-\min_{y\in C}\langle g,y\rangle.
\]
The proof has two parts. First, strong convexity of \(C\) gives an outer tangent
ball of radius \(R_{\max}=1/(2\beta_C)\). Second, the rolling-ball condition
gives an inner tangent ball of radius \(R_{\mathrm{in}}\ge R_{\min}\). The FW
gap over \(C\) is then sandwiched between the two explicit tangent-ball gaps.
\paragraph{Step 1: construction of an outer tangent ball.}
Let
\[
H:=\{w\in\mathbb R^n:\ \langle \mu,w-x\rangle\ge 0\}
\]
be the supporting halfspace to $C$ at $x$, so that $C\subset H$.

Fix any $y\in C$ and $\lambda\in(0,1)$. By the quadratic uniform convexity assumption \eqref{eq:UC_clean},
\[
B\!\left(\lambda x+(1-\lambda)y,\ \beta_C\,\lambda(1-\lambda)\,\|x-y\|^2\right)\subset C\subset H.
\]
Since every point of this ball belongs to $H$, the minimum of
\[
w\longmapsto \langle \mu,w-x\rangle
\]
over that ball must be nonnegative. The minimum is attained at
\[
w=\lambda x+(1-\lambda)y-\beta_C\lambda(1-\lambda)\|x-y\|^2\,\mu,
\]
hence
\[
0\le
\langle \mu,\lambda x+(1-\lambda)y-x\rangle
-\beta_C\lambda(1-\lambda)\|x-y\|^2.
\]
Since
\[
\lambda x+(1-\lambda)y-x=(1-\lambda)(y-x),
\]
this becomes
\[
0\le
(1-\lambda)\langle \mu,y-x\rangle
-\beta_C\lambda(1-\lambda)\|x-y\|^2.
\]
Dividing by $(1-\lambda)>0$ gives
\[
\langle \mu,y-x\rangle\ge \beta_C\,\lambda\,\|x-y\|^2.
\]
Letting $\lambda\uparrow 1$, we obtain
\begin{equation}
\langle \mu,y-x\rangle\ge \beta_C\,\|x-y\|^2,
\qquad \forall y\in C.
\label{eq:quad_support_mu_clean}
\end{equation}

Now define
\[
R_{\max}:=\frac{1}{2\beta_C},
\qquad
B_{\max}:=B(x+R_{\max}\mu,\;R_{\max}).
\]
For any $y\in\mathbb R^n$,
\[
y\in B_{\max}
\iff
\|y-(x+R_{\max}\mu)\|^2\le R_{\max}^2.
\]
Expanding the left-hand side gives
\[
\|y-x\|^2-2R_{\max}\langle \mu,y-x\rangle\le 0,
\]
that is,
\[
y\in B_{\max}
\iff
\|y-x\|^2\le 2R_{\max}\langle \mu,y-x\rangle
=
\frac{1}{\beta_C}\langle \mu,y-x\rangle.
\]
By \eqref{eq:quad_support_mu_clean}, every $y\in C$ satisfies this inequality. Therefore
\begin{equation}
C\subset B_{\max}.
\label{eq:outer_ball_mu_clean}
\end{equation}

\paragraph{Step 2: sandwich the Frank--Wolfe gap between two ball gaps.}
Let
\[
B_{\min}:=B(x+R_{\mathrm{in}}\mu,\;R_{\mathrm{in}})\subset C
\]
be the inner tangent ball at $x$. Since
\[
B_{\min}\subset C\subset B_{\max},
\]
minimizing the linear form $y\mapsto \langle g,y\rangle$ over these sets yields
\[
\min_{y\in B_{\max}}\langle g,y\rangle
\le
\min_{y\in C}\langle g,y\rangle
\le
\min_{y\in B_{\min}}\langle g,y\rangle.
\]
Subtracting from $\langle g,x\rangle$ gives
\begin{equation}
\Delta_{B_{\min}}(x)\le \Delta_C(x)\le \Delta_{B_{\max}}(x),
\label{eq:gap_sandwich_mu_clean}
\end{equation}
where
\[
\Delta_S(x):=\langle g,x\rangle-\min_{y\in S}\langle g,y\rangle.
\]

\paragraph{Step 3: compute the two ball gaps explicitly.}
Let $B(x+R\mu,R)$ be any ball tangent at $x$ with inward normal $\mu$.
The minimizer of $\langle g,\cdot\rangle$ over this ball is
\[
x+R\mu-R\hat g,
\]
so
\[
\Delta_{B(x+R\mu,R)}(x)
=
\langle g,\;x-(x+R\mu-R\hat g)\rangle
=
R\|g\|\bigl(1-\langle \hat g,\mu\rangle\bigr).
\]
Applying this first with $R=R_{\mathrm{in}}$ and then with $R=R_{\max}$ gives
\[
\Delta_{B_{\min}}(x)=R_{\mathrm{in}}\|g\|(1-\rho),
\qquad
\Delta_{B_{\max}}(x)=R_{\max}\|g\|(1-\rho).
\]
Hence
\[
\Delta_{B_{\min}}(x)
=
\frac{R_{\mathrm{in}}}{R_{\max}}\,\Delta_{B_{\max}}(x).
\]
Using \eqref{eq:gap_sandwich_mu_clean}, we obtain
\[
\Delta_{B_{\min}}(x)
=
\frac{R_{\mathrm{in}}}{R_{\max}}\,\Delta_{B_{\max}}(x)
\ge
\frac{R_{\mathrm{in}}}{R_{\max}}\,\Delta_C(x).
\]
Since $R_{\mathrm{in}}=R(x_k)\ge R_{\min}$ and $R_{\max}=1/(2\beta_C)$,
\[
\Delta_{B_{\min}}(x)\ge
\bigl(2\beta_C R_{\min}\bigr)\,\Delta_C(x).
\]
Recalling that
\[
\Delta_{B_{\min}}(x)
=
R(x_k)\|g_k\|\bigl(1-\rho_k\bigr)
=
\Delta_{\mathrm{ball}}^{\partial}(x_k),
\]
this proves
\[
\Delta_{\mathrm{ball}}^{\partial}(x_k)\ge \kappa_0\,\Delta(x_k),
\qquad
\kappa_0:=2\beta_C R_{\min}.
\]

Finally, since $B_{\min}\subset B_{\max}$ and both balls are tangent at $x$ with the same inward normal line, one necessarily has
\[
R(x_k)\le R_{\max}=\frac{1}{2\beta_C},
\]
hence
\[
\kappa_0=2\beta_C R_{\min}\le 1.
\]
This completes the proof.
\end{proof}

\begin{proposition}[Uniform ball--gap comparability on all of $C$]
\label{prop:BallGap_global_clean}
Assume Assumption~\ref{ass:C_clean} and that $\partial C$ is $\mathcal{C}^2$. Let
\[
D:=\diam(C),
\qquad
R_{\min}:=\inf_{x\in\partial C}R(x)>0.
\]
Define
\begin{equation}
\kappa_{\mathrm{unif}}
:=
\min\!\left\{\kappa_0,\frac{R_{\min}}{D}\right\},
\label{eq:kappa_unif_def_clean}
\end{equation}
where $\kappa_0$ is the boundary constant from Lemma~\ref{lem:BallGap_clean}. Then for every $x\in C$ and every $g\neq 0$, there exists a ball $\mathcal B_x\subset C$ containing $x$ such that
\begin{equation}
\Delta_{\mathcal B_x}(x;g)\ge \kappa_{\mathrm{unif}}\,\Delta_C(x;g).
\label{eq:BallGapGlobal_clean}
\end{equation}
\end{proposition}

\begin{proof}
We first give an overview of the proof,
\paragraph{Proof idea:}
The boundary comparison from Lemma~\ref{lem:BallGap_clean} must be extended to
interior points. If \(x\) is far from the boundary, the ball \(B(x,R_{\min})\)
already gives the comparison. If \(x\) is close to the boundary, project it to a
nearest boundary point \(\bar x\) and use the tangent ball at \(\bar x\). Along
the normal segment from \(\bar x\) to the ball center, both the ball gap and the
full gap are affine in the base point, so checking the two endpoints is enough.

\medskip

Fix \(x\in C\) and \(g\neq 0\).

\noindent\textbf{Case 1: $\operatorname{dist}(x,\partial C)\ge R_{\min}$.}
Then
\[
\mathcal B_x:=B(x,R_{\min})\subset C.
\]
The minimizer of $y\mapsto \langle g,y\rangle$ over $\mathcal B_x$ is
\[
y^*=x-R_{\min}\frac{g}{\|g\|},
\]
hence
\begin{equation}
\Delta_{\mathcal B_x}(x;g)=R_{\min}\|g\|.
\label{eq:deep_ball_gap_global_clean}
\end{equation}
Also, for every $y\in C$,
\[
\langle g,x-y\rangle\le \|g\|\,\|x-y\|\le D\|g\|,
\]
so
\begin{equation}
\Delta_C(x;g)\le D\|g\|.
\label{eq:diam_gap_global_clean}
\end{equation}
Combining \eqref{eq:deep_ball_gap_global_clean} and \eqref{eq:diam_gap_global_clean} gives
\[
\Delta_{\mathcal B_x}(x;g)\ge \frac{R_{\min}}{D}\,\Delta_C(x;g)\ge \kappa_{\mathrm{unif}}\,\Delta_C(x;g).
\]

\medskip
\noindent\textbf{Case 2: $\operatorname{dist}(x,\partial C)<R_{\min}$.}
Let
\[
\eta:=\operatorname{dist}(x,\partial C),
\]
and choose $\bar x\in\partial C$ such that
\begin{equation}
\|x-\bar x\|=\eta.
\label{eq:nearest_boundary_point_global_clean}
\end{equation}
We claim that
\begin{equation}
x=\bar x+\eta\,\mu(\bar x).
\label{eq:x_on_normal_ray_global_clean}
\end{equation}
If \(\eta=0\), then \(x=\bar x\in\partial C\), and the claim follows directly
from Lemma~\ref{lem:BallGap_clean} with the tangent ball at \(x\), since
\(\kappa_{\mathrm{unif}}\le \kappa_0\). Hence assume \(\eta>0\). Then the closed ball $\overline B(x,\eta)$ is contained in $C$: otherwise some point of $B(x,\eta)$ would lie outside $C$, and the segment joining it to $x$ would meet $\partial C$ at distance $<\eta$ from $x$, contradicting the definition of $\eta$.
Since $\bar x\in \partial \overline B(x,\eta)\cap\partial C$ and $\overline B(x,\eta)\subset C$, the tangent hyperplane to $\overline B(x,\eta)$ at $\bar x$ is a supporting hyperplane to $C$ at $\bar x$.
Because $\partial C$ is $\mathcal C^2$, the supporting hyperplane is unique, so the inward unit normal to $C$ at $\bar x$ coincides with the inward unit normal to the ball (this is also true if $C$ is $\mathcal{C}^{1,1}$):
\[
\mu(\bar x)=\frac{x-\bar x}{\|x-\bar x\|}.
\]
This proves \eqref{eq:x_on_normal_ray_global_clean}.

Now let
\begin{equation}
R:=R(\bar x),
\qquad
\mathcal B_{\bar x}:=B\bigl(\bar x+R\mu(\bar x),\,R\bigr)\subset C.
\label{eq:tangent_ball_at_xbar_global_clean}
\end{equation}
Since $\eta<R_{\min}\le R$, we have
\[
\|x-(\bar x+R\mu(\bar x))\|
=
\|(\bar x+\eta\mu(\bar x))-(\bar x+R\mu(\bar x))\|
=
R-\eta
\le R,
\]
hence
\begin{equation}
x\in\mathcal B_{\bar x}.
\label{eq:x_in_tangent_ball_global_clean}
\end{equation}

For $t\in[0,R]$, define
\[
z_t:=\bar x+t\,\mu(\bar x).
\]
Since $\mathcal B_{\bar x}\subset C$ and both $\bar x$ and $\bar x+R\mu(\bar x)$ belong to $\mathcal B_{\bar x}$, convexity gives
\[
z_t\in \mathcal B_{\bar x}\subset C
\qquad\forall t\in[0,R].
\]
Now define
\begin{equation}
H(t):=\Delta_{\mathcal B_{\bar x}}(z_t;g)-\kappa_{\mathrm{unif}}\,\Delta_C(z_t;g).
\label{eq:H_def_global_clean}
\end{equation}
Both terms are affine in $t$, because
\[
\Delta_{\mathcal B_{\bar x}}(z_t;g)
=
\langle g,z_t\rangle-\min_{y\in\mathcal B_{\bar x}}\langle g,y\rangle,
\qquad
\Delta_C(z_t;g)
=
\langle g,z_t\rangle-\min_{y\in C}\langle g,y\rangle,
\]
and the minima are independent of $t$.
Hence $H$ is affine on $[0,R]$.

At $t=0$, we have $z_0=\bar x\in\partial C$, so Lemma~\ref{lem:BallGap_clean} yields
\[
\Delta_{\mathcal B_{\bar x}}(\bar x;g)\ge \kappa_0\,\Delta_C(\bar x;g).
\]
Since $\kappa_{\mathrm{unif}}\le \kappa_0$, it follows that
\begin{equation}
H(0)\ge 0.
\label{eq:H0_nonneg_global_clean}
\end{equation}

At $t=R$, the point $z_R=\bar x+R\mu(\bar x)$ is the center of the tangent ball, so
\[
\Delta_{\mathcal B_{\bar x}}(z_R;g)=R\|g\|\ge R_{\min}\|g\|.
\]
Using \eqref{eq:diam_gap_global_clean} at the point $z_R\in C$, we also have
\[
\Delta_C(z_R;g)\le D\|g\|.
\]
Therefore
\[
\Delta_{\mathcal B_{\bar x}}(z_R;g)\ge \frac{R_{\min}}{D}\,\Delta_C(z_R;g)\ge \kappa_{\mathrm{unif}}\,\Delta_C(z_R;g),
\]
so
\begin{equation}
H(R)\ge 0.
\label{eq:HR_nonneg_global_clean}
\end{equation}

Since $H$ is affine and nonnegative at both endpoints, it is nonnegative on all of $[0,R]$.
In particular, at $t=\eta$ we have $z_\eta=x$ by \eqref{eq:x_on_normal_ray_global_clean}, so
\[
\Delta_{\mathcal B_{\bar x}}(x;g)\ge \kappa_{\mathrm{unif}}\,\Delta_C(x;g).
\]
Thus \eqref{eq:BallGapGlobal_clean} holds with $\mathcal B_x:=\mathcal B_{\bar x}$ in Case~2.
\end{proof}

\paragraph{Proof of Proposition~\ref{prop:approxOracle_global_clean}} 
The general proof idea is:
first choose the comparison ball supplied by
Proposition~\ref{prop:BallGap_global_clean}. On that ball, the full ball
improvement already captures a fixed fraction of the full FW gap over \(C\).
Then use the exact ball-section formulas to show that a random section captures
a \(c_{n,d}\)-fraction of the ball improvement in expectation. The only
distinction is whether \(x_k\) is at the center, on the boundary, or strictly
inside the comparison ball. Finally, the true section oracle over \(C\) can only
improve on the oracle restricted to the comparison ball.
\begin{proof}
Fix \(x_k\in C\) and write \(g_k:=\nabla f(x_k)\). If \(g_k=0\), then the
claim is immediate. Assume \(g_k\neq0\) and set
\[
u_k:=\frac{g_k}{\|g_k\|}.
\]
By Proposition~\ref{prop:BallGap_global_clean}, there exists a ball
\[
\mathcal B_k=B(c_k,r_k)\subset C
\]
containing \(x_k\) such that
\[
\Delta_{\mathcal B_k}(x_k;g_k)
\ge
\kappa_{\mathrm{unif}}\,\Delta(x_k).
\]
Define
\[
v_k:=\frac{c_k-x_k}{r_k},
\qquad
t_k:=\|v_k\|\in[0,1].
\]
We first prove that the expected section improvement on this comparison ball is
at least a \(c_{n,d}\)-fraction of the full ball improvement:
\begin{equation}
\mathbb E\!\left[
\Delta_{P_k,\mathcal B_k}(x_k;g_k)
\mid x_k
\right]
\ge
c_{n,d}\,\Delta_{\mathcal B_k}(x_k;g_k),
\label{eq:expected_ball_section_fraction}
\end{equation}
where we choose, by Proposition~\ref{prop:EGamma_identity}, 
$c_{n,d}=
\frac dn-\frac{4(n-d)}{n(n-1)},$
and where \(\Delta_{P_k,\mathcal B_k}(x_k;g_k)\) denotes the improvement obtained by
minimizing the linear form over
\(\mathcal B_k\cap(x_k+\range(P_k^\top))\).

We distinguish three cases.

\smallskip
\noindent\textbf{Case 1: \(t_k=1\).}
Then \(x_k\) lies on the boundary of \(\mathcal B_k\). Write
\[
\mu_k:=v_k\in S^{n-1},
\qquad
\rho_k:=\langle u_k,\mu_k\rangle.
\]
If \(1-\rho_k=0\), then
\(\Delta_{\mathcal B_k}(x_k;g_k)=0\), and
\eqref{eq:expected_ball_section_fraction} is trivial. Otherwise, by
Lemma~\ref{lem:ballSectionExact_clean},
\[
\Delta_{P_k,\mathcal B_k}(x_k;g_k)
=
r_k\|g_k\|(1-\rho_k)\Gamma(P_k;u_k,\mu_k),
\]
while the full ball improvement is
\[
\Delta_{\mathcal B_k}(x_k;g_k)
=
r_k\|g_k\|(1-\rho_k).
\]
Taking expectation and using Proposition~\ref{prop:EGamma_identity} gives
\[
\mathbb E[
\Delta_{P_k,\mathcal B_k}(x_k;g_k)
\mid x_k]
\ge
c_{n,d}\,\Delta_{\mathcal B_k}(x_k;g_k).
\]

\smallskip
\noindent\textbf{Case 2: \(t_k=0\).}
Then \(x_k=c_k\). By Lemma~\ref{lem:ballSectionInterior_clean},
\[
\Delta_{P_k,\mathcal B_k}(x_k;g_k)
=
r_k\|g_k\|\|P_ku_k\|,
\]
and
\[
\Delta_{\mathcal B_k}(x_k;g_k)
=
r_k\|g_k\|.
\]
Since \(0\le \|P_ku_k\|\le1\), we have
\[
\mathbb E\|P_ku_k\|
\ge
\mathbb E\|P_ku_k\|^2
=
\frac dn
\ge
c_{n,d}.
\]
Therefore \eqref{eq:expected_ball_section_fraction} holds in this case.

\smallskip
\noindent\textbf{Case 3: \(0<t_k<1\).}
Set
\[
\mu_k:=\frac{v_k}{t_k}\in S^{n-1},
\qquad
\rho_k:=\langle u_k,\mu_k\rangle.
\]
By Lemma~\ref{lem:ballSectionInterior_clean}, the ratio between the section
ball improvement and the full ball improvement is
\[
\gamma(P_k;u_k,t_k\mu_k).
\]
The proof of Lemma~\ref{lem:gamma_scaled_unit} gives the stronger bound
\[
\gamma(P_k;u_k,t_k\mu_k)
\ge
\frac{(1-t_k)\|P_ku_k\|
+t_k(1-\rho_k)\Gamma(P_k;u_k,\mu_k)}
{(1-t_k)+t_k(1-\rho_k)}.
\]
The denominator is positive because \(t_k<1\). Taking expectation and using
\[
\mathbb E\|P_ku_k\|
\ge
\mathbb E\|P_ku_k\|^2
=
\frac dn
\ge
c_{n,d},
\]
together with Proposition~\ref{prop:EGamma_identity},
\[
\mathbb E\Gamma(P_k;u_k,\mu_k)\ge c_{n,d},
\]
yields
\[
\mathbb E[
\gamma(P_k;u_k,t_k\mu_k)
\mid x_k]
\ge
c_{n,d}.
\]
Hence \eqref{eq:expected_ball_section_fraction} also holds in the off-center
interior case.

Combining the three cases, \eqref{eq:expected_ball_section_fraction} holds for
every comparison ball produced by Proposition~\ref{prop:BallGap_global_clean}.
Finally, by Lemma~\ref{lem:subvsball}, the RSFW section oracle over \(C\)
dominates the oracle restricted to the comparison ball, so
\[
-\langle g_k,d_k\rangle
\ge
\Delta_{P_k,\mathcal B_k}(x_k;g_k).
\]
Taking conditional expectations gives
\[
\mathbb E[-\langle g_k,d_k\rangle\mid x_k]
\ge
c_{n,d}\,\Delta_{\mathcal B_k}(x_k;g_k)
\ge
c_{n,d}\kappa_{\mathrm{unif}}\Delta(x_k).
\]
Equivalently,
\[
\mathbb E[\langle g_k,d_k\rangle\mid x_k]
\le
-\beta_0^{\mathrm{unif}}\Delta(x_k),
\qquad
\beta_0^{\mathrm{unif}}:=\kappa_{\mathrm{unif}}c_{n,d}.
\]
This proves \eqref{eq:approxOracle_global_clean}.
\end{proof}

\begin{proof}[Proof of Theorem~\ref{thm:convergence_clean}]
Let $F_k:=f(x_k)-f^*\ge 0$.
If $g_k=0$, then $\Delta(x_k)=0$, hence $x_k$ is optimal and the claim is trivial from that point onward.
Thus we may assume $g_k\neq 0$ whenever the approximate-oracle inequality is invoked.

By Lemma~\ref{lem:curvature_clean},
\[
F_{k+1}\le F_k+\alpha_k\langle g_k,d_k\rangle+\frac{L}{2}\alpha_k^2\|d_k\|^2.
\]
Since $C$ is compact, $\|d_k\|\le D:=\operatorname{diam}(C)$.
Taking conditional expectation given $x_k$, using \eqref{eq:approxOracle_global_clean}, and then using $F_k\le \Delta(x_k)$, we obtain
\[
\mathbb E[F_{k+1}\mid x_k]\le (1-\alpha_k\beta_0)F_k+\frac{L}{2}\alpha_k^2D^2.
\]
Taking total expectation yields
\[
a_{k+1}\le (1-\alpha_k\beta_0)a_k+\frac{L}{2}\alpha_k^2D^2,
\qquad a_k:=\mathbb E[F_k].
\]
Since
\[
\alpha_k\beta_0=\frac{2}{k+2/\beta_0},
\]
the recursion becomes
\[
a_{k+1}
\le
\left(1-\frac{2}{k+2/\beta_0}\right)a_k
+
\frac{2LD^2/\beta_0^2}{(k+2/\beta_0)^2}.
\]
We prove by induction that
\[
a_k\le \frac{M}{k+2/\beta_0},
\qquad
M:=\max\!\left\{\frac{2}{\beta_0}a_0,\ \frac{2LD^2}{\beta_0^2}\right\}.
\]
The base case $k=0$ is immediate from the definition of $M$.
Assume it holds at step $k$. Then
\begin{align*}
a_{k+1}
&\le
\left(1-\frac{2}{k+2/\beta_0}\right)\frac{M}{k+2/\beta_0}
+
\frac{2LD^2/\beta_0^2}{(k+2/\beta_0)^2}\\
&=
\frac{M(k+2/\beta_0-2)}{(k+2/\beta_0)^2}
+
\frac{2LD^2/\beta_0^2}{(k+2/\beta_0)^2}.
\end{align*}
Since $M\ge 2LD^2/\beta_0^2$, we obtain
\[
a_{k+1}
\le
\frac{M(k+2/\beta_0-1)}{(k+2/\beta_0)^2}.
\]
It remains to note that
\[
\frac{k+2/\beta_0-1}{(k+2/\beta_0)^2}
\le
\frac{1}{k+1+2/\beta_0},
\]
because
\[
(k+2/\beta_0-1)(k+1+2/\beta_0)\le (k+2/\beta_0)^2
\]
is equivalent to $k+2/\beta_0-1\ge 0$, which holds for all $k\ge 0$ since $2/\beta_0\ge 1$ whenever $\beta_0\in(0,1]$.
Thus
\[
a_{k+1}\le \frac{M}{k+1+2/\beta_0}.
\]
This closes the induction.
\end{proof}

\subsection{High-probability section efficiency and short-step consequences}
\label{app:hp-gamma-short-step}

This subsection collects the high-probability geometric estimates used in the
short-step analysis. The expected identity for \(\Gamma(P;u,v)\) is sufficient
for the expected open-loop approximate-oracle argument. The short-step analysis
requires more: the one-step decrease contains both the visible gradient length
and the section efficiency. This leads to mixed quantities such as
\[
\mathbb E\bigl[\|Pu\|\,\Gamma(P;u,v)\bigr],
\]
or more generally \(\mathbb E[R^a\Gamma(P;u,v)]\) with
\(R=\|Pu\|^2\). Since these factors need not be positively correlated, we use a
high-probability lower bound for \(\Gamma(P;u,v)\), combined with a one-vector
JL lower bound for \(\|Pu\|\). The same high-probability control is also used
in the fixed-confidence and almost-sure open-loop analyses.

\subsubsection{High-probability lower bound for \(\Gamma(P;u,v)\)}

\begin{proof}[Proof of Proposition~\ref{prop:GammaHP}]
We split into two regimes.

\paragraph{Regime B: separated vectors, \(\rho\le \rho_0\).}
Equivalently, \(\delta=1-\rho\ge\delta_0\). Define
\begin{align}
\mathcal E_B:=\Big\{&
\left|\|Pu\|^2-\tfrac{d}{n}\right|\le \varepsilon\tfrac{d}{n},\ 
\left|\|Pv\|^2-\tfrac{d}{n}\right|\le \varepsilon\tfrac{d}{n}, \nonumber\\
&
\left|\|P(u+v)\|^2-\tfrac{d}{n}\|u+v\|^2\right|
\le \varepsilon\tfrac{d}{n}\|u+v\|^2
\Big\}.
\label{eq:EB_hp}
\end{align}
Since
\[
\mathcal Z_B:=\left\{u,\ v,\ \frac{u+v}{\|u+v\|}\right\}\subset S^{n-1}
\]
has cardinality \(3\), Lemma~\ref{lem:JL_clean} gives
\begin{equation}
\mathbb P(\mathcal E_B)
\ge
1-3C_{\mathrm{JL}}e^{-c_{\mathrm{JL}}\varepsilon^2d}.
\label{eq:PB_clean}
\end{equation}

On \(\mathcal E_B\),
\[
\|Pu\|^2\ge (1-\varepsilon)\frac dn,
\qquad
\|Pv\|^2\ge (1-\varepsilon)\frac dn,
\]
and therefore
\begin{equation}
\|Pu\|\,\|Pv\|
\ge
(1-\varepsilon)\frac dn.
\label{eq:prodNorm_lb_B_clean}
\end{equation}
Also,
\begin{equation}
\|P(u+v)\|^2
=
\|Pu\|^2+\|Pv\|^2+2\langle Pu,Pv\rangle.
\label{eq:uplusv_id_clean}
\end{equation}
Since \(\|u+v\|^2=2(1+\rho)\), the event \(\mathcal E_B\) implies
\[
\|P(u+v)\|^2
\le
(1+\varepsilon)\frac dn\,2(1+\rho).
\]
Using \eqref{eq:uplusv_id_clean} together with the lower bounds on
\(\|Pu\|^2\) and \(\|Pv\|^2\), we get
\begin{align}
2\langle Pu,Pv\rangle
&\le
(1+\varepsilon)\frac dn\,2(1+\rho)
-
2(1-\varepsilon)\frac dn \nonumber\\
&=
2\frac dn\bigl(\rho+\varepsilon(2+\rho)\bigr).
\label{eq:inner_ub_pre_B_clean}
\end{align}
Thus
\begin{equation}
\langle Pu,Pv\rangle
\le
\frac dn\bigl(\rho+\varepsilon(2+\rho)\bigr)
\le
\frac dn(\rho+3\varepsilon),
\label{eq:inner_ub_B_clean}
\end{equation}
because \(\rho\le1\). Combining \eqref{eq:prodNorm_lb_B_clean} and
\eqref{eq:inner_ub_B_clean},
\begin{align}
Z(P;u,v)
&=
\|Pu\|\,\|Pv\|-\langle Pu,Pv\rangle \nonumber\\
&\ge
\frac dn\bigl((1-\varepsilon)-(\rho+3\varepsilon)\bigr)
=
\frac dn(\delta-4\varepsilon).
\label{eq:Zlb_B_clean}
\end{align}
Since \(\delta\ge\delta_0\) and
\(\varepsilon\le\delta_0/8\le\delta/8\),
\[
\delta-4\varepsilon\ge \frac{\delta}{2}.
\]
Hence on \(\mathcal E_B\),
\begin{equation}
Z(P;u,v)\ge \frac{d}{2n}\delta,
\qquad
\Gamma(P;u,v)=\frac{Z(P;u,v)}{\delta}\ge \frac{d}{2n}.
\label{eq:regimeB_hp_clean}
\end{equation}

\paragraph{Regime A: nearly aligned vectors, \(\rho\ge \rho_0\).}
Equivalently, \(\delta\le\delta_0\). Define
\[
e_1:=u,
\qquad
e_2:=\frac{v-\rho u}{\sqrt{1-\rho^2}}.
\]
Then \((e_1,e_2)\) is an orthonormal pair and
\[
v=\rho e_1+\sqrt{1-\rho^2}\,e_2.
\]
Let
\[
\mathcal Z_A
:=
\left\{
e_1,\ e_2,\ \frac{e_1+e_2}{\sqrt2},\
\frac{e_1-e_2}{\sqrt2}
\right\}\subset S^{n-1},
\]
and let \(\mathcal E_A\) be the event that the JL estimate from
Lemma~\ref{lem:JL_clean} holds simultaneously for all points of
\(\mathcal Z_A\) with parameter \(\varepsilon\). Since
\(|\mathcal Z_A|=4\), Lemma~\ref{lem:JL_clean} gives
\begin{equation}
\mathbb P(\mathcal E_A)
\ge
1-4C_{\mathrm{JL}}e^{-c_{\mathrm{JL}}\varepsilon^2d}.
\label{eq:PA_clean}
\end{equation}

We prove that on \(\mathcal E_A\),
\begin{equation}
\Gamma(P;u,v)\ge (1-2\varepsilon)\frac dn.
\label{eq:GammaA_clean}
\end{equation}
On \(\mathcal E_A\),
\begin{equation}
\left|\|Pe_1\|^2-\frac dn\right|\le \varepsilon\frac dn,
\qquad
\left|\|Pe_2\|^2-\frac dn\right|\le \varepsilon\frac dn.
\label{eq:norms_zpm_clean}
\end{equation}
Moreover,
\begin{equation}
\left\|P\frac{e_1+e_2}{\sqrt2}\right\|^2
=
\frac{\|Pe_1\|^2+\|Pe_2\|^2+2\langle Pe_1,Pe_2\rangle}{2},
\label{eq:pm_plus_clean}
\end{equation}
and
\begin{equation}
\left\|P\frac{e_1-e_2}{\sqrt2}\right\|^2
=
\frac{\|Pe_1\|^2+\|Pe_2\|^2-2\langle Pe_1,Pe_2\rangle}{2}.
\label{eq:pm_minus_clean}
\end{equation}
Subtracting \eqref{eq:pm_minus_clean} from \eqref{eq:pm_plus_clean} yields
\[
2\langle Pe_1,Pe_2\rangle
=
\left\|P\frac{e_1+e_2}{\sqrt2}\right\|^2
-
\left\|P\frac{e_1-e_2}{\sqrt2}\right\|^2.
\]
Both terms on the right are within \(\varepsilon d/n\) of \(d/n\), so
\begin{equation}
|\langle Pe_1,Pe_2\rangle|
\le
\varepsilon\frac dn.
\label{eq:inner_zpm_bound_clean}
\end{equation}

For any \(x=ae_1+be_2\in\mathrm{span}\{u,v\}\), using
\eqref{eq:norms_zpm_clean} and \eqref{eq:inner_zpm_bound_clean},
\begin{align}
\|Px\|^2
&=
a^2\|Pe_1\|^2+b^2\|Pe_2\|^2
+2ab\langle Pe_1,Pe_2\rangle \nonumber\\
&\ge
(1-\varepsilon)\frac dn(a^2+b^2)
-
2|ab|\,\varepsilon\frac dn.
\label{eq:plane_pre_clean}
\end{align}
Since \(2|ab|\le a^2+b^2\),
\begin{equation}
\|Px\|^2
\ge
(1-2\varepsilon)\frac dn\|x\|^2
\qquad
\forall x\in\mathrm{span}\{u,v\}.
\label{eq:planeLower_clean}
\end{equation}

Because \(v\in\mathrm{span}\{u,v\}\) and \(\|v\|=1\),
\eqref{eq:planeLower_clean} gives
\[
\|Pv\|^2\ge (1-2\varepsilon)\frac dn>0.
\]
Thus
\[
\lambda:=\frac{\|Pu\|}{\|Pv\|}
\]
is well-defined and positive. We compute
\begin{align}
\|Pu-\lambda Pv\|^2
&=
\|Pu\|^2+\lambda^2\|Pv\|^2
-2\lambda\langle Pu,Pv\rangle \nonumber\\
&=
2\|Pu\|^2-2\lambda\langle Pu,Pv\rangle \nonumber\\
&=
2\lambda\bigl(\|Pu\|\,\|Pv\|-\langle Pu,Pv\rangle\bigr).
\label{eq:scaled_residual_expand_clean}
\end{align}
Therefore
\begin{equation}
Z(P;u,v)
=
\frac{1}{2\lambda}\|P(u-\lambda v)\|^2.
\label{eq:Zscaled_clean}
\end{equation}
Since \(u-\lambda v\in\mathrm{span}\{u,v\}\), the plane lower bound
\eqref{eq:planeLower_clean} gives
\[
\|P(u-\lambda v)\|^2
\ge
(1-2\varepsilon)\frac dn\|u-\lambda v\|^2.
\]
Using \(\rho=1-\delta\),
\begin{align}
\|u-\lambda v\|^2
&=
1+\lambda^2-2\lambda\rho \nonumber\\
&=
1+\lambda^2-2\lambda(1-\delta)
=
(\lambda-1)^2+2\lambda\delta
\ge
2\lambda\delta.
\label{eq:geom2_clean}
\end{align}
Substituting into \eqref{eq:Zscaled_clean},
\begin{equation}
Z(P;u,v)
\ge
(1-2\varepsilon)\frac dn\,\delta.
\label{eq:ZlbAfinal_clean}
\end{equation}
Dividing by \(\delta>0\) proves \eqref{eq:GammaA_clean}. Since
\(\varepsilon<1/8\), \(1-2\varepsilon\ge 3/4\), and therefore
\begin{equation}
\Gamma(P;u,v)\ge \frac{d}{2n},
\qquad
Z(P;u,v)\ge \frac{d}{2n}\delta.
\label{eq:regimeA_final_clean}
\end{equation}

\paragraph{Conclusion.}
Set
\[
C_{\Gamma}:=4C_{\mathrm{JL}},
\qquad
c_{\Gamma}:=c_{\mathrm{JL}}.
\]
In Regime B, \eqref{eq:PB_clean} and \eqref{eq:regimeB_hp_clean} give
\[
\mathbb P\!\left(
Z(P;u,v)\ge \frac{d}{2n}\delta
\right)
\ge
1-C_{\Gamma}e^{-c_{\Gamma}\varepsilon^2d}.
\]
In Regime A, \eqref{eq:PA_clean} and \eqref{eq:regimeA_final_clean} give the
same lower bound for \(Z(P;u,v)\). Since \(\delta>0\), dividing by
\(\delta\) proves \eqref{eq:Gammahp_unif}.
\end{proof}

\begin{remark}
\label{rem:hp_oracle_global}
The bound above is pointwise in \((u,v)\), which is exactly what is needed in
the one-step analysis: after conditioning on the past, the relevant directions
are deterministic and \(P_k\) is a fresh Haar--Stiefel draw. For boundary
comparison balls, the one-step lower bound follows directly from
Proposition~\ref{prop:GammaHP}. For interior comparison balls, the exact
improvement is governed by the ratio \(\gamma\), and the additional deterministic
comparison is Lemma~\ref{lem:gamma_scaled_unit}. The proof below combines the
high-probability lower bound on \(\Gamma(P;u,v)\) with a one-vector JL lower
bound on \(\|Pu\|\).
\end{remark}

\begin{proposition}[Uniform high-probability approximate oracle]
\label{prop:hp_oracle_uniform}
Assume Assumptions~\ref{ass:C_clean}--\ref{ass:f_clean}, assume
\(\partial C\) is \(\mathcal C^2\), and let \(x_k\in C\) with
\(g_k:=\nabla f(x_k)\neq0\). Let
\[
u_k:=\frac{g_k}{\|g_k\|}.
\]
Let \(\mathcal B_k\subset C\) be a comparison ball containing \(x_k\) such that
\[
\Delta_{\mathcal B_k}(x_k;g_k)
\ge
\kappa_{\mathrm{unif}}\Delta(x_k),
\]
as in Proposition~\ref{prop:BallGap_global_clean}. Define
\[
c_{\mathrm{hp}}:=\min\{c_{\Gamma},c_{\mathrm{JL}}\},
\qquad
C_{\mathrm{hp}}:=C_{\Gamma}+C_{\mathrm{JL}},
\qquad
p_{\mathrm{hp}}
:=
\left[
1-C_{\mathrm{hp}}e^{-c_{\mathrm{hp}}\varepsilon^2d}
\right]_+,
\]
and
\[
\beta_{\mathrm{hp}}
:=
\kappa_{\mathrm{unif}}\frac{d}{2n}.
\]
Then, conditionally on
\[
\mathcal F_k:=\sigma(P_0,\ldots,P_{k-1}),
\]
there exists an event \(\mathcal E_k\), depending only on the fresh draw
\(P_k\), such that
\begin{equation}
\mathbb P(\mathcal E_k\mid\mathcal F_k)
\ge
p_{\mathrm{hp}}.
\label{eq:hp_oracle_prob_uniform}
\end{equation}
On \(\mathcal E_k\), the RSFW direction \(d_k=s_k-x_k\) satisfies
\begin{equation}
\langle g_k,d_k\rangle
\le
-\beta_{\mathrm{hp}}\Delta(x_k),
\label{eq:hp_oracle_uniform}
\end{equation}
and the projected gradient obeys
\begin{equation}
\|P_k g_k\|
\ge
\sqrt{(1-\varepsilon)\frac dn}\,\|g_k\|.
\label{eq:hp_projected_gradient}
\end{equation}
\end{proposition}

\begin{proof}
\textbf{Proof idea:}
The high-probability oracle event is the intersection of two simple events: a
one-vector JL lower bound for \(\|P_ku_k\|\), and, when the comparison ball is
not centered, the high-probability lower bound for
\(\Gamma(P_k;u_k,\mu_k)\). The deterministic interior-ball inequality
Lemma~\ref{lem:gamma_scaled_unit} then converts these two events into a uniform
lower bound on the section improvement.

Condition on \(\mathcal F_k\). Then \(x_k\), \(g_k\), and \(u_k\) are fixed,
while \(P_k\) is an independent Haar--Stiefel draw. Let \(c_k,r_k\) be the
center and radius of \(\mathcal B_k\), and define
\[
v_k:=\frac{c_k-x_k}{r_k},
\qquad
t_k:=\|v_k\|\in[0,1].
\]

First define the one-vector projection event
\[
\mathcal E_{U,k}
:=
\left\{
\|P_ku_k\|^2
\ge
(1-\varepsilon)\frac dn
\right\}.
\]
By Lemma~\ref{lem:JL_clean},
\[
\mathbb P(\mathcal E_{U,k}\mid\mathcal F_k)
\ge
1-C_{\mathrm{JL}}e^{-c_{\mathrm{JL}}\varepsilon^2d}.
\]
If \(t_k=0\), set \(\mathcal E_k:=\mathcal E_{U,k}\). If \(t_k>0\), set
\[
\mu_k:=\frac{v_k}{t_k},
\qquad
\mathcal E_{\Gamma,k}
:=
\left\{
\Gamma(P_k;u_k,\mu_k)\ge\frac d{2n}
\right\},
\qquad
\mathcal E_k:=\mathcal E_{U,k}\cap\mathcal E_{\Gamma,k}.
\]
By Proposition~\ref{prop:GammaHP}, conditionally on \(\mathcal F_k\),
\[
\mathbb P(\mathcal E_{\Gamma,k}\mid\mathcal F_k)
\ge
1-C_{\Gamma}e^{-c_{\Gamma}\varepsilon^2d}
\qquad (t_k>0).
\]
The union bound gives
\[
\mathbb P(\mathcal E_k\mid\mathcal F_k)
\ge
1-(C_{\Gamma}+C_{\mathrm{JL}})
e^{-\min\{c_{\Gamma},c_{\mathrm{JL}}\}\varepsilon^2d}
\ge
p_{\mathrm{hp}},
\]
and in the case \(t_k=0\) the event \(\mathcal E_k=\mathcal E_{U,k}\) has even
larger probability.

It remains to prove the oracle inequality. On \(\mathcal E_k\), the deterministic
ball-section formulas from Lemmas~\ref{lem:ballSectionExact_clean},
\ref{lem:ballSectionInterior_clean}, and \ref{lem:gamma_scaled_unit} imply
\begin{equation}
\Delta_{P_k,\mathcal B_k}(x_k;g_k)
\ge
\frac d{2n}\,
\Delta_{\mathcal B_k}(x_k;g_k),
\label{eq:hp_ball_section_improvement}
\end{equation}
where \(\Delta_{P_k,\mathcal B_k}(x_k;g_k)\) denotes the FW improvement obtained
by restricting the comparison ball \(\mathcal B_k\) to the affine section
\(x_k+\range(P_k^\top)\). Indeed, the boundary case \(t_k=1\) uses the
\(\Gamma\)-event, the centered case \(t_k=0\) uses the projection event, and the
off-center interior case \(t_k\in(0,1)\) uses
\[
\gamma(P_k;u_k,v_k)
\ge
\min\{\|P_ku_k\|,\Gamma(P_k;u_k,\mu_k)\}.
\]
In all cases the lower bound is at least \(d/(2n)\), since
\[
\sqrt{(1-\varepsilon)\frac dn}\ge \frac d{2n}
\]
for \(d\le n\) and \(\varepsilon<1/2\).

By Lemma~\ref{lem:subvsball}, the RSFW section gap over \(C\) is at least the
section improvement over the comparison ball. Therefore, using the choice of
\(\mathcal B_k\),
\[
-\langle g_k,d_k\rangle
=
\langle g_k,x_k-s_k\rangle
\ge
\frac d{2n}\Delta_{\mathcal B_k}(x_k;g_k)
\ge
\kappa_{\mathrm{unif}}\frac d{2n}\Delta(x_k)
=
\beta_{\mathrm{hp}}\Delta(x_k).
\]
This proves \eqref{eq:hp_oracle_uniform}. Finally,
\eqref{eq:hp_projected_gradient} follows directly from \(\mathcal E_{U,k}\).
\end{proof}

\subsubsection{Mixed moments from high-probability section efficiency}

\begin{proposition}[Mixed-moment lower bound for \(R^a\Gamma(P;u,v)\)]
\label{prop:pos_corr_general}
Fix \(a\in(0,1]\), let \(u,v\in S^{n-1}\) with
\(\rho:=\langle u,v\rangle <1\), and let
\[
R:=\|Pu\|^2,
\]
where \(P\in\mathbb R^{d\times n}\) is Haar--Stiefel. Fix also
\[
0<\varepsilon<\min\!\left\{\frac18,\frac{\delta_0}{8}\right\}
\]
and define
\[
c_{\mathrm{hp}}:=\min\{c_{\Gamma},c_{\mathrm{JL}}\},
\qquad
C_{\mathrm{hp}}:=C_{\Gamma}+C_{\mathrm{JL}},
\qquad
p_{\mathrm{hp}}
:=
\left[
1-C_{\mathrm{hp}}e^{-c_{\mathrm{hp}}\varepsilon^2d}
\right]_+,
\]
where \([t]_+:=\max\{t,0\}\). Then
\begin{equation}
\mathbb E\bigl[R^a\Gamma(P;u,v)\bigr]
\ge
p_{\mathrm{hp}}(1-\varepsilon)^a
\left(\frac dn\right)^a
\frac{d}{2n}.
\label{eq:mixed_direct_lower}
\end{equation}
Consequently,
\begin{equation}
\mathbb E\bigl[R^a\Gamma(P;u,v)\bigr]
\ge
\vartheta_a\,\mathbb E[R^a]\,\mathbb E[\Gamma(P;u,v)],
\qquad
\vartheta_a:=\frac12(1-\varepsilon)^a p_{\mathrm{hp}}.
\label{eq:mixed_vs_product}
\end{equation}
In particular, for \(a=\tfrac12\),
\begin{equation}
\mathbb E\bigl[\|Pu\|\Gamma(P;u,v)\bigr]
\ge
\vartheta_{1/2}\,
\mathbb E\bigl[\|Pu\|\bigr]\,
\mathbb E\bigl[\Gamma(P;u,v)\bigr].
\label{eq:pos_corr_Gamma}
\end{equation}
\end{proposition}

\begin{proof}
Here the general idea is that
short-step decrease contains a product of two random quantities: the projected
gradient length and the section-efficiency ratio. Instead of trying to prove
positive correlation, we lower bound both quantities on a common
high-probability event. This gives a direct mixed-moment lower bound, which is
then compared to \(\mathbb E[R^a]\mathbb E[\Gamma]\) using the simple upper
bounds \(\mathbb E[R^a]\le(d/n)^a\) and \(\mathbb E[\Gamma]\le d/n\).

Define
\[
\mathcal E_U
:=
\left\{
\|Pu\|^2\ge (1-\varepsilon)\frac dn
\right\},
\qquad
\mathcal E_\Gamma
:=
\left\{
\Gamma(P;u,v)\ge \frac{d}{2n}
\right\}.
\]
By Lemma~\ref{lem:JL_clean}, applied to the singleton set \(\{u\}\),
\[
\mathbb P(\mathcal E_U)
\ge
1-C_{\mathrm{JL}}e^{-c_{\mathrm{JL}}\varepsilon^2d}.
\]
By Proposition~\ref{prop:GammaHP},
\[
\mathbb P(\mathcal E_\Gamma)
\ge
1-C_{\Gamma}e^{-c_{\Gamma}\varepsilon^2d}.
\]
The union bound gives
\[
\mathbb P(\mathcal E_U\cap\mathcal E_\Gamma)
\ge
1-
(C_{\Gamma}+C_{\mathrm{JL}})
e^{-\min\{c_{\Gamma},c_{\mathrm{JL}}\}\varepsilon^2d}
\ge
p_{\mathrm{hp}}.
\]
On \(\mathcal E_U\cap\mathcal E_\Gamma\),
\[
R^a\Gamma(P;u,v)
\ge
\left((1-\varepsilon)\frac dn\right)^a
\frac{d}{2n}.
\]
Taking expectations gives \eqref{eq:mixed_direct_lower}.

For the product comparison, Jensen's inequality gives
\[
\mathbb E[R^a]
\le
(\mathbb E R)^a
=
\left(\frac dn\right)^a,
\]
and Proposition~\ref{prop:EGamma_identity} gives
\[
\mathbb E[\Gamma(P;u,v)]\le \frac dn.
\]
Combining these two upper bounds with \eqref{eq:mixed_direct_lower} yields
\eqref{eq:mixed_vs_product}. Taking \(a=1/2\) gives
\eqref{eq:pos_corr_Gamma}.
\end{proof}

\subsubsection{Strong convexity of random sections}

\begin{lemma}[Subspace sections inherit strong convexity]
\label{lem:subspace_SC}
If \(C\) is \(\beta_C\)-strongly convex and \(U\subset\mathbb R^n\) is a linear
subspace, then for every \(x\in C\) the section
\[
C_U:=C\cap(x+U)
\]
is \(\beta_C\)-strongly convex in \(U\) with the same parameter \(\beta_C\).
\end{lemma}

\begin{proof}
Take \(y,z\in C_U\) and \(\lambda\in[0,1]\). By strong convexity of \(C\),
\[
B(\lambda y+(1-\lambda)z,\,
\beta_C\lambda(1-\lambda)\|y-z\|^2)
\subset C.
\]
Intersecting with the affine space \(x+U\) gives the corresponding ball in the
relative geometry of \(x+U\), hence it is contained in \(C_U\).
\end{proof}

\begin{corollary}[Scaling inequality in the subspace]
\label{cor:ScalingSCB_subspace}
Let \(C\) be \(\beta_C\)-strongly convex and let \(x_k\in C\) with
\(\Pi_k g_k\neq0\). Then the subspace FW direction
\(d_k=s_k-x_k\in U_k\) satisfies
\begin{equation}
\frac{\langle g_k,x_k-s_k\rangle}{\|d_k\|^2}
\ge
\frac{\beta_C}{4}\,\|\Pi_k g_k\|,
\label{eq:ScalingSCB_sub}
\end{equation}
and consequently
\begin{equation}
\|d_k\|\le \frac4\beta_C.
\label{eq:step_bound}
\end{equation}
\end{corollary}

\begin{proof}
Let
\[
E_k:=C\cap(x_k+U_k),
\qquad
h_k^{\mathrm{sub}}:=\Pi_k g_k.
\]
By Lemma~\ref{lem:subspace_SC}, \(E_k\) is \(\beta_C\)-strongly convex in
\(x_k+U_k\). Since \(s_k\) minimizes \(y\mapsto\langle g_k,y\rangle\) over
\(E_k\), and \(x_k-s_k\in U_k\),
\[
\langle g_k,x_k-s_k\rangle
=
\langle h_k^{\mathrm{sub}},x_k-s_k\rangle.
\]
Fix \(\lambda\in(0,1)\) and set
\[
\widehat h_k:=\frac{h_k^{\mathrm{sub}}}{\|h_k^{\mathrm{sub}}\|}.
\]
By strong convexity of \(E_k\), the point
\[
z_\lambda
:=
\lambda x_k+(1-\lambda)s_k
-
\beta_C\lambda(1-\lambda)\|x_k-s_k\|^2\,\widehat h_k
\]
belongs to \(E_k\). Since \(s_k\) minimizes the linear form over \(E_k\),
\[
\langle g_k,z_\lambda-s_k\rangle\ge0.
\]
Using
\[
z_\lambda-s_k
=
\lambda(x_k-s_k)
-
\beta_C\lambda(1-\lambda)\|d_k\|^2\widehat h_k
\]
and
\[
\langle g_k,\widehat h_k\rangle=\|h_k^{\mathrm{sub}}\|,
\]
we get
\[
0
\le
\lambda\langle g_k,x_k-s_k\rangle
-
\beta_C\lambda(1-\lambda)\|d_k\|^2
\|h_k^{\mathrm{sub}}\|.
\]
Dividing by \(\lambda>0\) and taking \(\lambda=1/2\),
\[
\langle g_k,x_k-s_k\rangle
\ge
\frac{\beta_C}{2}\|d_k\|^2\|h_k^{\mathrm{sub}}\|.
\]
This gives the stronger bound
\[
\frac{\langle g_k,x_k-s_k\rangle}{\|d_k\|^2}
\ge
\frac{\beta_C}{2}\|\Pi_k g_k\|.
\]
In particular, the weaker bound \eqref{eq:ScalingSCB_sub} holds. The same
argument and Cauchy--Schwarz give the stronger estimate
\(\|d_k\|\le2/\beta_C\), hence also \eqref{eq:step_bound}.
\end{proof}

\subsubsection{Proof of the short-step linear theorem, Theorem~\ref{thm:linear_convergence}}
\paragraph{Proof idea.}
The clipped short step gives a deterministic decrease involving
\(\mathrm{gap}_k^2/\|d_k\|^2\). Strong convexity of the section converts this
quantity into projected-gradient length times the section gap. On the good
event from Proposition~\ref{prop:hp_oracle_uniform}, the section gap controls
the full FW gap and the projected gradient is not too small. The gradient lower
bound then turns this into a fixed fractional decrease in \(h_k\), and averaging
over the good event gives the linear contraction.

\begin{proof}
Let
\[
h_k:=f(x_k)-f^*,
\qquad
d_k:=s_k-x_k,
\qquad
\mathrm{gap}_k:=\langle g_k,x_k-s_k\rangle\ge0.
\]
If \(d_k=0\), then \(\mathrm{gap}_k=0\) and the update is trivial. Otherwise,
by \(L\)-smoothness, for every \(\alpha_k\in[0,1]\),
\[
f(x_k+\alpha_kd_k)
\le
f(x_k)
-\alpha_k\,\mathrm{gap}_k
+
\frac{L}{2}\alpha_k^2\|d_k\|^2.
\]
With the clipped short-step rule
\[
\alpha_k
=
\min\!\left\{
1,\,
\frac{\mathrm{gap}_k}{L\|d_k\|^2}
\right\},
\]
we obtain
\begin{equation}
h_{k+1}
\le
h_k
-
\min\!\left\{
\frac{\mathrm{gap}_k}{2},\,
\frac{\mathrm{gap}_k^2}{2L\|d_k\|^2}
\right\}.
\label{eq:shortstep_descent_clean}
\end{equation}
By Corollary~\ref{cor:ScalingSCB_subspace},
\begin{equation}
\frac{\mathrm{gap}_k^2}{\|d_k\|^2}
\ge
\frac{\beta_C}{4}\|\Pi_k g_k\|\,\mathrm{gap}_k.
\label{eq:gap2_bound}
\end{equation}
Combining \eqref{eq:shortstep_descent_clean} and \eqref{eq:gap2_bound},
\begin{equation}
h_{k+1}
\le
h_k
-
\min\!\left\{
\frac{\mathrm{gap}_k}{2},\,
\frac{\beta_C}{8L}\|\Pi_k g_k\|\,\mathrm{gap}_k
\right\}.
\label{eq:descent_linear}
\end{equation}

By Proposition~\ref{prop:hp_oracle_uniform}, conditionally on
\(\mathcal F_k\), with probability at least \(p_{\mathrm{hp}}\),
\[
\mathrm{gap}_k\ge \beta_{\mathrm{hp}}\Delta(x_k)\ge \beta_{\mathrm{hp}}h_k
\]
and
\[
\|\Pi_k g_k\|=\|P_kg_k\|
\ge
\sqrt{(1-\varepsilon)\frac dn}\|g_k\|
\ge
c_0\sqrt{(1-\varepsilon)\frac dn}.
\]

Thus, on \(\mathcal E_k\),
\[
\min\!\left\{
\frac{\mathrm{gap}_k}{2},\,
\frac{\beta_C}{8L}\|\Pi_k g_k\|\,\mathrm{gap}_k
\right\}
\ge
\beta_{\mathrm{hp}}
\min\!\left\{
\frac12,\,
\frac{\beta_C c_0}{8L}
\sqrt{(1-\varepsilon)\frac dn}
\right\}h_k.
\]
Outside \(\mathcal E_k\), the short-step update is nonincreasing, so the
decrease term in \eqref{eq:descent_linear} is nonnegative.
Taking conditional expectations in \eqref{eq:descent_linear} yields
\begin{equation}
\mathbb E[h_{k+1}\mid\mathcal F_k]
\le
(1-r_{\mathrm{lin}})h_k,
\label{eq:linear_conditional_contraction}
\end{equation}
where
\begin{equation}
r_{\mathrm{lin}}
:=
p_{\mathrm{hp}}\beta_{\mathrm{hp}}
\min\!\left\{
\frac12,\,
\frac{\beta_C c_0}{8L}
\sqrt{(1-\varepsilon)\frac dn}
\right\}.
\label{eq:linear_rate_constant}
\end{equation}
Taking expectations gives
\[
a_{k+1}\le (1-r_{\mathrm{lin}})a_k,
\qquad
a_k:=\mathbb E[h_k].
\]
Hence
\[
a_k\le (1-r_{\mathrm{lin}})^ka_0,
\]
which proves \eqref{eq:linear_rate}.
\end{proof}

\section{High-probability and almost-sure linear convergence}
\label{app:hp-linear}

This section upgrades expectation bounds to pathwise statements. The common
theme is to view the residuals, or residuals plus deterministic correction
terms, as nonnegative supermartingales and then apply maximal or convergence
inequalities.

\subsection{Supermartingale structure}

The key observation is that the RSFW iterates produce a
\emph{non-negative supermartingale} (or a correctable one) in all
three convergence regimes.

\begin{lemma}[Ville's maximal inequality {\cite[p.~100]{Ville1939}}]
\label{lem:Ville}
Let $(S_k)_{k\ge 0}$ be a non-negative supermartingale
adapted to a filtration $(\mathcal F_k)$.
Then for every $\lambda>0$:
\[
\mathbb P\!\left(\sup_{k\ge 0}S_k\ge\lambda\right)
\;\le\;\frac{\mathbb E[S_0]}{\lambda}.
\]
\end{lemma}

\begin{proof}
This is a direct consequence of the optional stopping theorem
applied to $S_k$ with the stopping time
$\tau:=\inf\{k:S_k\ge\lambda\}$.
\end{proof}

\begin{proposition}[Supermartingale estimates for RSFW iterates]
\label{prop:supermartingale}
Let \(h_k:=f(x_k)-f^\star\) and
\(\mathcal F_k:=\sigma(P_0,\ldots,P_{k-1})\).

\begin{enumerate}
\item[\textup{(i)}]
\textbf{Short-step linear case.}
Under the conditions of Theorem~\ref{thm:linear_convergence},
\[
\mathbb E[h_{k+1}\mid\mathcal F_k]
\le
(1-r_{\mathrm{lin}})h_k.
\]
In particular, \((h_k)_{k\ge0}\) is a nonnegative supermartingale.

\item[\textup{(ii)}]
\textbf{Open-loop sublinear case.}
Under the conditions of Theorem~\ref{thm:convergence_clean}, with
\[
\alpha_k=\frac{2}{\beta_0(k+2/\beta_0)},
\]
define
\[
T_k
:=
\sum_{j=k}^{\infty}
\frac{2LD^2}{\beta_0^2(j+2/\beta_0)^2},
\qquad
V_k:=h_k+T_k.
\]
Then \((V_k)_{k\ge0}\) is a nonnegative supermartingale.
\end{enumerate}
\end{proposition}

\begin{proof}
\textbf{(i)}
This is exactly the conditional contraction estimate
\eqref{eq:linear_conditional_contraction}. Since \(h_k\ge0\), it follows that
\((h_k)_{k\ge0}\) is a nonnegative supermartingale.

\textbf{(ii)}
By the curvature inequality, with \(d_k=s_k-x_k\) and
\(\|d_k\|\le D:=\diam(C)\),
\[
h_{k+1}
\le
h_k+\alpha_k\langle g_k,d_k\rangle
+\frac{L}{2}\alpha_k^2D^2.
\]
Since the subspace oracle can choose \(x_k\), we have
\[
\langle g_k,d_k\rangle\le0.
\]
Therefore
\[
h_{k+1}
\le
h_k+\frac{L}{2}\alpha_k^2D^2.
\]
Moreover,
\[
\frac{L}{2}\alpha_k^2D^2
=
\frac{2LD^2}{\beta_0^2(k+2/\beta_0)^2}
=
T_k-T_{k+1}.
\]
Hence
\[
V_{k+1}
=
h_{k+1}+T_{k+1}
\le
h_k+T_k
=
V_k.
\]
Thus \((V_k)_{k\ge0}\) is pathwise nonincreasing, hence a nonnegative
supermartingale.
\end{proof}

\subsection{High-probability linear convergence}

\begin{theorem}[High-probability geometric convergence]
\label{thm:hp_linear_app}
Under the conditions of Theorem~\ref{thm:linear_convergence},
for any $\theta\in(0,r_{\mathrm{lin}})$ and $\delta\in(0,1)$, with probability
at least $1-\delta$:
\begin{equation}
h_k\;\le\;\frac{h_0}{\delta}\,(1-\theta)^k
\qquad\text{for all }k\ge 0\text{ simultaneously.}
\label{eq:hp_linear_app}
\end{equation}
\end{theorem}

\begin{proof}
Define $S_k:=h_k/(1-\theta)^k$.
By Proposition~\ref{prop:supermartingale}(i),
\[
\mathbb E[S_{k+1}\mid\mathcal F_k]
=\frac{\mathbb E[h_{k+1}\mid\mathcal F_k]}{(1-\theta)^{k+1}}
\le\frac{(1-r_{\mathrm{lin}})\,h_k}{(1-\theta)^{k+1}}
=\frac{1-r_{\mathrm{lin}}}{1-\theta}\,S_k
\le S_k,
\]
since $\theta<r_{\mathrm{lin}}$ implies $(1-r_{\mathrm{lin}})/(1-\theta)<1$.
Hence $S_k$ is a non-negative supermartingale with $S_0=h_0$.
By Lemma~\ref{lem:Ville} with $\lambda=h_0/\delta$:
\[
\mathbb P\!\left(\sup_{k\ge 0}\frac{h_k}{(1-\theta)^k}
\ge\frac{h_0}{\delta}\right)\le\delta.
\]
Rearranging: with probability $\ge 1-\delta$,
$h_k\le (h_0/\delta)(1-\theta)^k$ for all $k\ge 0$.
\end{proof}

\begin{corollary}[Almost-sure geometric convergence]
\label{cor:as_linear_app}
Under the conditions of Theorem~\ref{thm:linear_convergence},
for every $\theta\in(0,r_{\mathrm{lin}})$:
\[
\frac{h_k}{(1-\theta)^k}\;\xrightarrow{k\to\infty}\;0
\qquad\text{almost surely.}
\]
\end{corollary}
\begin{proof}
Fix \(\theta\in(0,r_{\mathrm{lin}})\) and set
\(S_k:=h_k/(1-\theta)^k\). As in the proof of
Theorem~\ref{thm:hp_linear_app}, \((S_k)_{k\ge0}\) is a nonnegative
supermartingale. Therefore, by Doob's supermartingale convergence theorem,
\(S_k\to S_\infty\) almost surely for some \(S_\infty\ge0\). Moreover,
\[
\mathbb E S_k
\le
\left(\frac{1-r_{\mathrm{lin}}}{1-\theta}\right)^k h_0
\to0.
\]
By Fatou's lemma,
\[
\mathbb E S_\infty\le \liminf_{k\to\infty}\mathbb E S_k=0.
\]
Hence \(S_\infty=0\) almost surely, which proves
\(h_k/(1-\theta)^k\to0\) almost surely.
\end{proof}

\begin{remark}[Cost of high probability]
\label{rem:hp_cost}
Compared to the in-expectation bound
$\mathbb E[h_k]\le(1-r_{\mathrm{lin}})^k h_0$,
the high-probability bound \eqref{eq:hp_linear_app} pays:
\begin{itemize}
\item a $1/\delta$ multiplicative factor in the constant, and
\item a rate $(1-\theta)^k$ with $\theta<r_{\mathrm{lin}}$ (arbitrarily close to $r_{\mathrm{lin}}$).
\end{itemize}
In terms of iteration complexity for $h_k\le\varepsilon$:
the in-expectation bound requires
$k\ge\log(h_0/\varepsilon)/\log(1/(1{-}r_{\mathrm{lin}}))$,
while the high-probability bound requires
$k\ge\log(h_0/(\varepsilon\delta))/\log(1/(1{-}\theta))$---an
additive $O(\log(1/\delta))$ cost.
The bound is also \emph{uniform} in $k$
(it holds for all $k$ simultaneously), which is strictly stronger
than a pointwise high-probability guarantee at a single $k$.
\end{remark}

\subsection{High-probability and almost-sure convergence for the open-loop step-size}
\label{app:high-prob}

\paragraph{Good-event products and martingale control.}
The open-loop high-probability analysis differs from the expectation proof in
one respect: the approximate-oracle decrease is only guaranteed on random
good events. Proposition~\ref{prop:hp_oracle_uniform} gives, conditionally on
the past, a good event \(\mathcal E_k\) of probability at least \(p_{\mathrm{hp}}\)
on which the section oracle achieves a fixed fraction
\(\beta_{\mathrm{hp}}\Delta(x_k)\) of the full FW gap. The resulting recursion
has random multiplicative factors. The role of the martingale below is to show
that the weighted number of good events cannot deviate too far below its
conditional mean. Once this is controlled, the open-loop products decay as in
the deterministic approximate-oracle proof, up to a random multiplicative
constant.

\begin{lemma}[Maximal Bernstein--Freedman inequality]
\label{lem:max_freedman}
Let \((M_m)_{m\ge0}\) be a martingale with \(M_0=0\) and increments
\[
\Delta M_m:=M_m-M_{m-1},
\qquad m\ge1.
\]
Assume that for some deterministic constants \(V_*>0\) and \(c_*>0\),
\[
|\Delta M_m|\le c_*
\qquad\text{a.s. for all }m\ge1,
\]
and
\[
\sum_{m=1}^{\infty}
\mathbb E\!\left[(\Delta M_m)^2\mid\mathcal F_{m-1}\right]
\le V_*
\qquad\text{a.s.}
\]
Then for every \(x>0\),
\begin{equation}
\mathbb P\!\left(
\sup_{m\ge0}|M_m|
\ge
\sqrt{2V_*x}+\frac{c_*x}{3}
\right)
\le 2e^{-x}.
\label{eq:max_freedman}
\end{equation}
\end{lemma}

\begin{proof}
For each finite horizon \(N\), the usual maximal Bernstein--Freedman inequality
applied to \((M_m)_{0\le m\le N}\) gives the one-sided bound for
\(\sup_{0\le m\le N}M_m\). Applying the same inequality to \(-M_m\) and taking
a union bound gives the two-sided estimate. Letting \(N\to\infty\) yields
\eqref{eq:max_freedman}.
\end{proof}

\subsubsection{Proof of Theorem~\ref{thm:hp_open_loop_as}}
\paragraph{Proof idea.}
The open-loop recursion has good multiplicative decrease only on random good
events. We weight the good-event indicators by the step sizes and compare their
partial sums to their conditional expectations using a martingale. Since the
step-size squares are summable, this martingale has almost surely bounded
fluctuations. Therefore the random product of factors behaves like the
deterministic product in the expectation proof, up to a finite random constant.
\begin{proof}
Recall that
\[
\bar\beta:=p_{\mathrm{hp}}\beta_{\mathrm{hp}}>0,
\qquad
\alpha_k=\frac{2}{\bar\beta(k+\tau)},
\qquad
\tau=\frac{2}{\bar\beta},
\]
so that
\begin{equation}
\bar\beta\,\alpha_k=\frac{2}{k+\tau}.
\label{eq:barbeta_alpha_identity_as}
\end{equation}
Let
\[
I_k:=\mathbf 1_{\mathcal E_k},
\qquad
\omega_k:=\beta_{\mathrm{hp}}\alpha_k,
\]
where \(\mathcal E_k\) is the good event from
Proposition~\ref{prop:hp_oracle_uniform}. We work with the filtration
\[
\mathcal F_m:=\sigma(P_0,\ldots,P_{m-1}),
\]
so that \(I_i\) is \(\mathcal F_{i+1}\)-measurable and
\(\mathbb E[I_i\mid\mathcal F_i]\ge p_{\mathrm{hp}}\).

Define
\[
M_m
:=
\sum_{i=0}^{m-1}
\omega_i\Bigl(I_i-\mathbb E[I_i\mid\mathcal F_i]\Bigr),
\qquad
m\ge0.
\]
Then \((M_m)_{m\ge0}\) is a martingale with \(M_0=0\).

Since \(x_k\in C\cap(x_k+U_k)\), the subspace oracle always satisfies
\[
\langle g_k,d_k\rangle\le0.
\]
On the event \(\mathcal E_k\), Proposition~\ref{prop:hp_oracle_uniform} gives
\[
\langle g_k,d_k\rangle
\le
-\beta_{\mathrm{hp}}\Delta(x_k).
\]
Therefore, for every outcome,
\[
\langle g_k,d_k\rangle
\le
-\beta_{\mathrm{hp}}I_k\Delta(x_k)
\le
-\beta_{\mathrm{hp}}I_kF_k,
\qquad
F_k:=f(x_k)-f^*,
\]
where the last inequality uses the convexity bound \(F_k\le\Delta(x_k)\).

Applying Lemma~\ref{lem:curvature_clean} with \(y=s_k\), and using
\(\|d_k\|\le D\), gives
\begin{equation}
F_{k+1}
\le
(1-\omega_k I_k)F_k
+
Q\alpha_k^2,
\qquad
Q:=\frac{L}{2}D^2.
\label{eq:hp_recursion_as}
\end{equation}

We first show that
\[
B_\infty:=\sup_{m\ge0}|M_m|<\infty
\qquad\text{a.s.}
\]
The martingale increments satisfy
\[
|M_{m+1}-M_m|
\le
\omega_m
\le
\omega_0
=
\beta_{\mathrm{hp}},
\]
where we use \(\alpha_0=1\). Moreover,
\[
\sum_{m=0}^{\infty}
\mathbb E\!\left[(M_{m+1}-M_m)^2\mid\mathcal F_m\right]
\le
\sum_{m=0}^{\infty}\omega_m^2
<\infty,
\]
because \(\alpha_m=O(1/m)\). Hence \(M_m\) converges almost surely and in
\(L^2\), and therefore \(B_\infty<\infty\) almost surely.

Fix a sample point for which \(B_\infty<\infty\). For integers
\(0\le j\le k-1\),
\begin{align}
\sum_{i=j}^{k-1}\omega_i I_i
&=
\sum_{i=j}^{k-1}\omega_i\,\mathbb E[I_i\mid\mathcal F_i]
+
\sum_{i=j}^{k-1}
\omega_i\Bigl(I_i-\mathbb E[I_i\mid\mathcal F_i]\Bigr)
\nonumber\\
&\ge
p_{\mathrm{hp}}\sum_{i=j}^{k-1}\omega_i
-
2B_\infty
=
\bar\beta\sum_{i=j}^{k-1}\alpha_i
-
2B_\infty.
\label{eq:weighted_sum_lower_as_1}
\end{align}
Using \eqref{eq:barbeta_alpha_identity_as},
\begin{equation}
\sum_{i=j}^{k-1}\omega_i I_i
\ge
2\sum_{i=j}^{k-1}\frac1{i+\tau}
-
2B_\infty.
\label{eq:weighted_sum_lower_as_3}
\end{equation}

Since \(0\le \omega_iI_i\le1\), we may use \(1-z\le e^{-z}\) to obtain
\begin{align}
\prod_{i=j}^{k-1}(1-\omega_iI_i)
&\le
\exp\!\left(-\sum_{i=j}^{k-1}\omega_iI_i\right)
\nonumber\\
&\le
\exp\!\left(
2B_\infty
-
2\sum_{i=j}^{k-1}\frac1{i+\tau}
\right)
\nonumber\\
&\le
e^{2B_\infty}
\left(\frac{j+\tau}{k+\tau}\right)^2,
\label{eq:product_bound_as_2}
\end{align}
where the last inequality uses
\[
\sum_{i=j}^{k-1}\frac1{i+\tau}
\ge
\int_{j+\tau}^{k+\tau}\frac{ds}{s}
=
\log\!\left(\frac{k+\tau}{j+\tau}\right).
\]

Unrolling \eqref{eq:hp_recursion_as} gives
\begin{equation}
F_k
\le
F_0\prod_{i=0}^{k-1}(1-\omega_iI_i)
+
Q\sum_{j=0}^{k-1}
\alpha_j^2
\prod_{i=j+1}^{k-1}(1-\omega_iI_i).
\label{eq:unrolled_as}
\end{equation}
Applying \eqref{eq:product_bound_as_2} first with \(j=0\), and then with
\(j+1\) in place of \(j\), yields
\begin{align}
F_k
&\le
e^{2B_\infty}F_0
\left(\frac{\tau}{k+\tau}\right)^2
+
Qe^{2B_\infty}
\sum_{j=0}^{k-1}
\alpha_j^2
\left(\frac{j+1+\tau}{k+\tau}\right)^2
\nonumber\\
&=
e^{2B_\infty}F_0
\left(\frac{\tau}{k+\tau}\right)^2
+
\frac{Qe^{2B_\infty}}{(k+\tau)^2}
\sum_{j=0}^{k-1}
\frac{4(j+1+\tau)^2}{\bar\beta^2(j+\tau)^2}.
\label{eq:unrolled_as_2}
\end{align}
Because \(\bar\beta\le1\), we have \(\tau=2/\bar\beta\ge2\), and hence
\[
\frac{j+1+\tau}{j+\tau}\le2
\qquad\text{for all }j\ge0.
\]
Therefore
\[
\sum_{j=0}^{k-1}
\frac{4(j+1+\tau)^2}{\bar\beta^2(j+\tau)^2}
\le
\frac{16k}{\bar\beta^2}.
\]
Substituting this into \eqref{eq:unrolled_as_2},
\[
F_k
\le
e^{2B_\infty}F_0
\left(\frac{\tau}{k+\tau}\right)^2
+
\frac{16Qe^{2B_\infty}}{\bar\beta^2}
\frac{k}{(k+\tau)^2}.
\]
Using
\[
\left(\frac{\tau}{k+\tau}\right)^2
\le
\frac{\tau}{k+\tau},
\qquad
\frac{k}{(k+\tau)^2}\le \frac1{k+\tau},
\]
and \(Q=LD^2/2\), we conclude
\[
F_k
\le
\frac{e^{2B_\infty}}{k+\tau}
\left(
\tau F_0+\frac{8LD^2}{\bar\beta^2}
\right).
\]
Thus \eqref{eq:as_ok_rate} holds with
\[
C_\infty
:=
\exp(2B_\infty)
\left(
\tau F_0+\frac{8LD^2}{\bar\beta^2}
\right).
\]
\end{proof}

\begin{proof}[Proof of Corollary~\ref{cor:hp_open_loop_deterministic}]
Let
\[
V_*:=\sum_{m=0}^{\infty}\omega_m^2,
\qquad
c_*:=\beta_{\mathrm{hp}}.
\]
For the martingale \(M_m\) defined in the proof of
Theorem~\ref{thm:hp_open_loop_as}, we have
\[
|M_{m+1}-M_m|\le c_*,
\qquad
\sum_{m=0}^{\infty}
\mathbb E\!\left[(M_{m+1}-M_m)^2\mid\mathcal F_m\right]
\le V_*.
\]
Applying Lemma~\ref{lem:max_freedman} gives
\[
\mathbb P\!\left(
\sup_{m\ge0}|M_m|
\ge
\sqrt{2V_*x}+\frac{c_*x}{3}
\right)
\le
2e^{-x}.
\]
Taking \(x=\log(2/\eta)\), we obtain with probability at least \(1-\eta\),
\[
B_\infty
\le
B_\eta
:=
\sqrt{2V_*\log\frac2\eta}
+
\frac{\beta_{\mathrm{hp}}}{3}
\log\frac2\eta.
\]
Therefore, with probability at least \(1-\eta\),
\[
C_\infty
\le
C_\eta
:=
\exp(2B_\eta)
\left(
\tau F_0+\frac{8LD^2}{\bar\beta^2}
\right).
\]
Combining this deterministic bound on \(C_\infty\) with
Theorem~\ref{thm:hp_open_loop_as} proves the claim.
\end{proof}

\subsection{Finite-sum stochastic extension}
\label{app:finite_sum_stochastic}

We prove Theorem~\ref{thm:main_stoch_rsfw}.

\paragraph{Proof idea.}
Condition on the sampled mini-batch gradient \(\widehat g_k\). The subspace is
still Haar and independent, so the deterministic approximate-oracle inequality
applies to the realized vector \(\widehat g_k\). The only price is the
difference between the true gap for \(g_k\) and the realized gap for
\(\widehat g_k\), which is controlled by the diameter of \(C\) and the
mini-batch variance. Choosing \(b_k\) of order \(k^2\) makes this stochastic
error summable at the same scale as the usual smoothness error.

Let
\[
f(x)=\frac1N\sum_{i=1}^N f_i(x),
\qquad x\in C,
\]
where each \(f_i\) is convex and \(L\)-smooth, so that \(f\) is also
\(L\)-smooth. At iteration \(k\), let
\[
g_k:=\nabla f(x_k),
\qquad
\widehat g_k
=
\frac1{b_k}\sum_{j=1}^{b_k}\nabla f_{i_{k,j}}(x_k)
\]
be a mini-batch estimator sampled independently of the Haar subspace
\(P_k\). We assume
\[
\mathbb E[\widehat g_k\mid x_k]=g_k,
\qquad
\mathbb E[\|\widehat g_k-g_k\|^2\mid x_k]
\le
\frac{\sigma_{\mathrm{fs}}^2}{b_k}.
\]
The stochastic RSFW oracle is
\[
s_k\in
\argmin\{\langle \widehat g_k,s\rangle:
s\in C\cap(x_k+\range(P_k^\top))\},
\qquad
d_k:=s_k-x_k,
\]
and the update is
\[
x_{k+1}=x_k+\alpha_kd_k,
\qquad
\alpha_k=\frac{2}{\beta_0 k+2}.
\]

\begin{proof}[Proof of Theorem~\ref{thm:main_stoch_rsfw}]
Let
\[
h_k:=f(x_k)-f^*,
\qquad
e_k:=\widehat g_k-g_k,
\qquad
D:=\diam(C).
\]
By \(L\)-smoothness,
\begin{equation}
f(x_{k+1})
\le
f(x_k)
+
\alpha_k\langle g_k,d_k\rangle
+
\frac{L}{2}\alpha_k^2\|d_k\|^2.
\label{eq:stoch_smooth_step}
\end{equation}
Since \(x_k,s_k\in C\), we have \(\|d_k\|\le D\). Moreover,
\[
\langle g_k,d_k\rangle
=
\langle \widehat g_k,d_k\rangle
-
\langle e_k,d_k\rangle
\le
\langle \widehat g_k,d_k\rangle
+
D\|e_k\|.
\]

Now condition on \((x_k,\widehat g_k)\). Since \(P_k\) is independent of the
mini-batch, it is still Haar--Stiefel under this conditioning. The
approximate-oracle inequality is geometric and therefore applies to the
realized vector \(\widehat g_k\):
\begin{equation}
\mathbb E_{P_k}
\bigl[-\langle \widehat g_k,d_k\rangle
\mid x_k,\widehat g_k\bigr]
\ge
\beta_0\,\Delta_C(x_k;\widehat g_k),
\label{eq:stoch_realized_approx_oracle}
\end{equation}
where
\[
\Delta_C(x;g):=\max_{s\in C}\langle g,x-s\rangle .
\]
Also, \(\Delta_C(x;\cdot)\) is \(D\)-Lipschitz:
\[
|\Delta_C(x;g)-\Delta_C(x;\tilde g)|
\le
D\|g-\tilde g\|.
\]
Indeed, for all \(s\in C\), \(\|x-s\|\le D\). Hence
\[
\Delta_C(x_k;\widehat g_k)
\ge
\Delta_C(x_k;g_k)-D\|e_k\|.
\]
Combining this with \eqref{eq:stoch_realized_approx_oracle}, and using
\(\beta_0\le1\), gives
\[
\mathbb E_{P_k}
\bigl[\langle g_k,d_k\rangle
\mid x_k,\widehat g_k\bigr]
\le
-\beta_0\Delta_C(x_k;g_k)
+
2D\|e_k\|.
\]
Taking expectation over the mini-batch and using Jensen's inequality,
\[
\mathbb E[\|e_k\|\mid x_k]
\le
\sqrt{\mathbb E[\|e_k\|^2\mid x_k]}
\le
\frac{\sigma_{\mathrm{fs}}}{\sqrt{b_k}}.
\]
Therefore, from \eqref{eq:stoch_smooth_step},
\begin{equation}
\mathbb E[h_{k+1}\mid x_k]
\le
h_k
-
\alpha_k\beta_0\Delta_C(x_k;g_k)
+
2\alpha_kD\frac{\sigma_{\mathrm{fs}}}{\sqrt{b_k}}
+
\frac{L}{2}\alpha_k^2D^2.
\label{eq:stoch_one_step}
\end{equation}
By convexity,
\[
h_k=f(x_k)-f^*\le \Delta_C(x_k;g_k).
\]
Thus
\begin{equation}
\mathbb E[h_{k+1}\mid x_k]
\le
(1-\alpha_k\beta_0)h_k
+
2\alpha_kD\frac{\sigma_{\mathrm{fs}}}{\sqrt{b_k}}
+
\frac{L}{2}\alpha_k^2D^2.
\label{eq:stoch_gap_recursion_conditional}
\end{equation}

Set
\[
a_k:=\mathbb E h_k,
\qquad
t_k:=k+\frac{2}{\beta_0}.
\]
Then
\[
\alpha_k
=
\frac{2}{\beta_0 k+2}
=
\frac{2}{\beta_0 t_k},
\qquad
\alpha_k\beta_0=\frac{2}{t_k}.
\]
Choose
\[
b_k=
\left\lceil \frac{t_k^2}{A_{\mathrm{mb}}}\right\rceil
\]
for some fixed \(A_{\mathrm{mb}}>0\). Then
\[
\frac{1}{\sqrt{b_k}}
\le
\frac{\sqrt{A_{\mathrm{mb}}}}{t_k}.
\]
Taking total expectation in
\eqref{eq:stoch_gap_recursion_conditional}, we obtain
\begin{equation}
a_{k+1}
\le
\left(1-\frac{2}{t_k}\right)a_k
+
\frac{B_{\mathrm{fs}}}{t_k^2},
\label{eq:stoch_scalar_recursion}
\end{equation}
where
\begin{equation}
B_{\mathrm{fs}}
:=
\frac{2LD^2}{\beta_0^2}
+
\frac{4D\sigma_{\mathrm{fs}}\sqrt{A_{\mathrm{mb}}}}{\beta_0}.
\label{eq:Bfs_def}
\end{equation}

Let
\[
M_{\mathrm{fs}}
:=
\max\left\{
\frac{2a_0}{\beta_0},
B_{\mathrm{fs}}
\right\}.
\]
We prove by induction that
\[
a_k\le \frac{M_{\mathrm{fs}}}{t_k}
=
\frac{M_{\mathrm{fs}}}{k+2/\beta_0}.
\]
The base case follows from \(M_{\mathrm{fs}}\ge 2a_0/\beta_0\). Assume the
claim holds at step \(k\). Using
\eqref{eq:stoch_scalar_recursion} and \(M_{\mathrm{fs}}\ge B_{\mathrm{fs}}\),
\[
a_{k+1}
\le
\left(1-\frac{2}{t_k}\right)\frac{M_{\mathrm{fs}}}{t_k}
+
\frac{M_{\mathrm{fs}}}{t_k^2}
=
\frac{M_{\mathrm{fs}}(t_k-1)}{t_k^2}.
\]
Since
\[
\frac{t_k-1}{t_k^2}
\le
\frac{1}{t_k+1},
\]
we obtain
\[
a_{k+1}
\le
\frac{M_{\mathrm{fs}}}{t_k+1}
=
\frac{M_{\mathrm{fs}}}{k+1+2/\beta_0}.
\]
This closes the induction and proves
\[
\mathbb E[f(x_k)-f^*]
\le
\frac{M_{\mathrm{fs}}}{k+2/\beta_0}.
\]
\end{proof}

\section{Random Haar spectral bounds and quadratic ellipsoidal short steps}
\label{app:random_haar_quadratic}

This appendix proves the Haar spectral estimate used in
Proposition~\ref{prop:random_ellipsoid_short_step_improvement}, and then proves
the ellipsoidal short-step bound. The proof uses a direct Haar argument rather
than Gaussian normalization. The key input is a spherical Hanson--Wright
inequality, obtained from the dependent Hanson--Wright inequality for random
vectors with the convex concentration property
\cite{Adamczak2015DependentHW}, together with standard concentration on the
sphere \cite{Ledoux2001Concentration}.

\subsection{Direct Haar compression}

For a symmetric matrix \(H\), write
\[
\bar\lambda_H:=\frac{\operatorname{tr}(H)}{n},
\qquad
H_c:=H-\bar\lambda_H I_n.
\]
The matrix \(H_c\) is the centered anisotropic part of \(H\). Since
\(PP^\top=I_d\),
\[
PHP^\top-\bar\lambda_H I_d=PH_cP^\top.
\]
The next lemma is a spherical quadratic-form concentration bound. We derive it
from the Hanson--Wright inequality for dependent random vectors satisfying the
convex concentration property \cite{Adamczak2015DependentHW}, together with
standard concentration of Lipschitz functions on the Euclidean sphere
\cite{Ledoux2001Concentration}.
\begin{lemma}[Spherical Hanson--Wright consequence]
\label{lem:spherical_hanson_wright}
Let \(\theta\sim\mathrm{Unif}(S^{n-1})\), and let
\(B=B^\top\in\mathbb R^{n\times n}\). There exist universal constants
\(c,C>0\) such that, for every \(u\ge0\),
\begin{equation}
\mathbb P\!\left(
\left|
\theta^\top B\theta-\frac{\operatorname{tr}(B)}{n}
\right|
\ge
C\left(
\frac{\|B\|_F}{n}\sqrt u
+
\frac{\|B\|_2}{n}u
\right)
\right)
\le
2e^{-cu}.
\label{eq:spherical_hanson_wright}
\end{equation}
In particular, if \(\operatorname{tr}(B)=0\), then
\begin{equation}
\mathbb P\!\left(
|\theta^\top B\theta|
\ge
C\left(
\frac{\|B\|_F}{n}\sqrt u
+
\frac{\|B\|_2}{n}u
\right)
\right)
\le
2e^{-cu}.
\label{eq:spherical_hanson_wright_centered}
\end{equation}
\end{lemma}

\begin{proof}
Set \(X:=\sqrt n\,\theta\). Then
\[
\mathbb E[XX^\top]=I_n,
\]
so \(X\) is isotropic. Concentration of Lipschitz functions on the Euclidean
sphere implies that \(X\) satisfies the convex concentration property with a
universal constant \cite{Ledoux2001Concentration}. Therefore, Adamczak's
dependent Hanson--Wright inequality applies \cite{Adamczak2015DependentHW},
and gives
\[
\mathbb P\!\left(
|X^\top B X-\mathbb E[X^\top B X]|
\ge t
\right)
\le
2\exp\!\left(
-c\min\left\{
\frac{t^2}{\|B\|_F^2},
\frac{t}{\|B\|_2}
\right\}
\right).
\]
Since
\[
X^\top B X=n\,\theta^\top B\theta,
\qquad
\mathbb E[X^\top B X]=\operatorname{tr}(B),
\]
taking
\[
t=C\left(\|B\|_F\sqrt u+\|B\|_2u\right)
\]
with \(C\) large enough and dividing by \(n\) gives
\eqref{eq:spherical_hanson_wright}. The centered case follows when
\(\operatorname{tr}(B)=0\).
\end{proof}

For \(\eta\in(0,1)\), define the Haar-compression error
\begin{equation}
\operatorname{err}_H(d,\eta)
:=
C\left(
\frac{\|H_c\|_F}{n}\sqrt{d+\log\frac{2}{\eta}}
+
\frac{\|H_c\|_2}{n}\left(d+\log\frac{2}{\eta}\right)
\right),
\label{eq:haar_compression_error_app}
\end{equation}
where \(C>0\) is a universal numerical constant.

\begin{proposition}[Direct Haar spectral sandwich]
\label{prop:haar_spectral_sandwich}
Let \(H\succeq0\), let \(P\in\mathbb R^{d\times n}\) be Haar--Stiefel with
\(PP^\top=I_d\), and let \(\eta\in(0,1)\). Then, with probability at least
\(1-\eta\),
\begin{equation}
\bigl\|PHP^\top-\bar\lambda_H I_d\bigr\|
\le
\operatorname{err}_H(d,\eta).
\label{eq:haar_spectral_sandwich_norm}
\end{equation}
Consequently,
\begin{equation}
\bigl[\bar\lambda_H-\operatorname{err}_H(d,\eta)\bigr]_+
\le
\lambda_{\min}(PHP^\top)
\le
\lambda_{\max}(PHP^\top)
\le
\bar\lambda_H+\operatorname{err}_H(d,\eta).
\label{eq:haar_spectral_sandwich_eigs}
\end{equation}
Equivalently, if \(\bar\lambda_H>0\) and
\(r_H:=\operatorname{err}_H(d,\eta)/\bar\lambda_H<1\), then all eigenvalues of
\(PHP^\top\) lie in
\begin{equation}
\bar\lambda_H[1-r_H,1+r_H].
\label{eq:haar_spectral_sandwich_relative}
\end{equation}
\end{proposition}
\paragraph{Proof idea.}
For each fixed unit vector \(z\in S^{d-1}\), the vector \(P^\top z\) is uniform
on \(S^{n-1}\). A spherical Hanson--Wright bound therefore controls the
quadratic form \(z^\top PH_cP^\top z\). A \(1/4\)-net and a union bound upgrade
this pointwise control to an operator-norm bound on \(PH_cP^\top\), giving the
spectral sandwich around the average eigenvalue \(\operatorname{tr}(H)/n\).
\begin{proof}
Let
\[
B_P:=PH_cP^\top.
\]
For every fixed \(z\in S^{d-1}\), the vector \(P^\top z\) is uniformly
distributed on \(S^{n-1}\). Since \(\operatorname{tr}(H_c)=0\), Lemma
\ref{lem:spherical_hanson_wright} gives, for every \(u\ge0\),
\[
\mathbb P\!\left(
|z^\top B_Pz|
\ge
C\left(
\frac{\|H_c\|_F}{n}\sqrt u
+
\frac{\|H_c\|_2}{n}u
\right)
\right)
\le
2e^{-cu}.
\]

Let \(\mathcal N\) be a \(1/4\)-net of \(S^{d-1}\) with
\[
|\mathcal N|\le 9^d.
\]
Choose
\[
u=C_0\left(d+\log\frac{2}{\eta}\right)
\]
with \(C_0\) large enough. A union bound gives, with probability at least
\(1-\eta\),
\[
\max_{z\in\mathcal N}|z^\top B_Pz|
\le
C\left(
\frac{\|H_c\|_F}{n}\sqrt{d+\log\frac{2}{\eta}}
+
\frac{\|H_c\|_2}{n}\left(d+\log\frac{2}{\eta}\right)
\right).
\]
For a symmetric matrix \(B_P\), the standard net inequality gives
\[
\|B_P\|\le2\max_{z\in\mathcal N}|z^\top B_Pz|.
\]
Absorbing the factor \(2\) into the universal constant \(C\), and using
\[
B_P=PHP^\top-\bar\lambda_HI_d,
\]
proves \eqref{eq:haar_spectral_sandwich_norm}. The eigenvalue bounds follow
from Weyl's inequality and \(H\succeq0\). If
\(r_H=\operatorname{err}_H(d,\eta)/\bar\lambda_H<1\), the relative form follows
immediately.
\end{proof}

\subsection{Random ellipsoidal section and short-step descent}
\paragraph{Proof idea.}
The proof has two independent parts. First, apply the Haar spectral sandwich to
\(M\) and \(Q\), which replaces the ambient spectral extremes by compressed
quantities close to \(\operatorname{tr}(M)/n\) and \(\operatorname{tr}(Q)/n\).
Second, for a fixed section, rewrite the ellipsoidal section as a Euclidean
ball by completing the square. In those normalized coordinates, the exact LMO is
explicit, and a ball-geometry estimate gives
\(\mathrm{gap}_k/\|d_k\|^2\ge
\beta_{\mathrm{sec},k}\|P_kg_k\|\). Combining this with the exact quadratic
short-step decrease proves the result.
\begin{proof}[Proof of Proposition~\ref{prop:random_ellipsoid_short_step_improvement}]
Apply Proposition~\ref{prop:haar_spectral_sandwich} to \(H=M\) with failure
probability \(\eta/2\), and to \(H=Q\) with failure probability \(\eta/2\). By
the union bound, with probability at least \(1-\eta\), both spectral sandwiches
hold. Since \(M\succ0\), \(\bar\mu>0\), and \(r_M<1\), we get
\[
\bar\mu(1-r_M)
\le
\lambda_{\min}(P_kMP_k^\top)
\le
\lambda_{\max}(P_kMP_k^\top)
\le
\bar\mu(1+r_M).
\]
Since \(Q\succeq0\) and \(\operatorname{tr}(Q)>0\), \(\bar\lambda>0\), and the
same spectral sandwich gives
\[
\bigl[\bar\lambda-\operatorname{err}_Q(d,\eta/2)\bigr]_+
\le
\lambda_{\min}(P_kQP_k^\top)
\le
\lambda_{\max}(P_kQP_k^\top)
\le
\bar\lambda(1+r_Q).
\]

It remains to prove the short-step estimate on this event. The argument is
deterministic once \(P_k\) is fixed. Write
\[
P=P_k,\qquad x=x_k,\qquad g=g_k,\qquad A=A_k,\qquad b=b_k,\qquad
\delta=\delta_k.
\]
Since \(M\succ0\) and \(P^\top:\mathbb R^d\to\mathbb R^n\) is injective,
\[
A=PMP^\top\succ0.
\]
Every point in the affine section \(x+\operatorname{range}(P^\top)\) has the
unique form \(x+P^\top z\), \(z\in\mathbb R^d\). The ellipsoidal constraint is
\[
(x+P^\top z)^\top M(x+P^\top z)
=
z^\top A z+2b^\top z+x^\top Mx
\le1.
\]
Completing the square gives
\begin{equation}
(z+A^{-1}b)^\top A(z+A^{-1}b)
\le
\delta,
\qquad
\delta:=1-x^\top Mx+b^\top A^{-1}b.
\label{eq:ellipsoid_completed_square_corrected}
\end{equation}
Because \(x\in C\), \(z=0\) is feasible, and therefore
\[
b^\top A^{-1}b\le \delta.
\]
If \(\delta=0\), the section is a singleton. We now assume \(\delta>0\). The
change of variables
\[
y:=\frac1{\sqrt\delta}A^{1/2}(z+A^{-1}b)
\]
identifies the section with the unit ball \(\|y\|\le1\). The current point
\(x\), corresponding to \(z=0\), is mapped to
\[
y_x:=\frac1{\sqrt\delta}A^{-1/2}b,
\qquad
\|y_x\|\le1.
\]
The linear objective over the section is, up to an additive constant,
\[
\sqrt\delta\,\langle A^{-1/2}Pg,y\rangle.
\]
Assume \(Pg\neq0\), and set
\[
w:=A^{-1/2}Pg.
\]
The minimizing point in the normalized ball is
\[
y^*:=-\frac{w}{\|w\|}.
\]
Returning to \(z\)-coordinates,
\[
z^*
=
-A^{-1}b
-
\sqrt\delta\,
\frac{A^{-1}Pg}{\sqrt{(Pg)^\top A^{-1}(Pg)}}.
\]
Thus \(s=x+P^\top z^*\). With \(d_k=s-x=P^\top z^*\) and
\(\mathrm{gap}:=\langle g,x-s\rangle\), we have
\[
d_k
=
\sqrt\delta\,P^\top A^{-1/2}(y^*-y_x),
\]
and
\[
\mathrm{gap}
=
\sqrt\delta\,\langle w,y_x-y^*\rangle
=
\sqrt\delta\left(\langle w,y_x\rangle+\|w\|\right).
\]
Since \(\|y_x\|\le1\),
\begin{align*}
\|y_x-y^*\|^2
&=
\left\|y_x+\frac{w}{\|w\|}\right\|^2 \\
&=
\|y_x\|^2+1+2\left\langle y_x,\frac{w}{\|w\|}\right\rangle \\
&\le
2\left(1+\left\langle y_x,\frac{w}{\|w\|}\right\rangle\right) \\
&=
\frac{2}{\|w\|}
\left(\langle w,y_x\rangle+\|w\|\right).
\end{align*}
Therefore
\[
\|y_x-y^*\|^2
\le
\frac{2\,\mathrm{gap}}{\sqrt\delta\,\|w\|}.
\]
Because \(P^\top\) is an isometry from \(\mathbb R^d\) to
\(\operatorname{range}(P^\top)\),
\[
\|d_k\|^2
=
\delta\|A^{-1/2}(y^*-y_x)\|^2
\le
\frac{\delta}{\lambda_{\min}(A)}\|y^*-y_x\|^2.
\]
Combining the last two displays gives
\[
\|d_k\|^2
\le
\frac{2\sqrt\delta}{\lambda_{\min}(A)\|w\|}\,\mathrm{gap}.
\]
Thus
\[
\frac{\mathrm{gap}}{\|d_k\|^2}
\ge
\frac{\lambda_{\min}(A)}{2\sqrt\delta}\|w\|.
\]
Using
\[
\|w\|=\|A^{-1/2}Pg\|
\ge
\frac{\|Pg\|}{\sqrt{\lambda_{\max}(A)}},
\]
we obtain
\begin{equation}
\frac{\mathrm{gap}}{\|d_k\|^2}
\ge
\frac{\lambda_{\min}(A)}
{2\sqrt{\delta\,\lambda_{\max}(A)}}\|Pg\|.
\label{eq:ellipsoid_gap_over_step_corrected}
\end{equation}
Define
\begin{equation}
\beta_{\mathrm{sec}}(P,x)
:=
\frac{\lambda_{\min}(A)}
{2\sqrt{\delta\,\lambda_{\max}(A)}}.
\label{eq:beta_sec_app}
\end{equation}
Then \eqref{eq:ellipsoid_gap_over_step_corrected} becomes
\[
\frac{\mathrm{gap}}{\|d_k\|^2}
\ge
\beta_{\mathrm{sec}}(P,x)\|Pg\|.
\]

Now use the quadratic model. Since \(f\) has Hessian \(Q\), for every
\(\alpha\in[0,1]\),
\[
f(x+\alpha d_k)
=
f(x)+\alpha\langle g,d_k\rangle
+\frac{\alpha^2}{2}d_k^\top Qd_k.
\]
Writing \(d_k=P^\top z^*\), and using \(\|P^\top z^*\|=\|z^*\|\),
\[
d_k^\top Qd_k
=
(z^*)^\top P QP^\top z^*
\le
L_P\|d_k\|^2,
\qquad
L_P:=\lambda_{\max}(PQP^\top).
\]
With the quadratic short-step choice
\[
\alpha_k^{\rm quad}
=
\min\left\{
1,
\frac{\mathrm{gap}}{L_P\|d_k\|^2}
\right\},
\]
defining $h_k:=f(x_k)-f^*$, one has
\[
h_{k+1}
\le
h_k
-
\min\left(
\frac{\mathrm{gap}}{2},
\frac{\mathrm{gap}^2}{2L_P\|d_k\|^2}
\right).
\]
On the genuine short-step branch \(\alpha_k^{\rm quad}\in(0,1)\),
\[
h_{k+1}
\le
h_k-\frac{\mathrm{gap}^2}{2L_P\|d_k\|^2}.
\]
Combining this with \eqref{eq:ellipsoid_gap_over_step_corrected} gives
\begin{equation}
h_{k+1}
\le
h_k
-
\frac{\beta_{\mathrm{sec}}(P,x)}{2L_P}\|Pg\|\,\mathrm{gap}.
\label{eq:ellipsoid_beta_sec_descent_app}
\end{equation}

Finally, the spectral bounds imply
\[
\beta_{\mathrm{sec}}(P,x)
\ge
\frac{\sqrt{\bar\mu}}{2\sqrt{\delta}}
\frac{1-r_M}{\sqrt{1+r_M}},
\qquad
L_P\le \bar\lambda(1+r_Q).
\]
Therefore
\[
\frac{\beta_{\mathrm{sec}}(P,x)}{2L_P}
\ge
\frac{\sqrt{\bar\mu}}{4\bar\lambda}
\frac{1-r_M}{(1+r_Q)\sqrt{\delta(1+r_M)}}.
\]
Returning to \(P=P_k\), \(x=x_k\), and \(\delta=\delta_k\), this proves
\eqref{eq:random_ellipsoid_short_step_descent_beta_sec} and
\eqref{eq:beta_sec_and_LP_random_bounds}.
\end{proof}

\subsection{Why curvature and boundary regularity matter: a corner failure mode}
\label{app:polytope_failure}

The geometric assumptions in the main results are not only proof conveniences.
They rule out a basic failure mode of random-section Frank--Wolfe near sharp
faces or corners. The issue is angular rather than purely dimensional: the
directions that yield meaningful Frank--Wolfe progress may be confined to a
cone of very small solid angle. Even when this cone is full-dimensional, a
random low-dimensional affine section can intersect it in a useful way with
probability that is exponentially small in the ambient dimension.

Consider the simplex-like polytope
\[
C=\{x\in\mathbb R^n:\ x_i\ge 0,\ \sum_{i=1}^n x_i\le 1\}
\]
and the quadratic objective
\[
f(x)=\frac12\|x-a\|^2,
\qquad
a=5e_1 .
\]
The constrained minimizer is the vertex \(x^*=e_1\). If RSFW is initialized at
a strictly interior point, for example
\[
x_0=\frac{\delta}{n}{\bf 1},
\qquad
0<\delta\ll1,
\]
then the full FW direction points toward the vertex \(e_1\). A random
\(d\)-dimensional affine section \(x_k+U_k\), with \(d\ll n\), almost surely
does not contain that vertex. More importantly, near the relevant corner the
section must capture a narrow feasible descent cone in order to recover the
remaining FW progress. This cone is full-dimensional inside the ambient space,
but its solid angle can be small. Thus the obstruction is not only
that the section misses a particular atom; it is that the useful descent
geometry occupies too small an angular region.

This phenomenon is not limited to perfectly flat polytopes. If the facets are
slightly curved while the simplex-like vertices are kept sharp, one obtains a
strongly convex but nonsmooth feasible body. At such a corner there is no
positive-radius inner tangent ball: any Euclidean ball touching the set at the
corner necessarily exits the feasible region. Hence the rolling-ball comparison
used in Proposition~\ref{prop:BallGap_global_clean} fails. The feasible descent
region can again collapse into a narrow cone, and a random low-dimensional
section may capture this cone only with exponentially small probability in
\(n\). This is why the main theory requires both curvature of the feasible set
and boundary regularity.

Figure~\ref{fig:polytope_failure} illustrates the failure mode for the polytope
with \(n=100\) and \(d=10\). The method makes many nontrivial section moves, but
the objective gap stagnates away from zero. This is not a contradiction of the
RSFW theory: the simplex violates the curvature and rolling-ball hypotheses,
and the section-efficiency constants used in the main analysis need not be
bounded away from zero.

For smooth strongly convex bodies satisfying a uniform rolling-ball condition,
the situation is different. The boundary has quantitative curvature without
sharp corners. Random affine sections through the current iterate then contain
feasible chords whose size can be controlled in terms of the projected gradient
direction. This prevents the useful descent region from collapsing into an
arbitrarily narrow cone and is the geometric mechanism behind the
approximate-oracle inequalities used in the main text.

\begin{figure}[t]
\centering
\includegraphics[width=0.9\linewidth]{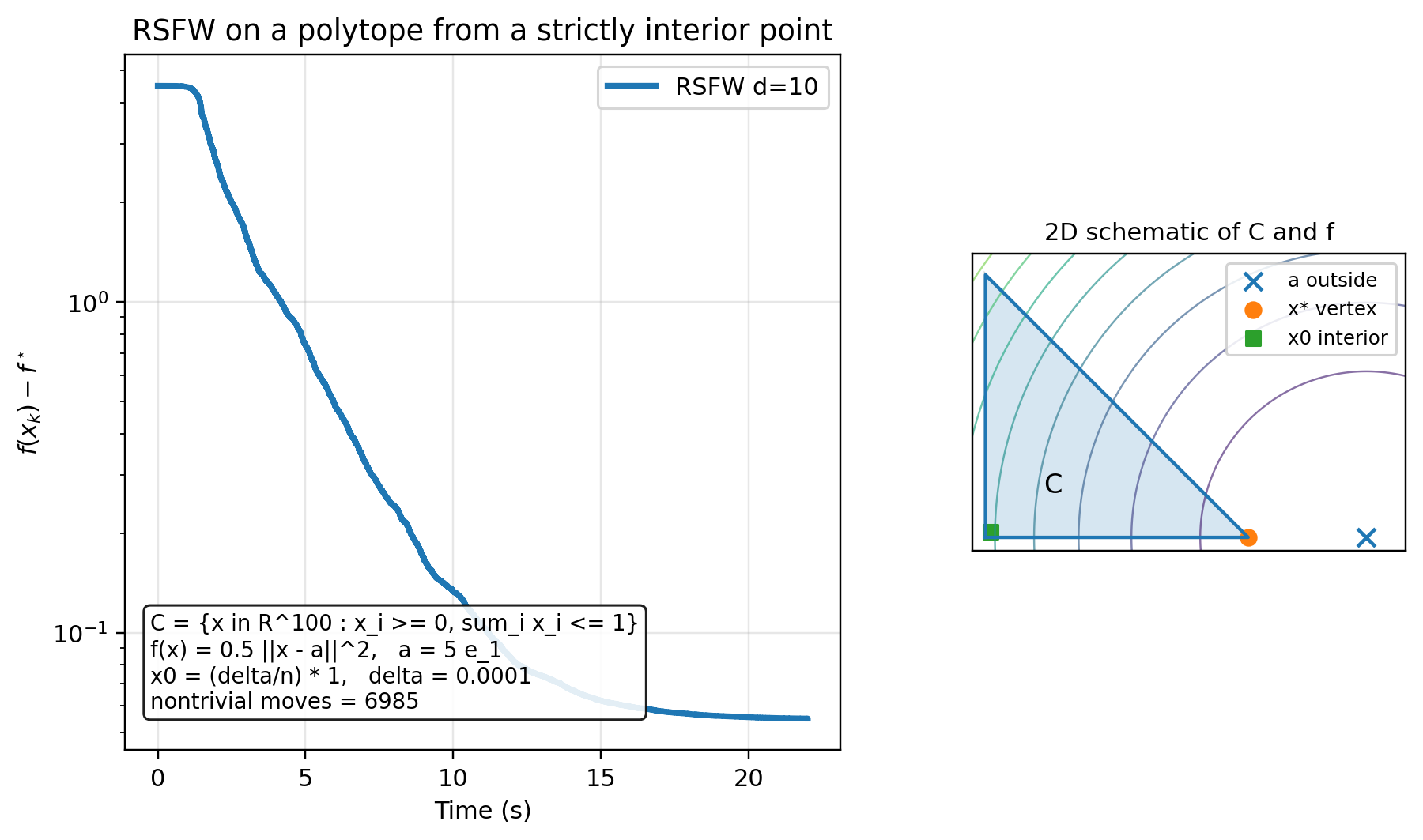}
\caption{
Failure mode of vanilla RSFW outside the smooth strongly convex setting. Left:
objective gap for RSFW with \(d=10\) on
\(C=\{x\in\mathbb R^{100}:x_i\ge0,\ \sum_i x_i\le1\}\) and
\(f(x)=\frac12\|x-5e_1\|^2\), initialized from the strictly interior point
\(x_0=(\delta/n){\bf 1}\). The method makes many nontrivial section moves but
stagnates away from the optimum. Right: schematic explanation in two
dimensions. Near a sharp corner, useful descent directions can occupy a very
small angular region, so a random low-dimensional section may see only an
unfavorable local slice of the feasible region.
}
\label{fig:polytope_failure}
\end{figure}

\section{Numerical details}
\label{app:numerics_details}
This appendix gives the implementation details for the four numerical
experiments in Section~\ref{sec:numerics}. All HIGGS experiments use binary
kernel-logistic regression. Given training data
\[
    (x_i,y_i)_{i=1}^n,
    \qquad
    y_i\in\{-1,+1\},
\]
we build a random-feature matrix
\(\Phi\in\mathbb R^{n\times m}\), whose \(i\)-th row is the feature vector
\(\phi(x_i)^\top\). The corresponding matrix-free kernel approximation is
\[
    K=\Phi\Phi^\top+\rho_K I_n,
    \qquad
    Kv=\Phi(\Phi^\top v)+\rho_Kv .
\]
The optimization variable is the coefficient vector \(a\in\mathbb R^n\), and
\((Ka)_i\) is the prediction score on training point \(x_i\). The objective is
\[
    f(a)=\frac1n\sum_{i=1}^n
    \log\!\left(1+\exp(-y_i(Ka)_i)\right)
    +\frac{\lambda}{2}a^\top Ka .
\]
The random-feature map \(\Phi\) is sampled once and then frozen before
optimization, so all methods optimize the same deterministic objective. The
feasible region is an ellipsoid
\[
    C=\{a\in\mathbb R^n:a^\top Ha\le R^2\}.
\]
For a full FW step, the ellipsoid LMO is
\[
    s(g)=
    -\frac{R\,H^{-1}g}{\sqrt{g^\top H^{-1}g}} .
\]
For RSFW, we draw a sketch basis \(S\in\mathbb R^{n\times d}\) and solve over
\(s=a+Sz\). The ellipsoid constraint becomes
\[
    z^\top A z+2b^\top z+a^\top Ha\le R^2,
    \qquad
    A=S^\top H S,\quad b=S^\top Ha .
\]
Once these reduced quantities are available, the section LMO is solved exactly
by dense \(d\times d\) linear algebra.

\paragraph{Fixed-\(L\) curvature grid.}
This experiment tests whether RSFW can use smaller stable short-step constants
than full FW when all methods choose from the same prescribed grid. The
ellipsoid is \(H=D\), where \(D\) is diagonal and has condition number \(10^4\).
The reported run uses HIGGS with
\[
    n_{\rm train}=20000,\qquad n_{\rm test}=10000,\qquad m=10000,
\]
\[
    \rho_K=10^{-3},\qquad \lambda=2\cdot10^{-2},\qquad R=1,
\]
and horizon \(100\). We test \(d\in\{30,50,100\}\) and the common grid
\[
    \mathcal L=\{10^{-3},10^{-2},10^{-1},1,10,10^2,10^3\}.
\]
For each method, we keep the smallest value in \(\mathcal L\) that does not
increase the objective over the run. In the reported run this selected
\[
    L_{\rm full}=10^2,\qquad L_{\rm RSFW}=1
\]
for the tested RSFW section dimensions. The plot reports \(f(a_k)\) against
wall-clock time for the selected values. RSFW is repeated over five fixed seeds,
and the shaded region shows one empirical standard deviation.

\paragraph{Section-curvature diagnostic.}
This experiment compares the ambient trace-based curvature bound with
representative section-wise curvature estimates. The run uses HIGGS with
\[
    n_{\rm train}=20000,\qquad n_{\rm test}=10000,\qquad m=1024,
\]
\[
    \rho_K=10^{-3},\qquad \lambda=2\cdot10^{-2},\qquad R=1,
\]
ellipsoid condition number \(10^6\), and horizon \(100\). Full FW uses the
trace-based ambient bound. RSFW uses one representative section curvature
estimate at initialization, with safety factor \(1.5\),
\[
    L_S
    =
    1.5\,
    \lambda_{\max}\!\left(
    \frac{(KS)^\top(KS)}{4n}+\lambda S^\top K S
    \right),
\]
and reuses this value during the run. The tested section dimensions are
\(d\in\{30,50,100\}\). Objective values are plotted against method time. Full FW
gap values, when computed, are evaluated offline from saved snapshots and are
not included in the RSFW timed loop.

\paragraph{Matrix-free ellipsoid oracle.}
This experiment uses the same HIGGS kernel-logistic objective but sets
\[
    H=K+\rho_H I,
    \qquad
    K=\Phi\Phi^\top+\rho_K I .
\]
The run uses HIGGS with
\[
    n_{\rm train}=50000,\qquad n_{\rm test}=10000,\qquad m=8192,
\]
\[
    \rho_K=10^{-8},\qquad \rho_H=10^{-10},\qquad
    \lambda=2\cdot10^{-2},\qquad R=1,
\]
short-step constant \(L=200\), and horizon \(100\). The current main-text plot
gives full FW a feature-space/direct full ellipsoid oracle as a stronger
baseline, rather than relying on a failed ambient CG solve. RSFW uses
\(d\in\{100,200,400\}\) and is repeated over five fixed seeds. The objective
plot is cropped to the first \(60\) seconds to emphasize the informative
transient regime.

In the feature-sketch implementation, RSFW forms
\[
    A=S^\top H S
     =
    (\Phi^\top S)^\top(\Phi^\top S)
    +(\rho_K+\rho_H)S^\top S,
    \qquad
    b=S^\top Ha,
\]
together with \(a^\top Ha\), and never stores \(HS\) or \(KS\) as
\(n\times d\) matrices. Once these reduced quantities are available, the section
LMO is solved exactly by dense \(d\times d\) linear algebra.

\paragraph{Graph-based semi-supervised experiment.}
This experiment tests RSFW on a vector-valued non-ellipsoidal feasible region.
We build a \(k\)-nearest-neighbor graph on a Covertype subset, let \(L_G\) be the
normalized graph Laplacian, and select a labeled set \(\mathcal I_{\rm lab}\). The
optimization variable is \(u\in\mathbb R^n\), interpreted as a score on each
node. The objective is the labeled squared loss
\[
    f(u)=\frac{1}{2|\mathcal I_{\rm lab}|}
    \sum_{i\in\mathcal I_{\rm lab}}(u_i-y_i)^2,
\]
and the constraint is
\[
    C_G=
    \left\{
    u:
    \frac{\mu}{2}\|u\|_2^2
    +\frac{\gamma_G}{2}u^\top L_Gu
    +\frac{\beta_4}{4}\sum_{i=1}^n u_i^4
    \le \tau
    \right\}.
\]
For \(\mu,\beta_4>0\), the defining function is smooth with uniformly positive
definite Hessian, and its gradient does not vanish on the boundary
\(\{h=\tau\}\). Thus \(C_G\) is a smooth strongly convex body; it is
non-ellipsoidal when \(\beta_4>0\). The graph term is the standard Dirichlet
energy encouraging neighboring nodes to have similar scores, while the quartic
term makes the constraint genuinely nonquadratic.

The full LMO at gradient \(g=\nabla f(u)\) solves
\[
    \min_{s\in C_G}\langle g,s\rangle .
\]
Since the constraint is active unless \(g=0\), the Karush--Kuhn--Tucker
optimality conditions introduce a multiplier \(\zeta\ge0\) and give
\[
    g+\zeta\big((\mu I+\gamma_G L_G)s+\beta_4s^{\circ 3}\big)=0,
    \qquad h(s)=\tau .
\]
We refer to this nonlinear system as the graph LMO optimality system. Hence, each full oracle call requires a nonlinear graph solve. In the
implementation, we use bisection on the multiplier \(\zeta\) and damped
Newton-CG for the inner nonlinear equation.

For RSFW, we draw a scaled dense Gaussian sketch
\(S\in\mathbb R^{n\times d}\) without QR and solve over \(s=u+Sz\). The reduced
constraint uses
\[
    S^\top(\mu I+\gamma_G L_G)S,
    \qquad
    S^\top(\mu I+\gamma_G L_G)u,
\]
so the graph part is compressed once per iteration. The remaining quartic term
is evaluated through \(u+Sz\), and the small constrained problem in \(z\) is
solved by dense nonlinear optimization. No full nonlinear graph solve is used in
the RSFW oracle.

The short-step rule uses the exact directional curvature of the objective.
Since
\[
    \nabla^2 f=|\mathcal I_{\rm lab}|^{-1}P_{\mathcal I_{\rm lab}},
\]
where \(P_{\mathcal I_{\rm lab}}\) is the diagonal projector onto labeled coordinates, for
any direction \(D=s-u\),
\[
    f(u+\eta D)
    =
    f(u)+\eta\langle\nabla f(u),D\rangle
    +\frac{\eta^2}{2|\mathcal I_{\rm lab}|}\|D_{\mathcal I_{\rm lab}}\|^2.
\]
The implemented closed-form short step is therefore
\[
    \eta
    =
    \min\left\{
    \frac{\langle\nabla f(u),u-s\rangle}
    {|\mathcal I_{\rm lab}|^{-1}\|D_{\mathcal I_{\rm lab}}\|^2},
    1
    \right\},
\]
with the usual convention when \(D_{\mathcal I_{\rm lab}}=0\). No backtracking or line
search is used.

In the current main-text graph run we use Covertype with
\[
    n=50000,\qquad k=20,\qquad |\mathcal I_{\rm lab}|=1001,
\]
and
\[
    \tau=1000,\qquad
    \mu=10^{-6},\qquad
    \gamma_G=0.05,\qquad
    \beta_4=0.01 .
\]
Full FW is run once, while RSFW is repeated over five fixed independent seeds;
the main plot reports the mean objective curve with one-standard-deviation
bands. The time-based plot is cropped to the first \(150\) seconds.

\paragraph{Compute environment and code.}
The reported experiments were run on a laptop with an Intel Core Ultra 7 155H
processor, \(64\)GB RAM, and an NVIDIA GeForce RTX 4070 Laptop GPU with \(8\)GB
memory. The experiments are CPU-based unless otherwise specified. The code will
be made available with the final version of the paper. The HIGGS and Covertype datasets are obtained from the UCI Machine Learning
Repository and are licensed under CC BY 4.0
\cite{Whiteson2014HIGGSUCI,Blackard1998CovertypeUCI}.

\begin{table}[t]
\centering
\small
\caption{Main numerical settings. All kernel-logistic experiments use HIGGS;
the graph SSL experiment uses Covertype.}
\label{tab:numerics_settings}
\begin{tabular}{@{}lccccc@{}}
\toprule
Experiment & dataset & problem size & \(m\) & geometry & \(d\) \\
\midrule
Fixed-\(L\) grid
& HIGGS & \(n_{\rm train}=20000\) & \(10000\)
& \(H=D,\ \kappa(H)=10^4\)
& \(30,50,100\) \\
Section-curvature diagnostic
& HIGGS & \(n_{\rm train}=20000\) & \(1024\)
& \(\kappa(H)=10^6\)
& \(30,50,100\) \\
Matrix-free ellipsoid oracle
& HIGGS & \(n_{\rm train}=50000\) & \(8192\)
& \(H=K+\rho_H I\)
& \(100,200,400\) \\
Graph SSL
& Covertype & \(n=50000\) & --
& \(C_G\)
& see legend \\
\bottomrule
\end{tabular}
\end{table}

\end{document}